\documentclass[12pt]{article}

\usepackage{amssymb,amsmath,amsfonts,amsthm}
\usepackage{graphicx}
\usepackage{epstopdf}
\usepackage{caption}
\usepackage[dvipsnames,usenames]{color}
\usepackage{float}
\usepackage{subfigure}
\usepackage[toc,page]{appendix}

\newtheorem{theorem}{Theorem}
\newtheorem{corollary}{Corollary}
\newtheorem{proposition}{Proposition}

\newtheorem{lemma}{Lemma}
\newtheorem{remark}{Remark}

\usepackage{hyperref}
\usepackage{authblk}

\begin{document}

\numberwithin{equation}{section}
\numberwithin{theorem}{section}

\title{THE GENERALIZED JOIN THE SHORTEST ORBIT QUEUE SYSTEM: STABILITY, EXACT TAIL ASYMPTOTICS  AND STATIONARY APPROXIMATIONS}
\author{IOANNIS DIMITRIOU}
\affil{Department of Mathematics, University of Patras, 26504 Patras, Greece}
\affil[ ]{\textit {E-mail: idimit@math.upatras.gr}}
\renewcommand\Authands{ and }
\providecommand{\keywords}[1]{\textbf{\textit{Keywords---}} #1}
\maketitle

\begin{abstract}
We introduce the \textit{generalized join the shortest queue model with retrials} and two infinite capacity orbit queues. Three independent Poisson streams of jobs, namely a \textit{smart}, and two \textit{dedicated} streams, flow into a single server
system, which can hold at most one job. Arriving jobs that find the server occupied are routed to the orbits as follows: Blocked jobs from the \textit{smart} stream are routed to the shortest orbit queue, and in case of a tie, they choose an orbit randomly. Blocked jobs from the \textit{dedicated} streams are routed directly to their orbits. Orbiting jobs retry to connect with the server at different retrial rates, i.e., heterogeneous orbit queues. Applications of such a system are found in the modelling of wireless cooperative networks. We are interested in the asymptotic behaviour of the stationary distribution of this model, provided that the system is stable. More precisely, we investigate the conditions under which the tail asymptotic of the minimum orbit queue length is exactly geometric. Moreover, we apply a heuristic asymptotic approach to obtain approximations of the steady-state joint orbit queue-length distribution. Useful numerical examples are presented, and shown that the results obtained through the asymptotic analysis and the heuristic approach agreed.
\end{abstract}


\section{Introduction}
In this work, we focus on the asymptotic stationary behaviour of the \textit{generalized join the shortest orbit queue} (GJSOQ) policy with retrials. This model is a natural generalization of the join the shortest orbit queue system, recently introduced in \cite{dimisq}, by considering both non-identical retrial rates, and additional dedicated arrival streams that route jobs directly to the orbit queues if an arriving job finds the server busy; for more information about developments on the analysis of retrial queues, see the seminal books in \cite{falin,artalejo}. 

We consider a single server retrial system with two infinite capacity orbit queues accepting three independent arrival streams. The service station can handle at most one job. Arriving jobs are directed initially to the service station. An arriving job (of either stream) that finds the server idle starts service immediately. In case the server is busy upon a job's arrival, the blocked arriving job is routed to an orbit queue as follows: two of the arrival streams are \textit{dedicated} to each orbit queue, i.e., an arriving job of stream $m$ that finds the server busy, it joins orbit queue $m$, $m=1,2$. An arriving job of the third stream, i.e., the \textit{smart} stream, that finds the server busy, it joins the least loaded (i.e., the shorter) orbit queue, and in case of a tie, the \textit{smart} job joins either orbit queue with probability $1/2$. Orbiting jobs retry independently to connect with the service station after a random time period (and according to the \textit{constant} retrial policy) that depends on the type of the orbit queue. Note that the model at hand is described by a non-homogeneous Markov modulated two-dimensional random walk; see Figure \ref{unif}. Our main concern is to investigate the asymptotic behaviour of the GJSOQ system, and to derive stationary approximations of its joint orbit queue length distribution.
\subsection{Related work}
The join-the-shortest-queue (JSQ) policy is widely used for load-balancing in stochastic networks. Yet the analytic derivation of the stationary distribution is known to be far from trivial.

The stationary behaviour of the standard (i.e., non-modulated, without retrials) two-dimensional JSQ problem had initially studied in \cite{haight}, and further investigated in \cite{king}, which was shown that the minimum queue length has exactly geometric asymptotics using the generating function approach. In \cite{coh1,fay1}, the authors presented a robust mathematical approach through generating functions and complex variables arguments to study its stationary behaviour. However, this approach does not provide explicit expressions for the equilibrium distribution, and it is not so useful for numerical computations. The compensation method (CM), introduced in \cite{ad0,ad2,ad3}, provides an elegant and direct method to obtain explicitly the equilibrium joint queue-length distribution as an infinite series of product form terms, by solving directly the equilibrium equations. 

Numerical/approximation methods were also applied: see the power series algorithm (PSA), e.g., \cite{blanc1987}, and the matrix geometric method; see e.g., \cite{gerts,raoposner}, for which connections with CM was recently reported in \cite{stella}. PSA is numerically satisfactory for relatively lower dimensional models, although, the theoretical foundation of this method is still incomplete. By expressing the equilibrium distribution as a power series in some variable based on the model parameters (usually the load), PSA transforms the balance equations into a recursively solvable set of equations by adding one dimension to the state space. For the (non-modulated) multidimensional JSQ model the authors in \cite{Houtum} constructed upper and lower bounds for any performance measure based on two related systems that were easier to analyse. For a comparative analysis of the methods used for the analysis of multidimensional queueing models (including JSQ) see \cite{Onno}.

The stationary behaviour of a two-queue system under the JSQ policy with Erlang arrivals was investigated in \cite{ad1} by using the CM. The queueing model in \cite{ad1} is described by a multilayer random walk in the quarter plane. Quite recently, in \cite{dimisq} the CM was applied to investigate the joint stationary distribution for a Markov-modulated random walk in the quarter plane, which describes a standard (i.e., without dedicated traffic streams) symmetric join the shortest orbit queue system with retrials. For such a model, was also shown that the tail decay rate of the minimum orbit queue has an exactly geometric asymptotic.

Since the exact solutions discussed above are extremely complicated, it is useful to evaluate these
expressions in certain limiting cases, in order to gain more insight into the qualitative structure of
the particular model. Asymptotic formulas often clearly show the dependence of the solutions on
the various variables/parameters in the problem, whereas the full exact expressions may be difficult
to interpret in terms of the underlying model. Clearly, an asymptotic formula can never contain
as much quantitative (numerical) information as an exact answer, but it provides reasonably
accurate numerical results at a greatly reduced computational cost. 

The tail asymptotic behaviour for the two-queue case have been extensively studied. For the standard JSQ, i.e., without retrials, the problem for the case of homogeneous servers was answered in \cite{king}, while for the case of heterogeneous servers in \cite{taka}. The behaviour of the standard generalized JSQ (GJSQ) problem was investigated by using a random walk structure in \cite{foley}, and by using a quasi birth-death (QBD) formulation in \cite{limiya}; see also \cite{kurksuh} which extends Malyshev's approach \cite{malyas} to the GJSQ paradigm. However, those two papers have not completely solved the tail asymptotic problem, since they focus on the so-called \textit{strongly pooled} condition. By using the Markov additive approach and the optimization technique developed in \cite{miyatail}, the author in \cite{miya2sided} completely characterized the weak decay rates in terms of the transition probabilities and provided a complete solution. The decay rate
of stationary probabilities was also analyzed in heavy traffic, and via large deviations in \cite{turn1}, \cite{turn2}, respectively; see also \cite{puha}. We further mention \cite{kness}, where the authors obtain heuristics from the balance equations. 

In \cite{sakuma}, the authors studied the tail decay problem of JSQ system with two queues and PH-type arrivals; see also in \cite{sakuma2} for the case of more than two queues and threshold jockeying. These works investigate the tail behaviour of non-homogeneous multilayered random walks. We further mention the recent work in \cite{oza}, which focused on a thorough investigation of the tail behaviour of space homogeneous two-dimensional skip-free Markov modulated reflecting random walk; see also \cite{miya}. Using a QBD and a reflecting random walk formulation the authors in \cite{sakuma3} showed that the tail asymptotic of the minimum queue length in a Markovian system of $k\geq 3$ parallel queues is exactly geometric.
\subsection{Contribution}
This work considers for the first time the \textit{generalized join the shortest queue policy with retrials}. We use this model as a vehicle to investigate the asymptotic behaviour of non-homogeneous Markov-modulated two dimensional random walks, and with a particular interest on a modulation that allows a completely tractable analysis. We first investigate the stability conditions. Then, we focus on the tail asymptotic behaviour of the stationary distribution, and show that under specific conditions the tail asymptotic of the minimum orbit queue length is exactly geometric. Finally, by using directly the equilibrium equations, we provide an asymptotic approach to obtain
approximations to the steady-state joint orbit queue length distribution for both states of the server. The approximation results agreed with the results obtained by the asymptotic analysis. 
\paragraph{Fundamental Contribution}
In \cite{dimisq} we provided exact expressions for the stationary distribution of the standard (i.e., with \textit{smart} arrival stream and \textit{no dedicated} arrival streams), symmetric (i.e., identical retrial rates) JSOQ system by using the CM. To our best knowledge, the work in \cite{dimisq} is the only work in the related literature that provided analytic results regarding the stationary behaviour of the JSQ policy with retrials. It is well-known from the related literature on the standard (i.e., without retrials) JSQ model with two queues that under the presence of the dedicated streams, the CM collapses. This is because under the presence of dedicated traffic, the corresponding two dimensional random walk allows transitions to the \textit{North}, and to the \textit{East}. A similar situation arises also in our case, where in the corresponding Markov modulated two-dimensional random walk we have a similar behaviour; see Figures \ref{unif}, \ref{struc}. Thus, this work provides the only analytic results for the generalized JSQ problem with retrials available in the literature.

Note that our system is described by a non-homogeneous two-dimensional random walk modulated by a two-state Markov process. In such a case, the phase process represents the state of the server, and affects the evolution of the level process, i.e, the orbit queue lengths in two ways: i) The rates at which certain transitions in the level process occur depend on the state of the phase process; see Figure \ref{unifx}. Thus, a change in the phase might not immediately trigger a transition of the level process,
but changes its dynamics (indirect interaction). ii) A phase change does trigger an immediate transition of the level process (direct interaction).

For such a queueing system, we show that the tail asymptotic of the minimum orbit queue length for fixed values of their difference, and server's state is exactly geometric. To accomplish this task we transform the original process to a Markov modulated random walk in the half plane. Then, we focused on the censored Markov chain referring to the busy states, which is now a standard two-dimensional random walk in the half plane, and studied the tail asymptotic behaviour. Using a relation among idle and busy states we also study the tail behaviour for the idle states. To our best knowledge there is no other analytic study on the tail asymptotics properties of Markov modulated random walks in the half plane.

Moreover, we provide a simple heuristic approach to approximate the stationary distribution, valid when one of the orbit queue lengths is large, by distinguishing the analysis between the symmetric and asymmetric cases (this is because the asymmetric case reveals additional technical difficulties compared with the symmetric case; see subsections \ref{prsym}, \ref{prasym}). Our derived theoretical results through the heuristic approach agreed with those derived by the tail asymptotic analysis. Moreover, the advantage of our approach is that we cope directly with the equilibrium equations, without using generating functions or diffusion approximations.

Stability conditions are also investigated. In particular, having in mind that the orbit queues grow only when the server is busy, we focused on the censored chain at busy states, and provide necessary and sufficient conditions for the ergodicity of the censored chain by using Foster-Lyapunov arguments. We conjecture that these conditions are also necessary and sufficient for the original process. Simulation experiments indicate that our conjecture is true but a more formal justification is needed. We postponed the formal proof in a future work.
\paragraph{Application oriented contribution}
Besides its theoretical interest, the queueing model at hand has interesting applications in the modelling of wireless relay-assisted cooperative networks. Such systems operate as follows: There is a finite number of source users that transmit packets (i.e., the arrival streams) to a common destination node (i.e., the single service station), and a finite number of relay nodes (i.e., the orbit queues) that assist the source users by retransmitting their blocked packets, i.e., the packets that cannot access upon arrival the destination node; see e.g., \cite{dim3,PappasTWC2015}. We may consider here a \textit{smart} source user (possibly the user of highest importance that transmits priority packets) that employs the JSQ protocol (i.e., the cooperation strategy among the \textit{smart} source and the relays, under which, the \textit{smart} user chooses to forward its blocked packet to the least loaded relay node), and two \textit{dedicated} source users (that transmit packets of lower priority, or packets that can only handled by the specific relay). This work serves as a major step towards the analysis of even general retrial models operating under the JSQ policy.

The rest of the paper is organized as follows. In Section \ref{model}, we present the mathematical model in detail. Stability conditions are investigated in Section \ref{stabo}. The main results of this work are presented in Section \ref{main}. More precisely, in subsection \ref{taildecay} we present the exact geometric behaviour in the minimum orbit queue length direction under specific conditions (see Theorem \ref{decay} and Corollary \ref{corol}), while in subsection \ref{heur} we present explicit asymptotic expressions for approximating the stationary distribution of the model; see Lemmas \ref{sym}, \ref{asym}. The proofs of the main results are presented in section \ref{prmain}. Our theoretical findings are validated through detailed numerical results that are presented in Section \ref{num}.
\section{System model}\label{model}
We consider a single-server retrial system with two infinite capacity orbit queues. Three independent Poisson streams of jobs, say $S_{m}$, $m=0,1,2$, flow into the single-server service
system, namely the \textit{smart} stream $S_{0}$, and the \textit{dedicated} streams $S_{m}$, $m=1,2$. The arrival rate of stream $S_{m}$ is $\lambda_{m}$, $m=0,1,2,$, and let $\lambda:=\lambda_{0}+\lambda_{1}+\lambda_{2}$. The service system can hold at most one job. The required service time of each job is independent of
its type and is exponentially distributed with rate $\mu$. If an arriving type-$m$, $m=1,2,$ job finds
the (main) server busy, it is routed to a dedicated retrial (orbit) queue that operates as an $\cdot/M/1/\infty$ queue. An arriving type-$0$ job (i.e., the smart job) that finds the server busy, it is routed to the shortest orbit queue (i.e., the least loaded orbit), while in case of a tie, it is routed to either orbit with probability 1/2. Orbiting jobs try to access the server according to a constant retrial policy (i.e., orbits behave as $\cdot/M/1/\infty$ queues). In particular, the orbit queue $m$ attempts to retransmit a job (if any) to the main service system at a Poisson rate of $\alpha_{m}$, $m=1,2$.

Let $N_{m}(t)$ be the number of jobs in orbit queue $m$, $m=1,2$, at time $t$, and $C(t)$ be the state of the server, i.e., $C(t)=1$, when the server is busy, and $C(t)=0$ when it is idle at time $t$, respectively. Then, $X(t)=\{(N_{1}(t),N_{2}(t),C(t));t\geq 0\}$ is an irreducible Markov process with state space $S=\mathbb{N}_{0}\times\mathbb{N}_{0}\times\{0,1\}$, where $\mathbb{N}_{0}=\{0,1,2,\ldots\}$. Let $X=\{(N_{1},N_{2},C)\}$ the stationary version of $\{X(t);t\geq0\}$, and define the stationary probabilities
\begin{displaymath}
p_{i,j}(k)=\lim_{t\to\infty}\mathbb{P}((N_{1}(t),N_{2}(t),C(t))=(i,j,k))=\mathbb{P}((N_{1},N_{2},C)=(i,j,k)),
\end{displaymath}
for $(i,j,k)\in S$. The equilibrium equations are
\begin{eqnarray}
(\lambda+\alpha_{1}1_{\{i>0\}}+\alpha_{2}1_{\{j>0\}})p_{i,j}(0)=\mu p_{i,j}(1),\,i,j\geq 0,\label{e1}\vspace{2mm}\\
(\lambda+\mu)p_{i,j}(1)=[\lambda_{0}H(i-j+1)+\lambda_{2}]p_{i,j-1}(1)+[\lambda_{0}H(j-i+1)+\lambda_{1}]p_{i-1,j}(1)\nonumber\\
+\lambda p_{i,j}(0)+\alpha_{1}p_{i+1,j}(0)+\alpha_{2}p_{i,j+1}(0),\,i,j\geq 1,\label{e2}\vspace{2mm}\\
(\lambda+\mu)p_{i,0}(1)=\lambda p_{i,0}(0)+\lambda_{1}p_{i-1,0}(1)+\alpha_{1}p_{i+1,0}(0)+\alpha_{2}p_{i,1}(0),\,i\geq 2,\label{e3}\vspace{2mm}\\
(\lambda+\mu)p_{0,j}(1)=\lambda p_{0,j}(0)+\lambda_{2}p_{0,j-1}(1)+\alpha_{1}p_{1,j}(0)+\alpha_{2}p_{0,j+1}(0),\,j\geq 2,\label{e4}\vspace{2mm}\\
(\lambda+\mu)p_{1,0}(1)=\lambda p_{1,0}(0)+\alpha_{1}p_{2,0}(0)+\alpha_{2}p_{1,1}(0)+(\lambda_{1}+\frac{\lambda_{0}}{2})p_{0,0}(1),\label{e5}\vspace{2mm}\\
(\lambda+\mu)p_{0,1}(1)=\lambda p_{0,1}(0)+\alpha_{1}p_{1,1}(0)+\alpha_{2}p_{0,2}(0)+(\lambda_{2}+\frac{\lambda_{0}}{2})p_{0,0}(1),\label{e6}\vspace{2mm}\\
(\lambda+\mu)p_{0,0}(1)=\lambda p_{0,0}(0)+\alpha_{1}p_{1,0}(0)+\alpha_{2}p_{0,1}(0),\label{e7}
\end{eqnarray}
with the normalization condition $\sum_{i=0}^{\infty}\sum_{j=0}^{\infty}(p_{i,j}(0)+p_{i,j}(1))=1$, and where
\begin{displaymath}
H(n)=\left\{\begin{array}{ll}
1,&n\geq 1,\\
\frac{1}{2},&n=0,\\
0,&n\leq-1.
\end{array}\right.
\end{displaymath}

For reasons that will become clear in the following sections, we consider the corresponding uniformized discrete time Markov chain through the uniformization technique. Since the total transition rate from each state is bounded by $\lambda+\mu+\alpha_{1}+\alpha_{2}$, we can construct by uniformization a discrete time Markov chain with the same stationary distribution as that of $\{X(t);t\geq 0\}$. Without loss of generality, let the uniformization parameter
$\theta =\lambda+\mu+\alpha_{1}+\alpha_{2}= 1$. The uniformized Markov chain $X(n)=\{(N_{1,n},N_{2,n},C_{n})\}$ of $\{X(t);t\geq 0\}$ has six regions of spatial homogeneity: two angles, say $r_{1}=\{i>j> 0,k=0,1\}$ and $r_{2}=\{j>i> 0,k=0,1\}$, three rays, say $h=\{j=0,i>0,k=0,1\}$, $v=\{i=0,j>0,k=0,1\}$, $d=\{i=j>0,k=0,1\}$ and the points $\{(0,0,k), k=0,1\}$. Then, the matrix transition diagram partitioned according to the state of the
server, is depicted in Figure \ref{unif}, where
\begin{displaymath}
\begin{array}{c}
A_{1,0}^{(r_1)}=\begin{pmatrix}
0&0\\
0&\lambda_{1}
\end{pmatrix}=A_{1,0}^{(h)}=A_{1,0}^{(0)},\,A_{1,0}^{(r_2)}=\begin{pmatrix}
0&0\\
0&\lambda_{1}+\lambda_{0}
\end{pmatrix}=A_{1,0}^{(v)},\,\,A_{1,0}^{(d)}=\begin{pmatrix}
0&0\\
0&\lambda_{1}+\lambda_{0}/2
\end{pmatrix},\\
A_{0,1}^{(r_2)}=\begin{pmatrix}
0&0\\
0&\lambda_{2}
\end{pmatrix}=A_{0,1}^{(v)}=A_{0,1}^{(0)},\,A_{0,1}^{(r_1)}=\begin{pmatrix}
0&0\\
0&\lambda_{0}+\lambda_{2}
\end{pmatrix}=A_{0,1}^{(h)},\,A_{0,1}^{(d)}=\begin{pmatrix}
0&0\\
0&\lambda_{0}/2+\lambda_{2}
\end{pmatrix},\\
A_{0,-1}^{(r_2)}=\begin{pmatrix}
0&\alpha_{2}\\
0&0
\end{pmatrix}=A_{0,-1}^{(v)}=A_{0,-1}^{(d)}=A_{0,-1}^{(r_1)},\,A_{0,0}^{(v)}=\begin{pmatrix}
\mu+\alpha_{1}&\lambda\\
\mu&\alpha_{1}+\alpha_{2}
\end{pmatrix},\\
A_{-1,0}^{(r_1)}=\begin{pmatrix}
0&\alpha_{1}\\
0&0
\end{pmatrix}=A_{-1,0}^{(h)}=A_{-1,0}^{(d)}=A_{-1,0}^{(r_2)},\,A_{0,0}^{(h)}=\begin{pmatrix}
\mu+\alpha_{2}&\lambda\\
\mu&\alpha_{1}+\alpha_{2}
\end{pmatrix},\\
A_{0,0}^{(r_1)}=\begin{pmatrix}
\mu&\lambda\\
\mu&\alpha_{1}+\alpha_{2}
\end{pmatrix}=A_{0,0}^{(r_2)}=A_{0,0}^{(d)},\,A_{0,0}^{(0)}=\begin{pmatrix}
\mu+\alpha_{1}+\alpha_{2}&\lambda\\
\mu&\alpha_{1}+\alpha_{2}
\end{pmatrix}.
\end{array}
\end{displaymath}
\begin{figure}[H]
\centering
\includegraphics[scale=0.65]{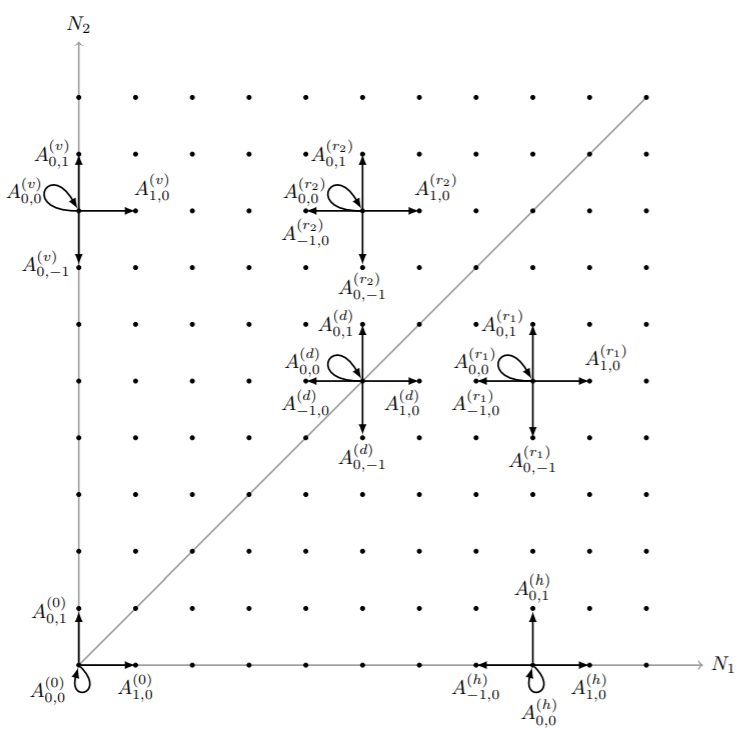}
\caption{Matrix transition diagram.}
\label{unif}
\end{figure}
An illustration of the transitions from a state belonging to the angle $r_{1}$ (e.g., from the state $(4,2,k)$, $k=0,1$), and from a state belonging to the ray $d$ (e.g., from the state $(3,3,k)$, $k=0,1$), is given in Figure \ref{unifx}, left and right, respectively.
\begin{figure}[H]
\centering
\includegraphics[scale=0.5]{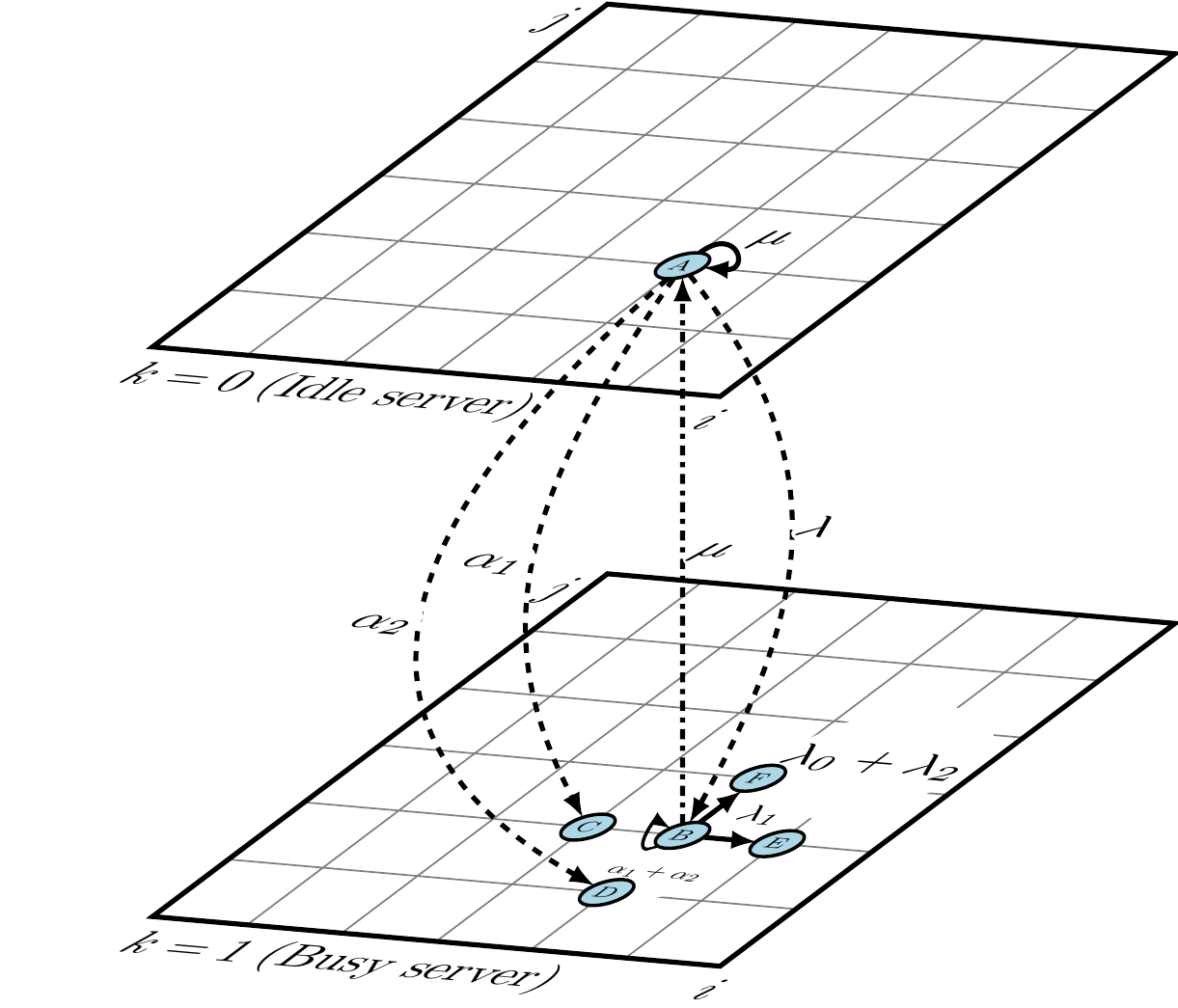}\includegraphics[scale=0.5]{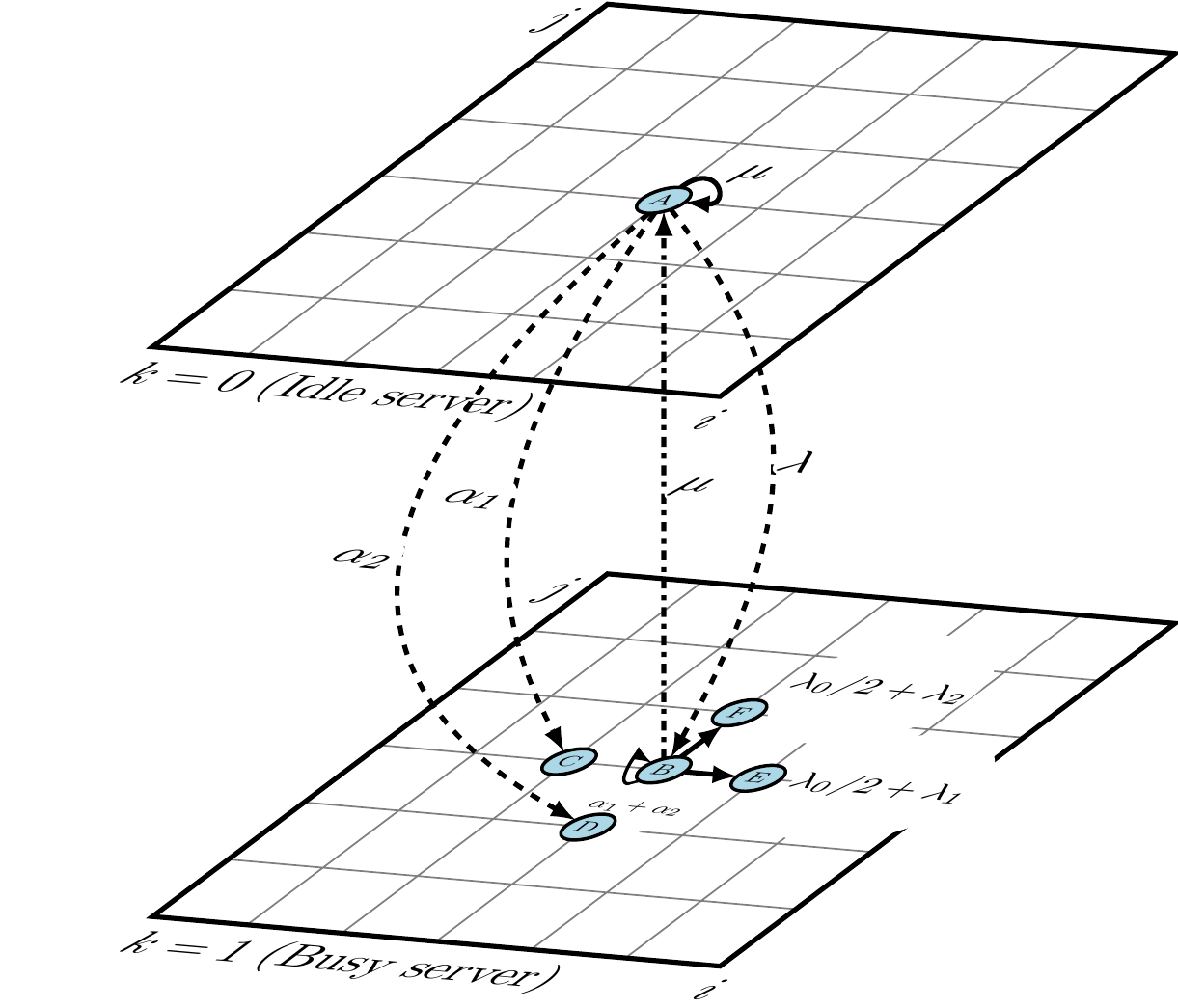}
\caption{An instance of state transitions from states belonging to angle $r_{1}$ (left), and ray $d$ (right), given the state of the server.}
\label{unifx}
\end{figure}
Lets say some more words regarding the derivation of $A_{i,j}^{(q)}$, $-1\leq i,j\leq 1$, $q=r_{1},r_{2},h,v,0$. For example, $A_{1,0}^{(r_1)}$ contains the transition probabilities from a state that belongs to the angle $r_{1}$ and results in an increase by one at the orbit queue 1, i.e.,  from $(i,j,k)$ to $(i+1,j,l)$, where $i>j$, $k,l=0,1$. Clearly, such a transition occurs only when we have a dedicated arrival of type 1, and the server is busy. Similarly, given a state in the ray $d$, $A_{1,0}^{(d)}$ contains transition probabilities that result in an increase by one at orbit queue 1, given that both orbits have the same occupancy. Such a transition is done either with the arrival of a dedicated job, or with the arrival of a smart job who sees both orbits with the same number of jobs and is routed with probability $1/2$ to orbit queue 1. The rest of $A_{i,j}^{(q)}$, $-1\leq i,j\leq 1$, $q=r_{1},r_{2},h,v,0$, are constructed similarly. 

In the following sections, we provide our main results that refer to the asymptotic behaviour of $\{X(t);t\geq 0\}$. We first investigate the stability conditions. Then, our first main result refers to the investigation of the decay rate of the tail probabilities for the shortest queue length in steady state; see Theorem \ref{decay}. We cope with this task in a number of steps, after considering the corresponding discrete time Markov chain $\{X(n);n\geq 0\}$ through the uniformization technique: 
\begin{enumerate}
\item We transform the uniformized Markov chain to create the minimum and the difference of the queue states process (i.e., the Markov modulated random walk in the half-plane), and then consider the censored process at busy states.
\item We investigate the stationary tail decay rate of the censored process at busy states.
\item Using a relation among busy and idle states, we show that the asymptotic properties of the shortest queue for the idle states is the same with its asymptotic properties for the busy states.
\end{enumerate}

Our second main result relies on a heuristic approach to construct approximations for the stationary distribution of the original process. In particular, our  aim (see Lemmas \ref{sym}, \ref{asym}) is to solve \eqref{e1}-\eqref{e7} as either $i$ or $j\to\infty$. This task is accomplished by constructing a solution of \eqref{e1}-\eqref{e7} separately in each of the regions defined as follows: Region I: $i\gg 1$, $j\gg 1$, Region II: $i=0$ or $1$, $j\gg 1$,
Region III: $i\gg 1$, $j=0$ or $1$.
\section{On the stability condition}\label{stabo}
In this section, we investigate the stability condition of the model at hand. Stability condition for the standard generalized join the shortest queue (GJSQ) model without retrials was recently investigated; see e.g., \cite{foley,kurk}. Clearly, the presence of retrials along with the presence of the dedicated arrival flows to each orbit when the server is busy, complicates the problem considerably. Note that our model is described by a non-homogeneous Markov modulated two-dimensional nearest-neighbour random walk, and its stability condition, to our best knowledge, is still an open problem. We mention here that the stability condition of a \textit{homogeneous} Markov modulated two-dimensional nearest neighbour random walk was recently investigated in \cite{ozawa} by using the concept of \textit{induced Markov chains}. 

Here on, by noting that the orbit queues grow only when the server is busy, we construct a new discrete time Markov chain embedded at epochs in which the server is busy, i.e., the \textit{censored} Markov chain at busy states (censored Markov chains have been widely used for proving the uniqueness of the stationary vector
for a recurrent countable-state Markov chain \cite{kem,zhaoliu}). Then, using standard Foster-Lyapunov arguments, we provide its stability conditions, and having in mind that orbits grow only when the server is busy (it is natural to assume that the behaviour of the original process at the busy states heavily affects its convergence), we conjecture that the obtained stability conditions for the censored Markov chain on the busy states coincides with those of the original one. Simulation experiments indicate that our conjecture is true. The formal justification of our conjecture is an interesting open problem and we let it be a future study.

We first consider the uniformized discrete time Markov
chain, say $X(n)=\{(N_{1,n},N_{2,n},C_{n});n\geq 0\}$, of $\{X(t);t\geq0\}$ with transition diagram given in \eqref{unif} and state space $S=E\cup \tilde{E}$, where $E=\mathbb{N}_{0}\times\mathbb{N}_{0}\times\{1\}$, $E^{c}=\mathbb{N}_{0}\times\mathbb{N}_{0}\times\{0\}$. Then, we partition the transition matrix $P$ of $X(n)$ according to $E$, $E^{c}$ into:
\begin{displaymath}
P=\bordermatrix{&E^{c}&E\cr 
E^{c}&P_{0,0}& P_{0,1}\cr
E& P_{1,0} &P_{1,1}},
\end{displaymath}
where 
\begin{displaymath}
\begin{array}{rl}
P_{0,0}=diag(A_{0},A_{1},A_{1},\ldots),&P_{1,0}=diag(L_{0},L_{0},L_{0},\ldots),\\
P_{0,1}=\begin{pmatrix}
B_{0}&&&&\\
B_{1}&B_{0}&&&\\
&B_{1}&B_{0}&&\\
&&\ddots&\ddots&
\end{pmatrix},&P_{1,1}=\begin{pmatrix}
D_{0,0}&D_{0,1}&&&\\
&D_{1,1}&D_{1,2}&&\\
&&D_{2,2}&D_{2,3}&\\
&&&\ddots&\ddots
\end{pmatrix},\\
A_{0}=diag(\mu+\alpha_{1}+\alpha_{2},\mu+\alpha_{1},\mu+\alpha_{1},\ldots),&A_{1}=diag(\mu+\alpha_{2},\mu,\mu,\ldots),\\
L_{0}=\mu I_{\infty},&B_{1}=\alpha_{1}I_{\infty},\\
B_{0}=\begin{pmatrix}
\lambda&&&&\\
\alpha_{2}&\lambda&&&\\
&\alpha_{2}&\lambda&&\\
&&\ddots&\ddots&
\end{pmatrix},& D_{i,i}=(\alpha_{1}+\alpha_{2})I_{\infty}+Q_{i+1,i+2},\,i=0,1,\ldots,
\end{array}
\end{displaymath}
where $I_{\infty}$ is the identity matrix of infinite dimension and $Q_{i+1,i+2}$ is a infinite dimension matrix with only non-zero entries at the superdiagonal and such that the $(k,k+1)-$entry equal to $\lambda_{0}+\lambda_{2}$, $k=1,2,\ldots,i$, the $(i+1,i+2)-$entry equals $\frac{\lambda_{0}}{2}+\lambda_{2}$, and the $(i+k,i+k+1)-$entry equals $\lambda_{2}$, $k=2,3,\ldots$. Finally, the matrix $D_{i,i+1}$, $i=0,1,\ldots$ is a diagonal matrix of infinite dimension with $(k,k)-$entry equal to $\lambda_{1}$, $k=1,2,\ldots,i$, the $(i+1,i+1)-$entry equals $\frac{\lambda_{0}}{2}+\lambda_{1}$, and the $(i+k,i+k)-$entry equals $\lambda_{0}+\lambda_{1}$, $k=2,3,\ldots$. In what follows, denote $\widehat{\lambda}:=\lambda(\lambda+\alpha_{1}+\alpha_{2})$, $\widehat{\lambda}_{k}:=\lambda_{k}(\lambda+\alpha_{1}+\alpha_{2})$, $k=0,1,2,$ and $\widehat{\mu}_{k}=\mu\alpha_{k}$, $k=1,2$.

Since $P_{0,0}$ is a diagonal matrix, its fundamental matrix, say $\tilde{P}_{0,0}$ has the form
\begin{displaymath}
\begin{array}{rl}
\tilde{P}_{0,0}=&\sum_{n=0}^{\infty}P_{0,0}^{n}=diag(\tilde{A}_{0},\tilde{A}_{1},\tilde{A}_{1}),\\
\tilde{A}_{0}=&diag(\frac{1}{\lambda},\frac{1}{\lambda+\alpha_{2}},\frac{1}{\lambda+\alpha_{2}},\ldots)\\
\tilde{A}_{1}=&diag(\frac{1}{\lambda+\alpha_{1}},\frac{1}{\lambda+\alpha_{1}+\alpha_{2}},\frac{1}{\lambda\alpha_{1}+\alpha_{2}},\ldots).
\end{array}
\end{displaymath}
The censored chain $\{\tilde{X}(n);n\geq0\}$ at busy states has six regions
of spatial homogeneity: two angles, say $r_{1}=\{i>j> 0\}$ and $r_{2}=\{j>i> 0\}$, three rays, say $h=\{j=0,i>0\}$, $v=\{i=0,j>0\}$, $d=\{i=j>0\}$ and the point $(0,0)$. Then, the one step transition probability matrix of the censored chain $\{\tilde{X}(n);n\geq0\}$, given by $P^{(E)}=P_{1,1}+P_{1,0}\tilde{P}_{0,0}P_{0,1}$ is as follows:
\begin{itemize}
\item In region $i>j>0$ (labelled as $r_{1}$),
\begin{displaymath}
\begin{array}{c}
p_{1,0}^{(r_{1})}=\lambda_{1},p_{0,1}^{(r_{1})}=\lambda_{0}+\lambda_{2},p_{-1,0}^{(r_{1})}=\frac{\widehat{\mu}_{1}}{\lambda+\alpha_{1}+\alpha_{2}},p_{0,-1}^{(r_{1})}=\frac{\widehat{\mu}_{2}}{\lambda+\alpha_{1}+\alpha_{2}},p_{0,0}^{(r_{1})}=\alpha_{1}+\alpha_{2}+\frac{\lambda\mu}{\lambda+\alpha_{1}+\alpha_{2}},
\end{array}
\end{displaymath}
\item In region $j>i>0$ (labelled as $r_{2}$),
\begin{displaymath}
\begin{array}{c}
p_{1,0}^{(r_{2})}=\lambda_{0}+\lambda_{1},p_{0,1}^{(r_{2})}=\lambda_{2},p_{-1,0}^{(r_{2})}=\frac{\widehat{\mu}_{1}}{\lambda+\alpha_{1}+\alpha_{2}},p_{0,-1}^{(r_{2})}=\frac{\widehat{\mu}_{2}}{\lambda+\alpha_{1}+\alpha_{2}},p_{0,0}^{(r_{2})}=\alpha_{1}+\alpha_{2}+\frac{\lambda\mu}{\lambda+\alpha_{1}+\alpha_{2}},
\end{array}
\end{displaymath}
\item In region $i=j>0$ (labelled as $d$),
\begin{displaymath}
\begin{array}{c}
p_{1,0}^{(d)}=\frac{\lambda_{0}}{2}+\lambda_{1},p_{0,1}^{(d)}=\frac{\lambda_{0}}{2}+\lambda_{2},p_{-1,0}^{(d)}=\frac{\widehat{\mu}_{1}}{\lambda+\alpha_{1}+\alpha_{2}},p_{0,-1}^{(d)}=\frac{\widehat{\mu}_{2}}{\lambda+\alpha_{1}+\alpha_{2}},p_{0,0}^{(d)}=\alpha_{1}+\alpha_{2}+\frac{\lambda\mu}{\lambda+\alpha_{1}+\alpha_{2}},
\end{array}
\end{displaymath}
\item In region $i=0$, $j>0$ (labelled as $v$),
\begin{displaymath}
\begin{array}{c}
p_{1,0}^{(v)}=\lambda_{0}+\lambda_{1},p_{0,1}^{(v)}=\lambda_{2},p_{0,-1}^{(v)}=\frac{\widehat{\mu}_{2}}{\lambda+\alpha_{2}},p_{0,0}^{(v)}=\alpha_{1}+\alpha_{2}+\frac{\lambda\mu}{\lambda+\alpha_{2}},
\end{array}
\end{displaymath}
\item In region $j=0$, $i>0$ (labelled as $h$),
\begin{displaymath}
\begin{array}{c}
p_{1,0}^{(h)}=\lambda_{1},p_{0,1}^{(h)}=\lambda_{0}+\lambda_{2},p_{-1,0}^{(h)}=\frac{\widehat{\mu}_{1}}{\lambda+\alpha_{1}},p_{0,0}^{(h)}=\alpha_{1}+\alpha_{2}+\frac{\lambda\mu}{\lambda+\alpha_{1}},
\end{array}
\end{displaymath}
\item For $i=j=0$ (labelled as $O$),
\begin{displaymath}
\begin{array}{c}
p_{1,0}^{(O)}=\lambda_{1}+\frac{\lambda_{0}}{2},p_{0,1}^{(O)}=+\frac{\lambda_{0}}{2}+\lambda_{2},p_{0,0}^{(O)}=\alpha_{1}+\alpha_{2}+\mu.
\end{array}
\end{displaymath}
\end{itemize}
\begin{remark}
Note that the censored chain $\{\tilde{X}(n);n\geq 0\}$ describes a new discrete time queueing system consisting of two parallel \textit{coupled queues} with three job arrival streams, where one of them joins the shortest queue, and the other two are dedicated to each queue. The \textit{coupling} feature of the two queues is easily realised by noting that e.g., $p_{-1,0}^{(h)}=\frac{\widehat{\mu}_{2}}{\lambda+\alpha_{1}}>\frac{\widehat{\mu}_{2}}{\lambda+\alpha_{1}+\alpha_{2}}=p_{-1,0}^{(r_{2})}$. The combination of the JSQ feature, along with the coupled processors feature considerably complicates the analysis.
\end{remark}

\begin{lemma}\label{stab}
The censored chain $\{\tilde{X}(n);n\geq 0\}$ is positive recurrent if and only if one of the following conditions hold:
\begin{enumerate}
\item $\rho_{1}:=\frac{\widehat{\lambda}_{1}}{\widehat{\mu}_{1}}<1$, $\rho_{2}:=\frac{\widehat{\lambda}_{2}}{\widehat{\mu}_{2}}<1$, $\rho=\frac{\widehat{\lambda}}{\widehat{\mu}_{1}+\widehat{\mu}_{2}}<1$,
\item $\rho_{1}\geq 1$, $f_{1}:=\frac{\lambda(\lambda_{1}+\alpha_{1})}{\widehat{\mu}_{1}}-1<0$,
\item $\rho_{2}\geq 1$, $f_{2}:=\frac{\lambda(\lambda_{2}+\alpha_{2})}{\widehat{\mu}_{2}}-1<0$.
\end{enumerate}
\end{lemma}
\begin{proof}
The proof is given in Appendix \ref{app1}.
\end{proof}

We now study the dynamics of the orbits in the original process $\{X(t);t\geq 0\}$ based on the stability criteria of the censored process $\{\tilde{X}(n);n\geq 0\}$ given in Lemma \ref{stab}. 

Figure \ref{ssst1} refers to the stability criteria 1. In Figure \ref{ssst1} (left) we have $\rho=0.78>max\{\rho_{1}=0.433,\rho_{2}=0.325\}$ and $\widehat{\lambda}_{0}-|\widehat{\lambda}_{1}-\widehat{\lambda}_{2}+\rho^{2}(\widehat{\mu}_{2}-\widehat{\mu}_{1})|=19.084>0$. It is seen that in such a case the original process is stable, and that both orbits are well balanced. This case corresponds to the strongly pooled case mentioned in \cite[Theorem 2]{foley}, i.e., the proportion of jobs that is routed to the least loaded orbit queue in case the server is busy is large. In Figure \ref{ssst1} (middle), we have the case $1>\rho>max\{\rho_{1},\rho_{2}\}$, with $\rho_{1}>\rho_{2}$ but now $\widehat{\lambda}_{0}-|\widehat{\lambda}_{1}-\widehat{\lambda}_{2}+\rho^{2}(\widehat{\mu}_{2}-\widehat{\mu}_{1})|<0$. The original process is still stable but we can observe that the orbit lengths are not so close any more. This means that the stream that joins the shortest orbit failed to keep the orbit queue lengths close enough to each other. In case $\rho>1$ both orbits are unstable, even though $\rho_{1}<1$, $\rho_{2}<1$ (Figure \ref{ssst1} (right)). 

\begin{figure}[H]
\centering
\includegraphics[scale=0.25]{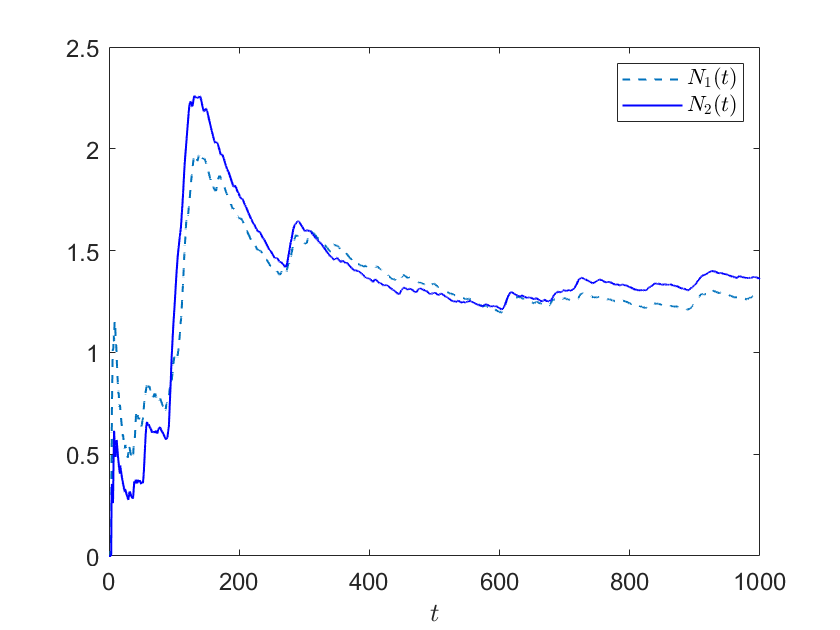}\includegraphics[scale=0.25]{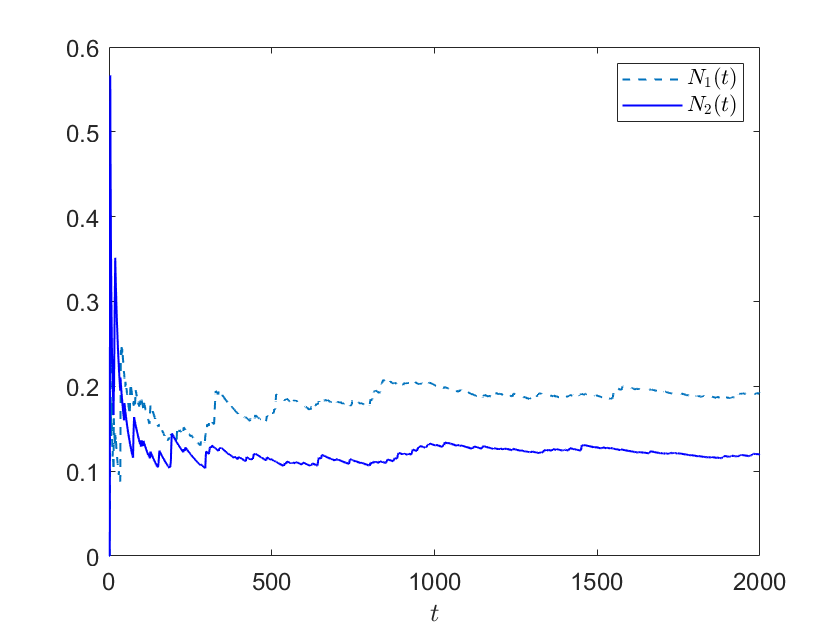}\includegraphics[scale=0.25]{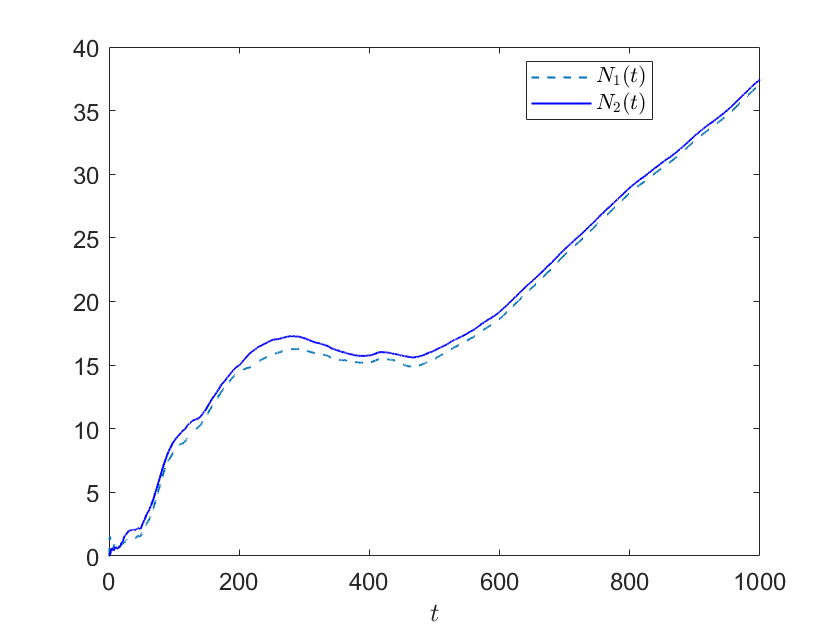}
\caption{Orbit dynamics for case $1>\rho>max\{\rho_{1},\rho_{2}\}$, $\widehat{\lambda}_{0}-|\widehat{\lambda}_{1}-\widehat{\lambda}_{2}+\rho^{2}(\widehat{\mu}_{2}-\widehat{\mu}_{1})|>0$ (left), $1>\rho>max\{\rho_{1},\rho_{2}\}$, $\widehat{\lambda}_{0}-|\widehat{\lambda}_{1}-\widehat{\lambda}_{2}+\rho^{2}(\widehat{\mu}_{2}-\widehat{\mu}_{1})|<0$ (middle), and for the case $\rho>1$, $\rho_{1}<1$, $\rho_{2}<1$ (right).}
\label{ssst1}
\end{figure}
In Figure \ref{ssstr2} (left) we have $\rho_{1}=0.428>\max\{\rho=0.4042,\rho_{2}=0.2675\}$, and $\widehat{\lambda}_{0}-|\widehat{\lambda}_{1}-\widehat{\lambda}_{2}+\rho^{2}(\widehat{\mu}_{2}-\widehat{\mu}_{1})|=-2.393<0$. In such a case the system is still stable, but the \textit{smart} stream failed to keep the orbit queue length very close. Similar observations can be also deduced from Figure \ref{ssstr2} (right) for which $\rho_{1}=0.0944>\max\{\rho=0.061,\rho_{2}=0.021\}$, and $\widehat{\lambda}_{0}-|\widehat{\lambda}_{1}-\widehat{\lambda}_{2}+\rho^{2}(\widehat{\mu}_{2}-\widehat{\mu}_{1})|=0.7663>0$.
\begin{figure}[H]
\centering
\includegraphics[scale=0.25]{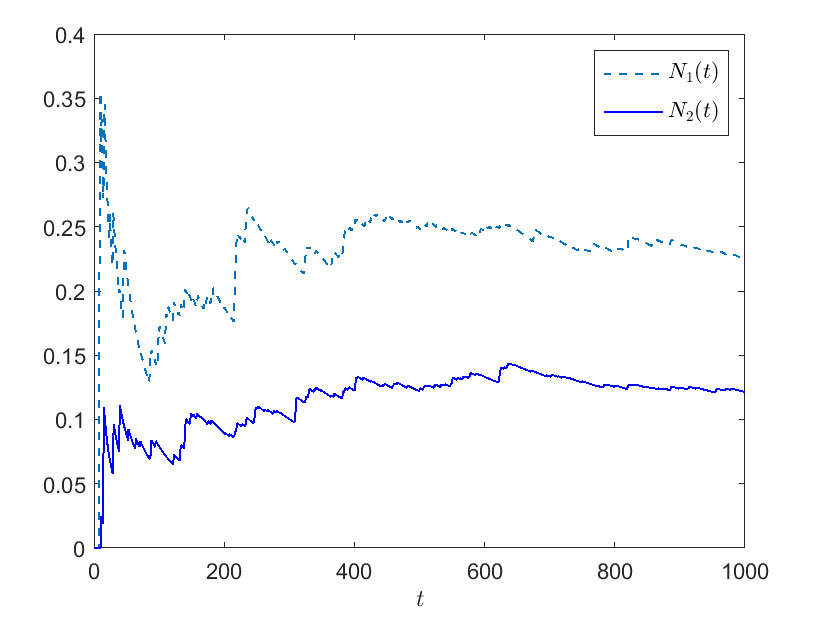}\includegraphics[scale=0.25]{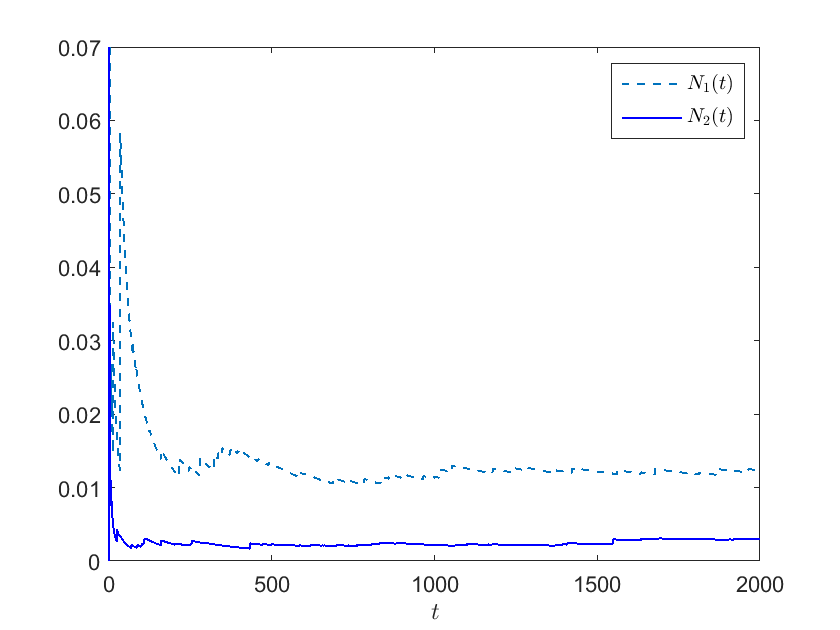}
\caption{Orbit dynamics for case $1>\rho_{1}>max\{\rho,\rho_{2}\}$, $\widehat{\lambda}_{0}-|\widehat{\lambda}_{1}-\widehat{\lambda}_{2}+\rho^{2}(\widehat{\mu}_{2}-\widehat{\mu}_{1})|<0$ (left), $1>\rho_{1}>max\{\rho,\rho_{2}\}$, $\widehat{\lambda}_{0}-|\widehat{\lambda}_{1}-\widehat{\lambda}_{2}+\rho^{2}(\widehat{\mu}_{2}-\widehat{\mu}_{1})|<0$ (right).}
\label{ssstr2}
\end{figure}

In  Figure \ref{ssst2} (left), although $\rho_{1}>1$ ($\rho<1$, $\rho_{2}<1$), the fact that $f_{1}<0$ keeps the network stable. On the other hand when $f_{1}>0$ (Fig. \ref{ssst2}, middle), orbit queue 1 becomes unstable, and so it is the network. Note that in such a case orbit queue 2 remains stable. In Figure \ref{ssst2} (right), when $\rho>1$ both orbits becomes unstable, even though $\rho_{2}<1$. Similar observations can be deduced from Figure \ref{ssst3} that refers to the stability criteria 3. 

\begin{figure}[H]
\centering
\includegraphics[scale=0.25]{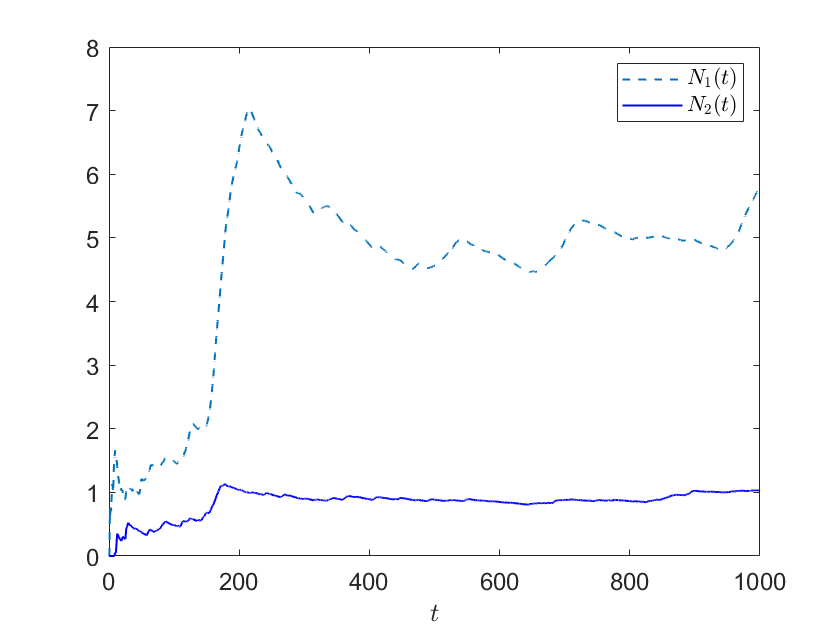}\includegraphics[scale=0.25]{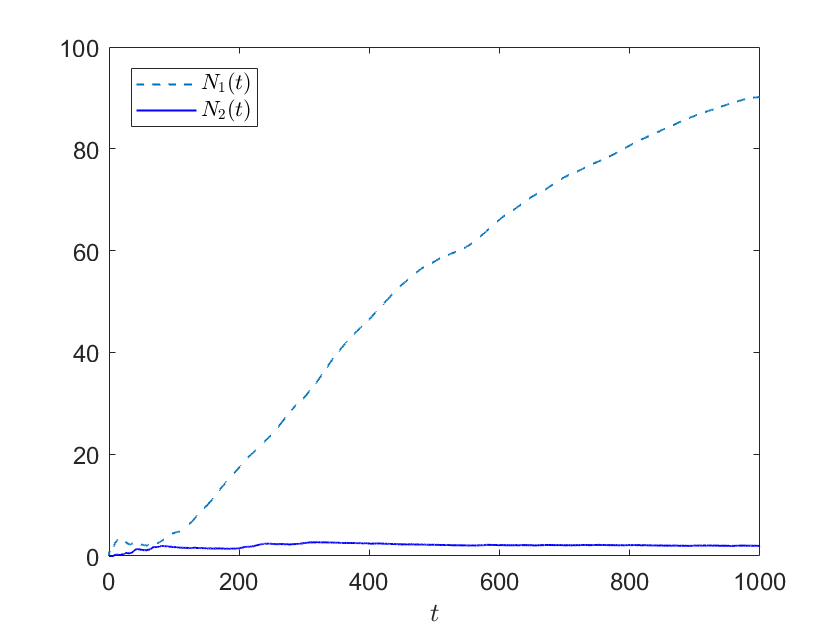}\includegraphics[scale=0.25]{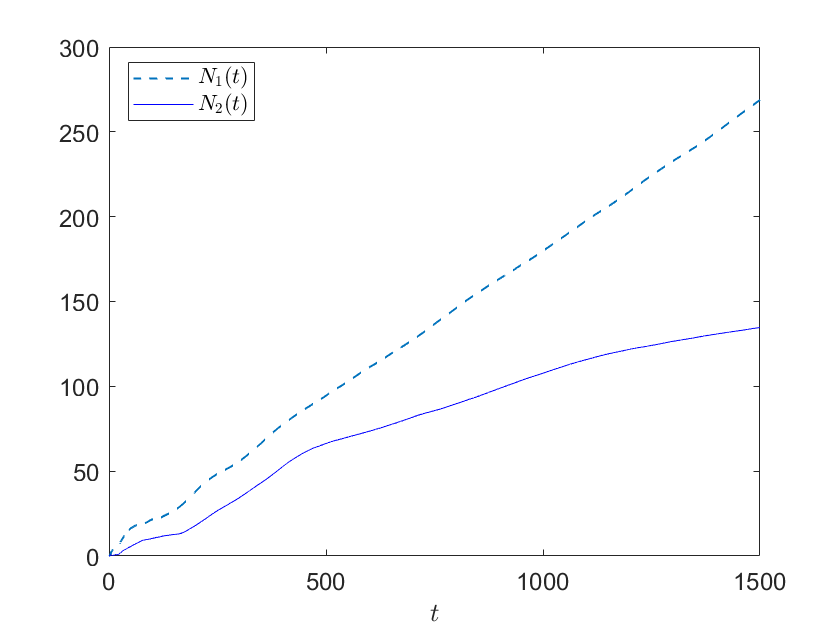}
\caption{Orbit dynamics for case $\rho<1$, $\rho_{1}>1$, $\rho_{2}<1$, $f_{1}<0$ (left), for the case $\rho<1$, $\rho_{1}>1$, $\rho_{2}<1$, $f_{1}>0$ (middle), and for the case $\rho>1$, $\rho_{1}>1$, $\rho_{2}<1$, $f_{1}>0$ (right).}
\label{ssst2}
\end{figure}
\begin{figure}[H]
\centering
\includegraphics[scale=0.25]{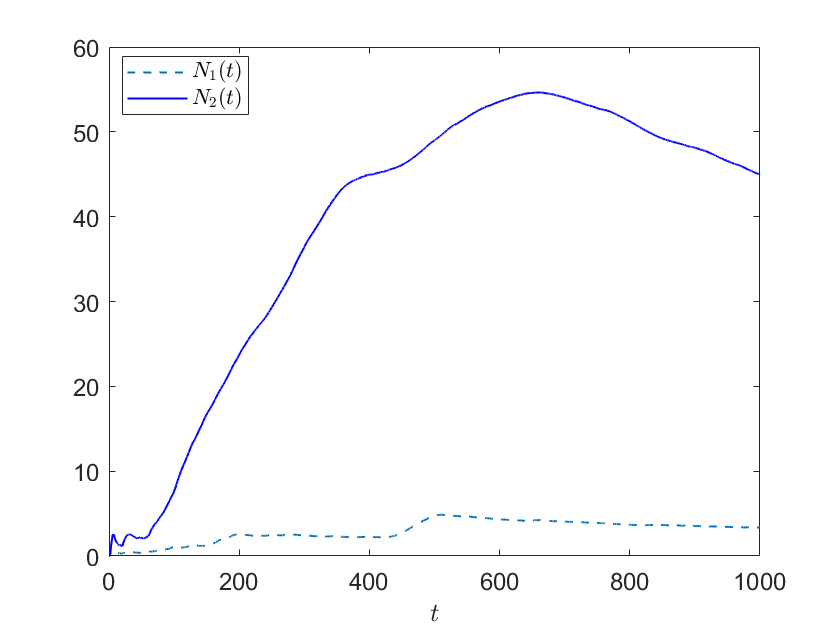}\includegraphics[scale=0.25]{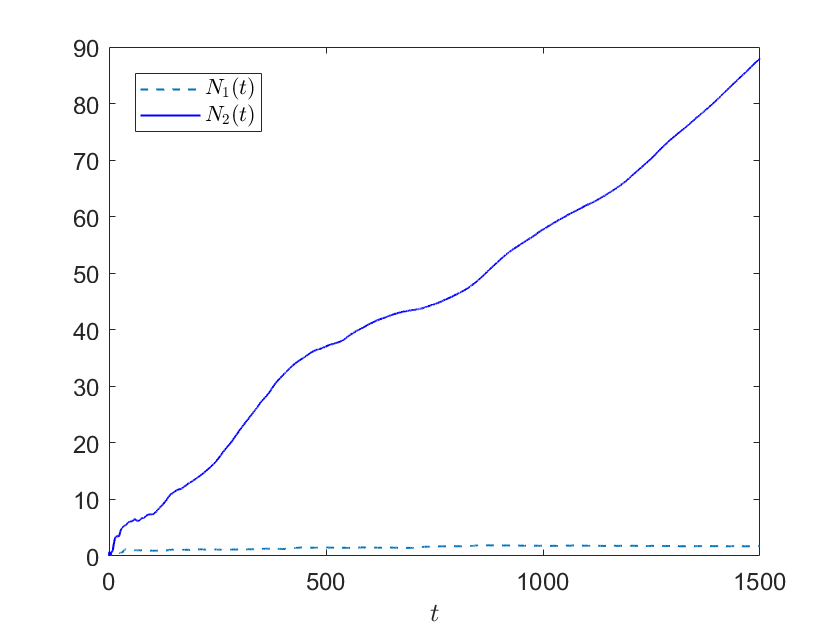}
\caption{Orbit dynamics for case $\rho<1$, $\rho_{1}<1$, $\rho_{2}>1$, $f_{2}<0$ (left), for the case $\rho<1$, $\rho_{1}<1$, $\rho_{2}>1$, $f_{2}>0$ (right).}
\label{ssst3}
\end{figure}

Therefore, simulations experiments indicate that the stability criteria given in Lemma \ref{stab} constitute stability criteria for the GJSOQ system. A formal justification of this result is postponed in a future work.
\section{Main results}\label{main}
\subsection{Geometric decay in the minimum direction}\label{taildecay}
In the following, we investigated the tail asymptotic of the stationary distribution of our model provided that $\rho_{1}<1$, $\rho_{2}<1$, $\rho<1$. We show that the tail asymptotic for the minimum orbit queue lengths for a fixed value of their difference, and server's state is exactly geometric. Moreover, we show that its decay rate is the square of the decay rate, i.e., $\rho^{2}$, of the corresponding modified single orbit queue, say \textit{the reference system}, which is formally described at the end of this section.

Clearly, when $\lambda_{1}=\lambda_{2}=0$, both orbit queues are well balanced by the \textit{smart} stream, and we again expect that the decay rate of the minimum of orbit queues equals $\rho^{2}$. Note that this result is motivated by the standard Markovian JSQ model (without retrials), in which the tail decay rate of the minimum of queue lengths is the square of the tail decay rate of the M/M/2. Theorem \ref{decay} states that our model has a similar behaviour. 

Tail asymptotic properties for random walks in the half plane are
still very limited including a recent paper by \cite{miya2sided}, in which the
method developed in \cite{miyatail} was extended for studying the standard generalized join the-shortest-queue model. Our goal here is to cope with the tail asymptotic properties of a Markov modulated random walk in the half-plane, with particular interest on a modulation that results in a tractable analysis. 

Our approach is summarized in the following steps:
\begin{enumerate}
\item[Step 1] We focus on the uniformized discrete time Markov chain $\{X(n);n\geq 0\}$, and perform the transformation $Z_{1,n}=min\{N_{1,n},N_{2,n}\}$, $Z_{2,n}=N_{2,n}-N_{1,n}$.
\item[Step 2] The discrete time Markov chain $Z(n)=\{(Z_{1,n},Z_{2,n},C_{n});n\geq 0\}$ is a Markov modulated random walk in the half plane with transition diagram as shown in Figure \ref{struc}, where the matrices $A_{i,j}^{(m)}$, $i,j=0,\pm1$, $m=r_{1},r_{2},d,v,h,0$ are given in Section \ref{model}. Then, we consider the censored Markov chain of $\{Z(n);n\geq 0\}$ on the busy states, which can be expressed explicitly. This censored Markov chain is a random walk in the half plane.
\item[Step 3] We investigate the tail asymptotic behaviour of this censored Markov chain at the busy states following \cite{limiya}, and show that the minimum orbit queue length decays geometrically; see Theorem \ref{decay}.
\item[Step 4] By exploiting a relation between idle and busy states (see \eqref{clop}), we further show that the minimum orbit queue length at the idle states decays also geometrically; see Corollary \ref{corol}.
\end{enumerate}
\begin{figure}[htp]
\centering
\includegraphics[scale=0.8]{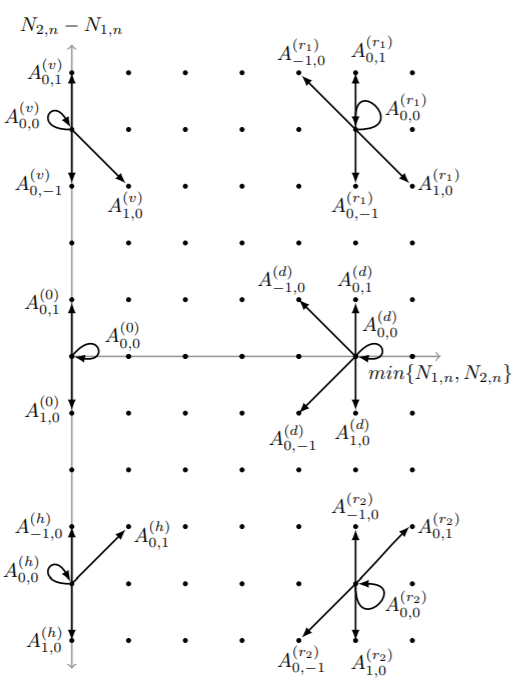}
\caption{The matrix transition diagram of $\{Z(n);n\geq 0\}$.}
\label{struc}
\end{figure}
Note that the presence of dedicated streams that are routed to the shortest orbit in case of a busy server complicates considerably the problem of decay rates. We are also interested in when the two orbit queues are
balanced. Intuitively, the two orbit queues being balanced when the smart arrival
stream, which routes the job to the least loaded orbit queue, keeps the two orbit queue lengths
close. Since the difference of the two orbit queues is the background state, the balancing of the two orbit queues is characterized
by how the background process behaves as the level process (i.e., the shortest orbit queue) goes up.

By applying a similar procedure as in Section \ref{stabo}, the censored Markov chain on the busy states of $\{Z(n);n\geq 0\}$, is a two-dimensional random walk on the half plane $\{(m,l)\in\mathbb{Z}^{2};l\geq 0\}$, say $Y(n)=\{(Y_{1,n},Y_{2,n},1);n\geq 0\}$, where $Y_{1,n}=N_{2,n}-N_{1,n}$, $Y_{2,n}=min\{N_{1,n},N_{2,n}\}$. Its transition rate diagram is given in Figure \ref{tr1}, where its one step transition probabilities are:
\begin{equation}
\begin{array}{c}
q_{0,-1}^{(-)}=\lambda_{1},\,q_{-1,-1}^{(-)}=\frac{\widehat{\mu}_{2}}{\lambda+\alpha_{1}+\alpha_{2}},\,q_{0,1}^{(-)}=\frac{\widehat{\mu}_{1}}{\lambda+\alpha_{1}+\alpha_{2}},\,q_{1,1}^{(-)}=\lambda_{0}+\lambda_{2},\,q_{0,0}^{(-)}=\alpha_{1}+\alpha_{2}+\frac{\lambda\mu}{\lambda+\alpha_{1}+\alpha_{2}},\\
q_{0,1}^{(+)}=\lambda_{2},\,q_{0,-1}^{(+)}=\frac{\widehat{\mu}_{2}}{\lambda+\alpha_{1}+\alpha_{2}},\,q_{-1,1}^{(+)}=\frac{\widehat{\mu}_{1}}{\lambda+\alpha_{1}+\alpha_{2}},\,q_{1,-1}^{(+)}=\lambda_{0}+\lambda_{1},\,q_{0,0}^{(+)}=\alpha_{1}+\alpha_{2}+\frac{\lambda\mu}{\lambda+\alpha_{1}+\alpha_{2}},\\
q_{0,1}^{(2)}=\lambda_{2}+\frac{\lambda_{0}}{2},\,q_{-1,-1}^{(2)}=\frac{\widehat{\mu}_{2}}{\lambda+\alpha_{1}+\alpha_{2}},\,q_{-1,1}^{(2)}=\frac{\widehat{\mu}_{1}}{\lambda+\alpha_{1}+\alpha_{2}},\,q_{0,-1}^{(2)}=\frac{\lambda_{0}}{2}+\lambda_{1},\,q_{0,0}^{(2)}=\alpha_{1}+\alpha_{2}+\frac{\lambda\mu}{\lambda+\alpha_{1}+\alpha_{2}},\\
q_{0,-1}^{(1-)}=\lambda_{1},\,q_{0,1}^{(1-)}=\frac{\widehat{\mu}_{1}}{\lambda+\alpha_{1}},\,q_{1,1}^{(1-)}=\lambda_{0}+\lambda_{2},\,q_{0,0}^{(1-)}=\alpha_{1}+\alpha_{2}+\frac{\lambda\mu}{\lambda+\alpha_{1}},\\
q_{0,1}^{(1+)}=\lambda_{2},\,q_{0,-1}^{(1+)}=\frac{\widehat{\mu}_{2}}{\lambda+\alpha_{2}},\,q_{1,-1}^{(1+)}=\lambda_{0}+\lambda_{1},\,q_{0,0}^{(1+)}=\alpha_{1}+\alpha_{2}+\frac{\lambda\mu}{\lambda+\alpha_{2}},\\
q_{0,1}^{(0)}=\lambda_{2}+\frac{\lambda_{0}}{2},\,q_{0,0}^{(0)}=\alpha_{1}+\alpha_{2}+\mu
,\,q_{0,-1}^{(0)}=\lambda_{1}+\frac{\lambda_{0}}{2}.
\end{array}\label{one}
\end{equation}
\begin{figure}[H]
\centering
\includegraphics[scale=0.8]{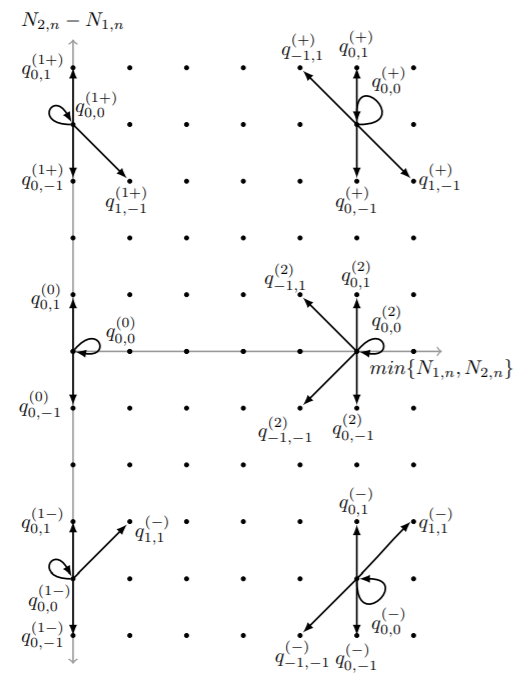}
\caption{Transition rate diagram of the censored chain $\{Y(n);n\geq 0\}$.}
\label{tr1}
\end{figure}
\begin{remark}
Note that $\{Y(n);n\geq 0\}$ can be also obtained by the censored Markov chain $\{\tilde{X}(n);n\geq 0\}$ (on the busy states of $\{X(n);n\geq 0\}$), by applying the transformation $Y_{1,n}=min\{N_{1,n},N_{2,n}\}$, $Y_{2,n}=N_{2,n}-N_{1,n}$.
\end{remark}
The main result is summarized in the following theorem. Let $\boldsymbol{\pi}\equiv\{\boldsymbol{\pi}_{m},m=0,1,\ldots\}$ be the stationary joint distribution of $min\{N_{1,n},N_{2,n}\}$ and $N_{2,n}-N_{1,n}$ when the server is busy, where $\boldsymbol{\pi}_{m}=\{\pi_{m,l}(1),l=0,\pm 1,\pm 2,\ldots\}$ is the subvector for level $m$, i.e.,
\begin{displaymath}
\begin{array}{c}
\pi_{m,l}(1)=\lim_{n\to\infty}\mathbb{P}(min\{N_{1,n},N_{2,n}\}=m,N_{2,n}-N_{1,n}=l,C_{n}=1),\,m\geq 0,l\in\mathbb{Z}.
\end{array}
\end{displaymath}
\begin{theorem}\label{decay}
For the generalized join the shortest orbit queue system satisfying $max\{\rho_{1},\rho_{2}\}<\rho<1$, and $\widehat{\lambda}_{0}>|\widehat{\lambda}_{2}-\widehat{\lambda}_{1}+\rho^{2}(\widehat{\mu}_{1}-\widehat{\mu}_{2})|$, the stationary probability vector $\boldsymbol{\pi}_{m}$ decays geometrically as $m\to\infty$ with decay rate $\rho^{2}$, i.e.,
\begin{equation}
\lim_{m\to\infty}\rho^{-2m}\pi_{m,l}(1)=x_{l}(\rho^{-2}),\label{bnmm}
\end{equation}
where
\begin{equation}
\begin{array}{rl}
x_{l}(\rho^{-2})=\left\{\begin{array}{ll}
\frac{\gamma_{2}+\widehat{\lambda}_{0}/2}{\gamma_{2}}x_{0}\left(\rho\frac{\gamma_{2}}{\gamma_{1}+\widehat{\lambda}_{0}}\right)^{|l|},&l\leq -1,\\
x_{0},&l=0,\\
\frac{\gamma_{1}+\widehat{\lambda}_{0}/2}{\gamma_{1}}x_{0}\left(\rho\frac{\gamma_{1}}{\gamma_{2}+\widehat{\lambda}_{0}}\right)^{l},&l\geq 1,
\end{array}\right.
\end{array}
\label{bnm}
\end{equation}
where $x_{0}$ is a constant and $\gamma_{1}=\widehat{\mu}_{1}\rho^{2}+\widehat{\lambda}_{2}$, $\gamma_{2}=\widehat{\mu}_{2}\rho^{2}+\widehat{\lambda}_{1}$. Moreover, the two orbit queues are \textit{strongly balanced} if and only if $\gamma_{2}<\rho(\gamma_{1}+\widehat{\lambda}_{0})$, $\gamma_{1}<\rho(\gamma_{2}+\widehat{\lambda}_{0})$.
\end{theorem}
\begin{proof}
The proof is based on a series of results, and is given in Section \ref{decayp}.
\end{proof}

Having known the exact tail asymptotic properties for $C(t) = 1$ (i.e., when the server is busy), we can investigate the tail asymptotic properties
for $C(t) = 0$ (i.e., when the server is idle) based on the relationship given in \eqref{clop} (which is similar to \eqref{e1}).

\begin{corollary}\label{corol}
Based on the relationship given in \eqref{clop},
\begin{equation}
\rho^{-2m}\pi_{m,l}(0)\sim \frac{\mu}{\lambda+\alpha_{1}+\alpha_{2}}x_{l}(\rho^{-2}),\,l\in\mathbb{Z},\label{idle}
\end{equation}
where $x_{l}(\rho^{-2})$, $l\in\mathbb{Z}$ as given in Theorem \ref{decay}.
\end{corollary} 
\begin{proof}
We consider the Markov chain $Z(n)=\{(Z_{1,n},Z_{2,n},C_{n});n\geq 0\}$ where $Z_{1,n}=min\{N_{1,n},N_{2,n}\}$, $Z_{2,n}=N_{2,n}-N_{1,n}$; see Figure \ref{struc}. Knowing the asymptotic properties of $\pi_{m,l}(1)$, we will investigate those of $\pi_{m,l}(0)=\lim_{n\to\infty}\mathbb{P}(min\{N_{1,n},N_{2,n}=m\},N_{2,n}-N_{1,n}=l,C_{n}=0)$. Note that the equilibrium equations that deal with the idle states are for $m>0$,
\begin{equation}
\pi_{m,l}(0)(\lambda+\alpha_{1}+\alpha_{2})=\mu\pi_{m,l}(1),\,l\in\mathbb{Z},\label{clop}
\end{equation}
which obviously leads to \eqref{idle}.
\end{proof}

Therefore, we shown that under the conditions $max\{\rho_{1},\rho_{2}\}<\rho<1$, $|\gamma_{1}-\gamma_{2}|<\widehat{\lambda}_{0}$, the tail decay rate for the shortest orbit queue in our GJSOQ system equals $\rho^{2}$. It is easy to see that this is the square of the tail decay rate for the orbit queue length of the \textit{reference system} mentioned at the beginning of this section. The \textit{reference system} operates as follows: Jobs arrive according to a Poisson process with rate $\lambda=\lambda_{0}+\lambda_{1}+\lambda_{2}$, and if they find the server available, they occupy it and get served after an exponentially distributed time with rate $\mu$. Otherwise they enter an infinite capacity orbit queue. Orbiting jobs retry to access the server according to the following rule: If there is only one job in orbit it retries after exponentially distributed time with rate $\alpha_{1}$ (note that this is not a restriction and whatever the retrial rate is in this case does not affect the final result). If there are more than one orbiting jobs, the first job in orbit queue retries with rate $\alpha_{1}+\alpha_{2}$ (This situation is similar with the case where there is a basic server in orbit queue, which transmit jobs to the service station at a rate $\alpha_{1}$, and in case there are more than one orbiting jobs, an additional server, which transmits jobs at a rate $\alpha_{2}$ helps the former one). Details on the stationary orbit queue length distribution of the reference system as well as its decay rate are given in Appendix \ref{app0}.

\begin{remark}
Another point of interest refers to the determination of the decay rate of the marginal distribution, since Theorem \ref{decay} refers to the decay rate of the stationary minimum orbit queue length distribution for a fixed value of the orbit queue difference, and server state. Clearly, the states of the server can be aggregated, since they are finite. However, the aggregation on the difference of queue sizes is not a trivial task and extra conditions maybe are needed. One may expect that
\begin{equation}
\begin{array}{c}
\lim_{m\to\infty}\rho^{-2m}\mathbb{P}(min\{N_{1},N_{2}\}=m)=\frac{1}{1-\mu}\sum_{l\in\mathbb{Z}}x_{l}(\rho^{-2}),
\end{array}\label{mlm}
\end{equation}
where $\{(N_{1},N_{2})\}$, the stationary version of $\{(N_{1,n},N_{2,n});n\geq 0\}$. This can be obtained using Theorem \ref{decay} and summing \eqref{bnmm}, \eqref{idle} over the difference of the two orbit queues. Note here that the sum in the right hand side of \eqref{mlm} definitely converges under the \textit{strongly balanced} condition (see Theorem \ref{decay}), since in such a case $\frac{\rho\gamma_{2}}{\widehat{\lambda}_{0}+\gamma_{1}}<\rho^{2}<1$, $\frac{\rho\gamma_{1}}{\widehat{\lambda}_{0}+\gamma_{2}}<\rho^{2}<1$. More importantly, under the conditions of Theorem 1, we always have $\frac{\rho\gamma_{1}}{\gamma_{2}+\widehat{\lambda}_{0}}<1$, $\frac{\rho\gamma_{1}}{\gamma_{2}+\widehat{\lambda}_{0}}<1$. Indeed, under conditions of Theorem 1, $\rho<1$, and $-\widehat{\lambda}_{0}<\gamma_{1}-\gamma_{2}<\widehat{\lambda}_{0}$. Thus, $\frac{\rho\gamma_{1}}{\gamma_{2}+\widehat{\lambda}_{0}}<1\Leftrightarrow\rho\gamma_{1}-\gamma_{2}<\widehat{\lambda}_{0}$, which is true since $\rho\gamma_{1}-\gamma_{2}<\gamma_{1}-\gamma_{2}<\widehat{\lambda}_{0}$. Similarly, $\frac{\rho\gamma_{2}}{\gamma_{1}+\widehat{\lambda}_{0}}<1\Leftrightarrow\rho\gamma_{2}-\gamma_{1}<\widehat{\lambda}_{0}$, which is true since $\rho\gamma_{2}-\gamma_{1}<\gamma_{2}-\gamma_{1}<\widehat{\lambda}_{0}$.
 The major problem to be resolved comes from the fact that since the difference of the two orbit queues is unbounded, in order to obtain the left hand side of \eqref{mlm}, it requires to verify the exchange of the summation and the limit. It seems that following the lines in \cite[Theorem 1.5]{miya2sided} by using a different framework based on \cite{miyatail}, and focusing on the \textit{weak} decay rates we can show that the marginal distribution of the $min\{N_{1},N_{2}\}$, under the same assumptions as given in Theorem \ref{decay}, also decays with rate $\rho^{2}$. 
 
 Moreover, for the standard GJSQ the authors in \cite{limiya} also focused on the decay rate for the marginal probabilities of the $\min\{Q_{1},Q_{2}\}$; see Corollary 3.2.1, Theorem 2.2.1 and Corollary 2.2.1 in \cite{limiya}. Following their lines, under the conditions of Theorem \ref{decay}, let
\begin{displaymath}
\lim_{l\to\infty}\frac{\pi_{0,l,1}}{x_{l}(\rho^{-2})}=s,
\end{displaymath}
where $\pi_{0,l,1}$, $l\geq 0$ is the stationary probability of an empty orbit queue for the censored Markov chain $\{Y(n);n\geq 0\}$, and $x_{l}(\rho^{-2})$ as given in \eqref{bnm}. Following \cite[Corollary 2.2.1]{limiya}, if $0<s<\infty$, the marginal distribution of $\min\{Q_{1},Q_{2}\}$ for the busy states decays geometrically with rate $\rho^{2}$. Now, using the relation in \eqref{clop}, the marginal distribution of $\min\{Q_{1},Q_{2}\}$ for the idle states decays also geometrically with the same rate. Thus, \eqref{mlm} has been proved. Therefore, the boundedness of $\pi_{0,l,1}/x_{l}(\rho^{-2})$ is very crucial in proving that the marginal probabilities of $\min\{Q_{1},Q_{2}\}$ decay geometrically with rate $\rho^{2}$. We postpone this verification in a future work. 
\end{remark}
\begin{remark}
In \cite{sakuma}, the authors proved that the tail decay rate of the stationary minimum queue length in the standard symmetric JSQ system fed by PH arrivals (no dedicated traffic) is the square of the decay rate for the queue length in the standard PH/M/2. Such a system was described by a Markov-modulated two-dimensional random walk in the quarter plane. They proved this result based on matrix analytic approach and standard results on the decay rates. We believe that it is possible to extend their approach to cope with Markov-modulated two-dimensional random walk in the half plane, and get the same result as in Theorem \ref{decay}. 
\end{remark}
\subsection{A heuristic approach for stationary approximations}\label{heur}
Our aim here is to develop a scheme to obtain approximations for the joint orbit queue length distribution for each server state, valid when one of the queue lengths is large. The main results are summarized in Lemmas \ref{sym}, and \ref{asym}, where we distinguish the analysis to the symmetric and the asymmetric case, respectively. With the term symmetric, we mean that $\lambda_{1}=\lambda_{2}\equiv\lambda_{+}$, $\alpha_{1}=\alpha_{2}\equiv \alpha$. 

We have to note that although the proofs of Lemmas \ref{sym}, and \ref{asym} (see subsections \ref{prsym}, \ref{prasym}, respectively) are similar, the asymmetric case reveals additional difficulties in the sense that we need further conditions to ensure that the approximated probabilities are well defined.

The main result for the symmetric case is the following.\begin{lemma}\label{sym}
For $\widehat{\lambda}<2\widehat{\mu}$,
\begin{equation}
\begin{array}{rl}
p_{i,j}(1)\sim& c\left\{\begin{array}{ll}
\left(\frac{\widehat{\lambda}(\widehat{\lambda}^{2}+4\widehat{\mu}(\widehat{\lambda}_{0}+\widehat{\lambda}_{+}))}{2\widehat{\mu}(\widehat{\lambda}^{2}+4\widehat{\mu}\widehat{\lambda}_{+})}\right)^{i}\left(\frac{\widehat{\lambda}(\widehat{\lambda}^{2}+4\widehat{\lambda}_{+}\widehat{\mu})}{2\widehat{\mu}(\widehat{\lambda}^{2}+4(\widehat{\lambda}_{0}+\widehat{\lambda}_{+})\widehat{\mu})}\right)^{j}[1+A_{-}\left(\frac{4\widehat{\mu}(\widehat{\lambda}_{0}+\widehat{\lambda}_{+})(\widehat{\lambda}^{2}+4\widehat{\mu}\widehat{\lambda}_{+})^{2}}{\widehat{\lambda}^{2}(\widehat{\lambda}^{2}+4\widehat{\mu}(\widehat{\lambda}_{0}+\widehat{\lambda}_{+}))^{2}}\right)^{i}],&i<j,\,j\gg 1,\vspace{2mm}\\
\left(\frac{\widehat{\lambda}}{2\widehat{\mu}}\right)^{i+j}\left(\frac{\widehat{\lambda}^{2}+4\widehat{\lambda}_{+}\widehat{\mu}}{\widehat{\lambda}(\widehat{\lambda}+2\widehat{\mu})}\right),&i=j\gg 1,\vspace{2mm}\\
\left(\frac{\widehat{\lambda}(\widehat{\lambda}^{2}+4\widehat{\mu}(\widehat{\lambda}_{0}+\widehat{\lambda}_{+}))}{2\widehat{\mu}(\widehat{\lambda}^{2}+4\widehat{\mu}\widehat{\lambda}_{+})}\right)^{j}\left(\frac{\widehat{\lambda}(\widehat{\lambda}^{2}+4\widehat{\lambda}_{+}\widehat{\mu})}{2\widehat{\mu}(\widehat{\lambda}^{2}+4(\widehat{\lambda}_{0}+\widehat{\lambda}_{+})\widehat{\mu}}\right)^{i}[1+A_{-}\left(\frac{4\widehat{\mu}(\widehat{\lambda}_{0}+\widehat{\lambda}_{+})(\widehat{\lambda}^{2}+4\widehat{\mu}\widehat{\lambda}_{+})^{2}}{\widehat{\lambda}^{2}(\widehat{\lambda}^{2}+4\widehat{\mu}(\widehat{\lambda}_{0}+\widehat{\lambda}_{+}))^{2}}\right)^{j}],&j<i,\,i\gg 1,
\end{array}\right.\vspace{2mm}\\
p_{i,j}(0)=&\frac{\mu}{\lambda+\alpha (1_{\{i\geq 1\}}+1_{\{j\geq 1\}})}p_{i,j}(1),
\end{array}
\end{equation}
where $c$ be the normalization constant, and
\begin{displaymath}
\begin{array}{lcr}
A_{-}=\frac{\lambda(\lambda+\alpha)+\widehat{\mu}(1-x_{-}(1+\frac{\lambda+\alpha}{\lambda+2\alpha}))+\frac{\lambda_{+}(\lambda+\alpha)}{x_{+}}}{\widehat{\mu}(\frac{\lambda+\alpha}{\lambda+2\alpha}x_{-}+x_{+}-1)-\lambda(\lambda+\alpha)-\frac{\lambda_{+}(\lambda+\alpha)}{x_{+}}},&
x_{+}=\frac{\widehat{\lambda}(\widehat{\lambda}^{2}+4\widehat{\mu}(\widehat{\lambda}_{0}+\widehat{\lambda}_{+}))}{2\widehat{\mu}(\widehat{\lambda}^{2}+4\widehat{\mu}\widehat{\lambda}_{+})},&x_{-}=\frac{\widehat{\lambda}_{0}+\widehat{\lambda}_{+}}{\widehat{\mu}x_{+}}.
\end{array}
\end{displaymath}
\end{lemma}
\begin{proof}
The proof is given in Section \ref{prsym}.
\end{proof}
The main result for the asymmetric case is the following.
\begin{lemma}\label{asym}
When, $\rho<1$, and $\widehat{\lambda}_{0}>|\rho^{2}(\widehat{\mu}_{2}-\widehat{\mu}_{1})+\widehat{\lambda}_{1}-\widehat{\lambda}_{2}|$,
\begin{equation}
\begin{array}{rl}
p_{i,j}(1)\sim c&\left\{\begin{array}{ll}
\sqrt{\frac{\epsilon_{-}}{\epsilon_{+}}}\left(\frac{\rho(\widehat{\lambda}_{2}+\rho^{2}\widehat{\mu}_{1})}{\widehat{\lambda}_{0}+\widehat{\lambda}_{1}+\rho^{2}\widehat{\mu}_{2}}\right)^{i}\left(\frac{\rho(\widehat{\lambda}_{0}+\widehat{\lambda}_{1}+\rho^{2}\widehat{\mu}_{2})}{\widehat{\lambda}_{2}+\rho^{2}\widehat{\mu}_{1}}\right)^{j}[1+A_{-}\left(\frac{x_{-}}{x_{+}}\right)^{i}],&j>i,\,j\gg 1,\vspace{2mm}\\
\rho^{i+j} [\frac{\widehat{\lambda}_{1}+\rho^{2}\widehat{\mu}_{2}}{\widehat{\lambda}(1+\rho)}\sqrt{\frac{\epsilon_{+}}{\epsilon_{-}}}+\frac{\widehat{\lambda}_{2}+\rho^{2}\widehat{\mu}_{1}}{\widehat{\lambda}(1+\rho)}\sqrt{\frac{\epsilon_{-}}{\epsilon_{+}}}],&i=j\gg 1,\\
\sqrt{\frac{\epsilon_{+}}{\epsilon_{-}}}\left(\frac{\rho(\widehat{\lambda}_{1}+\rho^{2}\widehat{\mu}_{2})}{\widehat{\lambda}_{0}+\widehat{\lambda}_{2}+\rho^{2}\widehat{\mu}_{1}}\right)^{i}\left(\frac{\rho(\widehat{\lambda}_{0}+\widehat{\lambda}_{2}+\rho^{2}\widehat{\mu}_{1})}{\widehat{\lambda}_{1}+\rho^{2}\widehat{\mu}_{2}}\right)^{j}[1+B_{-}\left(\frac{y_{-}}{y_{+}}\right)^{j}],&i>j,\,i\gg 1,
\end{array}\right.\\
p_{i,j}(0)=&\frac{\mu}{\lambda+\alpha_{1}1_{\{i>0\}}+\alpha_{2}1_{j>0}}p_{i,j}(1),
\end{array}\label{sol}
\end{equation}
where $c$ be the normalization constant, $\epsilon_{\pm}$ are given in \eqref{epsi}, and
\begin{displaymath}
\begin{array}{rl}
x_{+}=\frac{\rho(\widehat{\lambda}_{0}+\widehat{\lambda}_{1}+\rho^{2}\widehat{\mu}_{2})}{\widehat{\lambda}_{2}+\rho^{2}\widehat{\mu}_{1}},&x_{-}=\frac{(\widehat{\lambda}_{0}+\widehat{\lambda}_{1})(\widehat{\lambda}_{2}+\rho^{2}\widehat{\mu}_{1})}{\widehat{\mu}_{1}\rho(\widehat{\lambda}_{0}+\widehat{\lambda}_{1}+\rho^{2}\widehat{\mu}_{2})},\\
y_{-}=\frac{(\widehat{\lambda}_{0}+\widehat{\lambda}_{2})(\widehat{\lambda}_{1}+\rho^{2}\widehat{\mu}_{2})}{\widehat{\mu}_{2}\rho(\widehat{\lambda}_{0}+\widehat{\lambda}_{2}+\rho^{2}\widehat{\mu}_{1})},&y_{+}=\frac{\rho(\widehat{\lambda}_{0}+\widehat{\lambda}_{2}+\rho^{2}\widehat{\mu}_{1})}{\widehat{\lambda}_{1}+\rho^{2}\widehat{\mu}_{2}}\vspace{2mm}\\
A_{-}=&\frac{x_{+}^{2}(\widehat{\mu}_{1}\frac{\lambda+\alpha_{2}}{\lambda+\alpha_{1}+\alpha_{2}}+\frac{\lambda_{2}(\lambda+\alpha_{2})}{\rho^{2}})-x_{+}(\lambda(\lambda+\alpha_{2})+\widehat{\mu}_{2})+\rho^{2}\widehat{\mu}_{2}}{x_{+}(\lambda(\lambda+\alpha_{2})+\widehat{\mu}_{2})-[(\lambda_{0}+\lambda_{1})(\lambda+\alpha_{2})+\rho^{2}\widehat{\mu}_{2}]-\frac{\lambda_{2}(\lambda+\alpha_{2})}{\rho^{2}}x_{+}^{2}},\vspace{2mm}\\
B_{-}=&\frac{y_{+}^{2}(\widehat{\mu}_{2}\frac{\lambda+\alpha_{1}}{\lambda+\alpha_{1}+\alpha_{2}}+\frac{\lambda_{1}(\lambda+\alpha_{1})}{\rho^{2}})-y_{+}(\lambda(\lambda+\alpha_{1})+\widehat{\mu}_{1})+\rho^{2}\widehat{\mu}_{1}}{y_{+}(\lambda(\lambda+\alpha_{1})+\widehat{\mu}_{1})-[(\lambda_{0}+\lambda_{2})(\lambda+\alpha_{1})+\rho^{2}\widehat{\mu}_{1}]-\frac{\lambda_{1}(\lambda+\alpha_{1})}{\rho^{2}}y_{+}^{2}}.
\end{array}
\end{displaymath}
\end{lemma}
\begin{proof}
The proof is given in Section \ref{prasym}.
\end{proof}
\begin{remark}
Note that in proving Lemma \ref{asym}, we need to have the condition $\widehat{\lambda}_{0}>|\rho^{2}(\widehat{\mu}_{2}-\widehat{\mu}_{1})+\widehat{\lambda}_{1}-\widehat{\lambda}_{2}|$. This condition was also needed in Theorem \ref{decay}, and corresponds to the so-called \textit{strongly pooled case} in \cite{foley}. This condition refers to the case where the proportion of smart jobs is large compared with the dedicated traffic, and results in balancing the orbit queue lengths. In such a case the orbit queues overload in
 tandem; see Figure \ref{fghm} (left). More importantly, by technical point of view, this condition is crucial in proving that the approximated join orbit queue length probabilities are well defined as given in Lemma \ref{asym}; see subsection \ref{prasym}.

 In case $\widehat{\lambda}_{0}\leq |\rho^{2}(\widehat{\mu}_{2}-\widehat{\mu}_{1})+\widehat{\lambda}_{1}-\widehat{\lambda}_{2}|$, i.e., \textit{the weakly pooled case}, there is a spectrum of possible drift directions, depending on how large is the proportion of the dedicated traffic. Moreover, the approximated joint orbit queue length distribution cannot any more written in such an elegant form and is not well defined. An illustration of both scenarios are given in Figure \ref{fghm}, where we performed some simulations. Note that in the strongly pooled case, orbit queue lengths grow in tandem as expected; Figure \ref{fghm} (left). In Figure \ref{fghm} (right), $\widehat{\lambda}_{0}< |\rho^{2}(\widehat{\mu}_{2}-\widehat{\mu}_{1})+\widehat{\lambda}_{1}-\widehat{\lambda}_{2}|$, and $\lambda_{1}>\lambda_{0}>\lambda_{2}$, $\alpha_{1}>\alpha_{2}$, we observe that $N_{2}(t)$ lags behind $N_{1}(t)$.  \begin{figure}[H]
\centering
\includegraphics[scale=0.4]{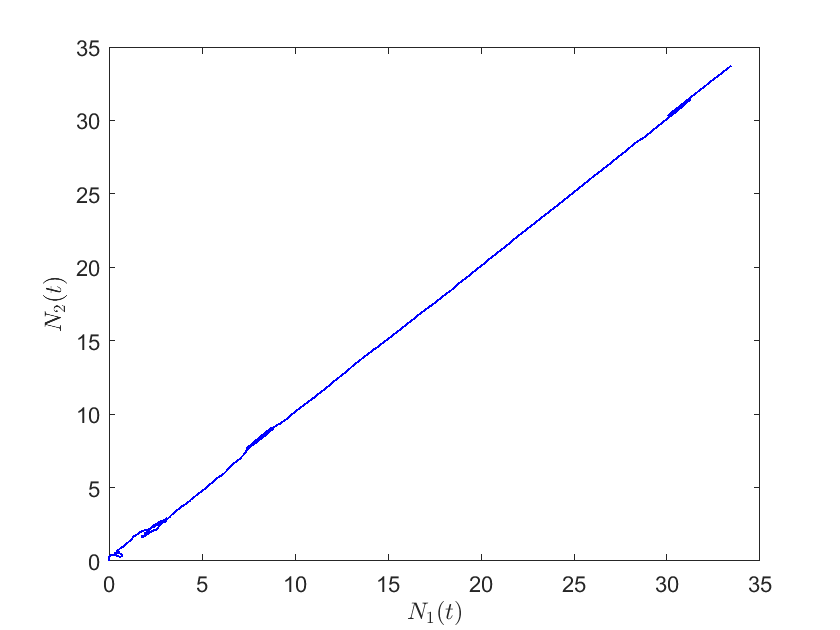}\includegraphics[scale=0.4]{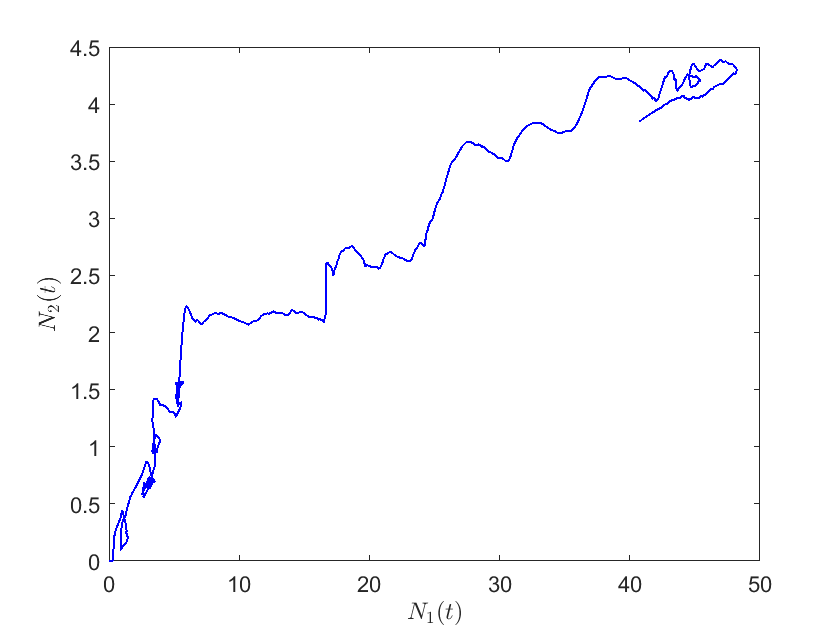}
\caption{Orbit dynamics when $\rho<1$, $\alpha_{1}>\alpha_{2}$, and $\widehat{\lambda}_{0}> |\rho^{2}(\widehat{\mu}_{2}-\widehat{\mu}_{1})+\widehat{\lambda}_{1}-\widehat{\lambda}_{2}|$ (left), and when $\widehat{\lambda}_{0}< |\rho^{2}(\widehat{\mu}_{2}-\widehat{\mu}_{1})+\widehat{\lambda}_{1}-\widehat{\lambda}_{2}|$, and $\lambda_{1}>\lambda_{0}>\lambda_{2}$ (right).}
\label{fghm}
\end{figure}
\end{remark}
\section{Proofs of main results}\label{prmain}
To improve the readability of the paper we group in this section the proofs of the main results presented in section \ref{main}.
\subsection{Proof of Theorem \ref{decay}}\label{decayp}
The proof follows the lines in \cite{limiya}, and a series of results are needed to be proven. Note that $\{Y(n);n\geq 0\}$ is a discrete time QBD process with the following block structured transition matrix 
\begin{displaymath}
P_{Y}=\begin{pmatrix}
L_{0}&A_{1}&&&\\
A_{-1}&A_{0}&A_{1}&&\\
&A_{-1}&A_{0}&A_{1}&\\
&&A_{-1}&A_{0}&A_{1}\\
&&&\ddots&\ddots&\ddots
\end{pmatrix},
\end{displaymath}
where the infinite dimensional matrices $L$, $A_{i}$, $i=0,\pm 1$ are
\begin{displaymath}
\begin{array}{rl}
(L)_{l,l^{\prime}}=&\mathbb{P}(Y(n+1)=(0,l^{\prime},1)|Y(n)=(0,l,1)),\,l,l^{\prime}\in\mathbb{Z},\\
(A_{i})_{l,l^{\prime}}=&\mathbb{P}(Y(n+1)=(m+i,l^{\prime},1)|Y(n)=(m,l,1)),\,l,l^{\prime}\in\mathbb{Z},\,m\geq 1,
\end{array}
\end{displaymath}
with their elements (i.e., the one step transition probabilities) are given in \eqref{one}. Under the stability condition, the stationary distribution exists. Denote it by the row vector $\boldsymbol{\pi}=(\boldsymbol{\pi_{0}},\boldsymbol{\pi_{1}},\ldots)$. It is well known that the stationary distribution has the following matrix geometric form:
\begin{equation}
\boldsymbol{\pi_{m}}=\boldsymbol{\pi_{0}}R^{m},\,m\geq 1,\label{solma}
\end{equation}
where $R$ is the minimal nonnegative solution of the equation
\begin{displaymath}
R=A_{1}+RA_{0}+R^{2}A_{-1}.
\end{displaymath}
Since the size of $R$ is infinite, conditions for geometric tail decay rate of \eqref{solma} are given in \cite{limiya}:
\begin{proposition}
(Theorem 2.1 in \cite{limiya}) Let $A\equiv A_{-1}+A_{0}+A_{1}$ is irreducible and aperiodic, and assume that the Markov additive process generated by $\{A_{i},i=0,\pm 1\}$ is 1-arithmetic. If there exists $z>1$ and positive vectors $\mathbf{x}$, $\mathbf{y}$ such that
\begin{eqnarray}
\mathbf{x}=\mathbf{x}A_{*}(z),\,\mathbf{y}=A_{*}(z)\mathbf{y},\label{s1}\\
\mathbf{x}\mathbf{y}<\infty,\label{s2}
\end{eqnarray}
where $A_{*}(w)=w^{-1}A_{-1}+A_{0}+wA_{1}$, $w\neq
0$. Then $R$ has left and right eigenvectors $\mathbf{x}$ and $\mathbf{r}\equiv (I-A_{0}-RA_{-1}-z^{-1}A_{-1})\mathbf{y}$, respectively, with eigenvalue $z^{-1}$. Moreover, if
\begin{equation}
\boldsymbol{\pi}_{0}\mathbf{y}<\infty,\label{s3}
\end{equation}
then
\begin{eqnarray}
\lim_{m\to\infty}z^{n}\boldsymbol{\pi}_{m}=\frac{\boldsymbol{\pi}_{0}\mathbf{r}}{\mathbf{x}\mathbf{r}}\mathbf{x}.\label{geomd}
\end{eqnarray}
\end{proposition}
\begin{remark}
It is easy to see that for our queueing model both the aperiodicity and the irreducibility of $A$ are easily verified. Moreover,the 1-arithmetic property of the Markov
additive process generated by $\{A_{i};i=0,\pm 1\}$ is readily satisfied
by its transition structure.
\end{remark}
For our case the matrix $A_{*}(z)$ is given by
\begin{displaymath}
A_{*}(z)=\bordermatrix{ \text{}&\ldots&-3
&-2&-1&0&1&2&3&\ldots\cr \vdots&\ddots & \ddots &
\ddots & &&&&&\cr
-2&& a_{-1} & a_{0} & a_{1} &&&&&\cr -1 & &&a_{-1}
&a_{0}&a_{1}&&&&\cr 0 &&&&c_{-1}
& c_0&c_{1}& &&\cr 1&&&&&b_{-1}&b_{0}&b_{1}&&\cr 2&&&&&&b_{-1}&b_{0}&b_{1}&\cr \vdots&&&&&&&\ddots&\ddots&\ddots},
\end{displaymath}
where
\begin{displaymath}
\begin{array}{c}
a_{-1}=\frac{\widehat{\mu}_{2}}{\lambda+\alpha_{1}+\alpha_{2}}z^{-1}+\lambda_{1},\,a_{0}=\alpha_{1}+\alpha_{2}+\frac{\lambda\mu}{\lambda+\alpha_{1}+\alpha_{2}},\,a_{1}=\frac{\widehat{\mu}_{1}}{\lambda+\alpha_{1}+\alpha_{2}}+(\lambda_{0}+\lambda_{2})z,\\
c_{-1}=\frac{\widehat{\mu}_{2}}{\lambda+\alpha_{1}+\alpha_{2}}z^{-1}+\lambda_{1}+\frac{\lambda_{0}}{2},\,c_{0}=a_{0},\,c_{1}=\frac{\widehat{\mu}_{1}}{\lambda+\alpha_{1}+\alpha_{2}}z^{-1}+\frac{\lambda_{0}}{2}+\lambda_{2},\\
b_{-1}=\frac{\widehat{\mu}_{2}}{\lambda+\alpha_{1}+\alpha_{2}}+(\lambda_{0}+\lambda_{1})z,\,b_{0}=a_{0},\,b_{1}=\frac{\widehat{\mu}_{1}}{\lambda+\alpha_{1}+\alpha_{2}}z^{-1}+\lambda_{2}.
\end{array}
\end{displaymath}
Denote the left and the right invariant vectors of $A_{*}(z)$, $z>1$ by
\begin{displaymath}
\begin{array}{c}
\mathbf{x}(z)=(\ldots,x_{-1}(z),x_{0},x_{1}(z),\ldots),\,\mathbf{y}(z)=(\ldots,y_{-1}(z),y_{0},y_{1}(z),\ldots)^{T}.
\end{array}
\end{displaymath}
We are left to examine conditions \eqref{s1}, \eqref{s2}, \eqref{s3}. To this end, let
\begin{displaymath}
\begin{array}{rl}
\beta_{i,j}(z)=&(\widehat{\lambda}+\widehat{\mu}_{1}+\widehat{\mu}_{2})^{2}-4(\widehat{\mu}_{i}+(\widehat{\lambda}_{0}+\widehat{\lambda}_{j})z)(\widehat{\mu}_{j}z^{-1}+\widehat{\lambda}_{i}),\,i,j=1,2,\\
f(z)=&\frac{\widehat{\mu}_{2}z^{-1}+\widehat{\lambda}_{1}+\frac{\widehat{\lambda}_{0}}{2}}{2(\widehat{\mu}_{2}z^{-1}+\widehat{\lambda}_{1})}(1-\frac{\sqrt{\beta_{1,2}(z)}}{\widehat{\lambda}+\widehat{\mu}_{1}+\widehat{\mu}_{2}})+\frac{\widehat{\mu}_{1}z^{-1}+\widehat{\lambda}_{2}+\frac{\widehat{\lambda}_{0}}{2}}{2(\widehat{\mu}_{1}z^{-1}+\widehat{\lambda}_{2})}(1-\frac{\sqrt{\beta_{2,1}(z)}}{\widehat{\lambda}+\widehat{\mu}_{1}+\widehat{\mu}_{2}}).
\end{array}
\end{displaymath}
\begin{lemma}\label{kl1}
For the censored chain $\{Y(n);n\geq 0\}$ on the busy states of our GJSOQ, the conditions \eqref{s1}, \eqref{s2} hold if and only if there exists a $z>1$ such that 
\begin{eqnarray}
\beta_{1,2}(z)>0,\,\beta_{2,1}(z)>0,\label{ss1}
\end{eqnarray} 
and $f(z)=1$. In such a case, the left and right invariant vectors $\mathbf{x}(z)$ and $\mathbf{y}(z)$ of
$A_{*}(z)$ satisfying $\mathbf{x}(z)\mathbf{y}(z)<\infty$ have the following forms:
\begin{equation}
\begin{array}{rl}
x_{l}(z)=\left\{\begin{array}{ll}
(1+\frac{\widehat{\lambda}_{0}}{2(\widehat{\mu}_{2}z^{-1}+\widehat{\lambda}_{1})})x_{0}\eta_{min}^{|l|},&l\leq -1,\\
(1+\frac{\widehat{\lambda}_{0}}{2(\widehat{\mu}_{1}z^{-1}+\widehat{\lambda}_{2})})x_{0}\theta_{min}^{l},&l\geq 1,
\end{array}\right.,& y_{l}(z)=\left\{\begin{array}{ll}
y_{0}\eta_{max}^{-|l|},&l\leq -1,\\
y_{0}\theta_{max}^{-l},&l\geq 1,
\end{array}\right.
\end{array}\label{inva}
\end{equation}
where $\eta_{min}<\eta_{max}$ are the roots of the equation:
\begin{eqnarray}
\phi_{-}(\eta):=\eta^{2}(\widehat{\mu}_{1}+(\widehat{\lambda}_{0}+\widehat{\lambda}_{2})z)-\eta(\widehat{\lambda}+\widehat{\mu}_{1}+\widehat{\mu}_{2})+\widehat{\mu}_{2}z^{-1}+\widehat{\lambda}_{1}=0,\label{pol1}
\end{eqnarray}
and $\theta_{min}<\theta_{max}$ are the roots of the equation:
\begin{eqnarray}
\phi_{+}(\theta):=\theta^{2}(\widehat{\mu}_{2}+(\widehat{\lambda}_{0}+\widehat{\lambda}_{1})z)-\theta(\widehat{\lambda}+\widehat{\mu}_{1}+\widehat{\mu}_{2})+\widehat{\mu}_{1}z^{-1}+\widehat{\lambda}_{2}=0.\label{pol2}
\end{eqnarray}
\end{lemma}
\begin{proof}
See Appendix \ref{app2}.\end{proof}
   
\begin{remark}
Note that from \eqref{pol1}, \eqref{pol2},
\begin{displaymath}
\begin{array}{llcrl}
\eta_{min}=&\frac{\widehat{\lambda}+\widehat{\mu}_{1}+\widehat{\mu}_{2}-\sqrt{\beta_{1,2}(z)}}{2(\widehat{\mu}_{1}+(\widehat{\lambda}_{2}+\widehat{\lambda}_{0})z)},&\eta_{max}=&\frac{\widehat{\lambda}+\widehat{\mu}_{1}+\widehat{\mu}_{2}+\sqrt{\beta_{1,2}(z)}}{2(\widehat{\mu}_{1}+(\widehat{\lambda}_{2}+\widehat{\lambda}_{0})z)},\\
\theta_{min}=&\frac{\widehat{\lambda}+\widehat{\mu}_{1}+\widehat{\mu}_{2}-\sqrt{\beta_{2,1}(z)}}{2(\widehat{\mu}_{2}+(\widehat{\lambda}_{1}+\widehat{\lambda}_{0})z)},&\theta_{max}=&\frac{\widehat{\lambda}+\widehat{\mu}_{1}+\widehat{\mu}_{2}+\sqrt{\beta_{2,1}(z)}}{2(\widehat{\mu}_{2}+(\widehat{\lambda}_{1}+\widehat{\lambda}_{0})z)}.
\end{array}
\end{displaymath}
\end{remark}

We now turn our attention in the solution of the equation $f(z)=1$. To improve the readability, set $\gamma_{1}=\widehat{\mu}_{1}\rho^{2}+\widehat{\lambda}_{2}$, $\gamma_{2}=\widehat{\mu}_{2}\rho^{2}+\widehat{\lambda}_{1}$. Note that $\gamma_{1}+\gamma_{2}+\widehat{\lambda}_{0}=\rho(\widehat{\lambda}+\widehat{\mu}_{1}+\widehat{\mu}_{2})$.
\begin{lemma}\label{kl2}
$z=\rho^{-2}$ is the only solution of $f(z)=1$, $z>1$, and
\begin{eqnarray}
(\eta_{min},\eta_{max})=(\rho\frac{\gamma_{2}}{\gamma_{1}+\widehat{\lambda}_{0}},\rho)\text{ or }(\rho,\rho\frac{\gamma_{2}}{\gamma_{1}+\widehat{\lambda}_{0}}),\label{rot1}\\
(\theta_{min},\theta_{max})=(\rho\frac{\gamma_{1}}{\gamma_{2}+\widehat{\lambda}_{0}},\rho)\text{ or }(\rho,\rho\frac{\gamma_{1}}{\gamma_{2}+\widehat{\lambda}_{0}}).\label{rot2}
\end{eqnarray}
\end{lemma}\begin{proof}
See Appendix \ref{appro}.
\end{proof}

The results given in Lemmas \ref{kl1}, \ref{kl2} provide the basic ingredients for the proof of Theorem \ref{decay}. Note that based on Lemma \ref{kl2} there are four cases for the expression of $\mathbf{x}(\rho^{-2})$ based on the possible values of $\eta_{min}$, $\theta_{min}$. In particular,
\begin{displaymath}
\begin{array}{rl}
\begin{pmatrix}
\eta_{min}\\
\theta_{min}
\end{pmatrix}=&\begin{pmatrix}
\rho\frac{\gamma_{2}}{\gamma_{1}+\widehat{\lambda}_{0}}\\
\rho\frac{\gamma_{1}}{\gamma_{2}+\widehat{\lambda}_{0}}
\end{pmatrix}\text{ or }\begin{pmatrix}
\rho\\
\rho\frac{\gamma_{1}}{\gamma_{2}+\widehat{\lambda}_{0}}
\end{pmatrix}\text{ or }\begin{pmatrix}
\rho\frac{\gamma_{2}}{\gamma_{1}+\widehat{\lambda}_{0}}\\
\rho\end{pmatrix}\text{ or }\begin{pmatrix}
\rho\\
\rho\end{pmatrix}.
\end{array}
\end{displaymath}
The last one is rejected since corresponds to the case $\gamma_{1}+\widehat{\lambda}_{0}\leq \gamma_{2}$, $\gamma_{2}+\widehat{\lambda}_{0}\leq \gamma_{1}$, which is impossible since $\lambda_{0}>0$. It is easily seen that only the first pair satisfies $\mathbf{x}(\rho^{-2})=\mathbf{x}(\rho^{-2})A_{*}(\rho^{-2})$, which means that $A_{*}(\rho^{-2})$ is 1-positive. Note that in such a case $\rho\frac{\gamma_{2}}{\gamma_{1}+\widehat{\lambda}_{0}}<\rho\Rightarrow\gamma_{2}-\gamma_{1}<\widehat{\lambda}_{0}$, and $\rho\frac{\gamma_{1}}{\gamma_{2}+\widehat{\lambda}_{0}}<\rho\Rightarrow\gamma_{1}-\gamma_{2}<\widehat{\lambda}_{0}$, which is equivalent to $|\gamma_{1}-\gamma_{2}|<\widehat{\lambda}_{0}$ or to $|\widehat{\lambda}_{2}-\widehat{\lambda}_{1}+\rho^{2}(\mu_{1}-\mu_{2})|<\widehat{\lambda}_{0}$ (note that this case is equivalent to the so called \textit{strongly pooled case} in \cite{foley}).

Now it remains to verify the boundary condition \eqref{s3}. From Lemma \ref{kl2} we know that $\eta_{max}=\theta_{max}=\rho$. Thus, $\boldsymbol{\pi}_{0}\mathbf{y}<\infty$ implies
\begin{displaymath}
\sum_{l=0}^{\infty}\rho^{-l}(\mathbb{P}(N_{1}=l,N_{2}=0,C=1)+\mathbb{P}(N_{1}=0,N_{2}=l,C=1))<\infty,
\end{displaymath}
where $\{(N_{1},N_{2},C)\}$ is the stationary version of $\{X(n);n\geq 0\}$. The last inequality is the same as condition C.10 in \cite{foley}, which was proved under the corresponding conditions $max\{\rho_{1},\rho_{2}\}<\rho<1$, and $|\widehat{\lambda}_{2}-\widehat{\lambda}_{1}+\rho^{2}(\mu_{1}-\mu_{2})|<\widehat{\lambda}_{0}$ in \cite[Proposition 1]{foley}, following \cite[Theorem 14.3.7]{seanmeyn}, so further details are omitted. 

Following \cite[Definition 2.2]{limiya}, the two orbit queue are \textit{strongly balanced}, if the difference between orbit queue lengths (i.e., the background process) is \textit{dominated} by the level process, i.e., $min\{N_{1,n},N_{2,n}\}$. Intuitively, it means that the level process has more influence when the same scaling is applied to the level and background process. This is equivalent with the condition
\begin{displaymath}
\lim_{l\to+\infty}\rho^{-2l}x_{l}(\rho^{-2})=\lim_{l\to-\infty}\rho^{2l}x_{l}(\rho^{-2})=0.
\end{displaymath}
Noting \eqref{bnm}, it is easily realized that this is the case if and only if $\gamma_{2}<\rho(\gamma_{1}+\widehat{\lambda}_{0})$, $\gamma_{1}<\rho(\gamma_{2}+\widehat{\lambda}_{0})$.

\subsection{Proof of Lemma \ref{sym}}\label{prsym}
Starting with region I ($i\gg 1$, $j\gg 1$), equations \eqref{e1}, \eqref{e2} are the only relevant. By introducing the new variables $m=i+j$, $n=j-i$, and setting $p_{i,j}(k)\equiv R_{m,n}(k)$, \eqref{e1}, \eqref{e2} become 
\begin{eqnarray}
R_{m,n}(0)=\frac{\mu}{\lambda+2\alpha}R_{m,n}(1),\label{se1}\\
(\lambda+2\alpha)R_{m,n}(1)=[\lambda_{0}H(1-n)+\lambda_{+}]R_{m-1,n-1}(1)\nonumber\\+[\lambda_{0}H(n+1)+\lambda_{+}]R_{m-1,n+1}(1)
+\lambda R_{m,n}(0)+\alpha[R_{m+1,n-1}(0)+R_{m+1,n+1}(0)],\label{se2}
\end{eqnarray}
Substitute \eqref{se1} to \eqref{se2} yields
\begin{eqnarray}
(\widehat{\lambda}+2\widehat{\mu})R_{m,n}(1)=[\widehat{\lambda}_{0}H(1-n)+\widehat{\lambda}_{+}]R_{m-1,n-1}(1) \nonumber \\+[\widehat{\lambda}_{0}H(n+1)+\widehat{\lambda}_{+}]R_{m-1,n+1}(1)
+\widehat{\mu}[R_{m+1,n-1}(1)+R_{m,n}(1)],\label{sse2}
\end{eqnarray}
where $\widehat{\lambda}=\lambda(\lambda+2\alpha)$, $\lambda=\lambda_{0}+2\lambda_{+}$, $\widehat{\lambda}_{l}=\lambda_{l}(\lambda+2\alpha)$, $l=0,+$, $\widehat{\mu}=\alpha\mu$. Note that regarding variable $m$, \eqref{sse2} is second-order difference equation with constant coefficients, which admits solutions of the form
$R(m,n,1)=\gamma^{m}Q(n)$, where $\gamma$ a constant to be determined. Substituting in \eqref{sse2} we realize that $Q(n)$ satisfies
\begin{equation}
\begin{array}{rl}
(\widehat{\lambda}+2\widehat{\mu})Q(n)=&[\frac{\widehat{\lambda}_{0}H(1-n)+\widehat{\lambda}_{+}}{\gamma}+\widehat{\mu}\gamma]Q(n-1)+[\frac{\widehat{\lambda}_{0}H(1+n)+\widehat{\lambda}_{+}}{\gamma}+\widehat{\mu}\gamma]Q(n+1).
\end{array}\label{sse21}
\end{equation}
Note that symmetry implies that $Q(n)=Q(-n)$, and thus, it is sufficient to only consider $n\geq 0$. Then, for $n=0$, \eqref{sse2} gives
\begin{equation}
(\widehat{\lambda}+2\widehat{\mu})Q(0)=2[\frac{\widehat{\lambda}_{0}+\widehat{\lambda}_{+}}{\gamma}+\widehat{\mu}\gamma]Q(1).\label{q0}
\end{equation}
Setting $n=1$ in \eqref{sse21}, and using \eqref{q0} we obtain,
\begin{equation}
\begin{array}{rl}
[\frac{\widehat{\lambda}_{0}+\widehat{\lambda}_{+}}{\gamma}+\widehat{\mu}\gamma]Q(2)=&\left[\widehat{\lambda}+2\widehat{\mu}-\frac{2}{\widehat{\lambda}+2\widehat{\mu}}[\frac{\widehat{\lambda}_{0}+\widehat{\lambda}_{+}}{\gamma}+\widehat{\mu}\gamma][\frac{\widehat{\lambda}}{2\gamma}+\widehat{\mu}\gamma]\right]Q(1),
\end{array}\label{q2}
\end{equation}
while, for $n\geq 2$,
\begin{equation}
\begin{array}{rl}
(\widehat{\lambda}+2\widehat{\mu})Q(n)=&[\frac{\widehat{\lambda}_{+}}{\gamma}+\widehat{\mu}\gamma]Q(n-1)+[\frac{\widehat{\lambda}_{0}+\widehat{\lambda}_{+}}{\gamma}+\widehat{\mu}\gamma]Q(n+1).\label{q3}
\end{array}
\end{equation}
The form of \eqref{q3} implies that $Q(n)=c\delta^{n}$, $n\geq 1$. Thus, by substituting in \eqref{q2}, \eqref{q3} we obtain
\begin{eqnarray}
\widehat{\lambda}+2\widehat{\mu}=[\frac{\widehat{\lambda}_{0}+\widehat{\lambda}_{+}}{\gamma}+\widehat{\mu}\gamma]\delta+\frac{2}{\widehat{\lambda}+2\widehat{\mu}}[\frac{\widehat{\lambda}_{0}+\widehat{\lambda}_{+}}{\gamma}+\widehat{\mu}\gamma][\frac{\widehat{\lambda}}{2\gamma}+\widehat{\mu}\gamma],\label{t1}\\
\widehat{\lambda}+2\widehat{\mu}=[\frac{\widehat{\lambda}_{+}}{\gamma}+\widehat{\mu}\gamma]\frac{1}{\delta}+[\frac{\widehat{\lambda}_{0}+\widehat{\lambda}_{+}}{\gamma}+\widehat{\mu}\gamma]\delta.\label{t2}
\end{eqnarray}
Equations \eqref{t1}, \eqref{t2} constitute a pair of algebraic curves in $(\gamma,\delta)-plane$ with four intersection points. From these points the only possible candidates are: $(1,\frac{\widehat{\lambda}_{+}+\widehat{\mu}}{\widehat{\lambda}_{0}+\widehat{\lambda}_{+}+\widehat{\mu}})$, $(\frac{\widehat{\lambda}}{2\widehat{\mu}},\frac{\widehat{\lambda}^{2}+4\widehat{\lambda}_{+}\widehat{\mu}}{\widehat{\lambda}^{2}+4(\widehat{\lambda}_{0}+\widehat{\lambda}_{+})\widehat{\mu}})$. Note that $\gamma=1$ yields $\sum_{i,j\geq 0}p_{i,j}(1)=\infty$, i.e., corresponds to an un-normalizable solution for $p_{i,j}(1)$, thus the pair $(\frac{\widehat{\lambda}}{2\widehat{\mu}},\frac{\widehat{\lambda}^{2}+4\widehat{\lambda}_{+}\widehat{\mu}}{\widehat{\lambda}^{2}+4(\widehat{\lambda}_{0}+\widehat{\lambda}_{+})\widehat{\mu}})$, with $\widehat{\lambda}<2\widehat{\mu}$ corresponds to the feasible pair. Note that $\widehat{\lambda}<2\widehat{\mu}$ corresponds to the \textit{stability condition} for our system. Finally, from \eqref{q0} by substituting $(\gamma,\delta)=(\frac{\widehat{\lambda}}{2\widehat{\mu}},\frac{\widehat{\lambda}^{2}+4\widehat{\lambda}_{+}\widehat{\mu}}{\widehat{\lambda}^{2}+4(\widehat{\lambda}_{0}+\widehat{\lambda}_{+})\widehat{\mu}})$ yields
\begin{displaymath}
Q(0)=c\frac{\widehat{\lambda}^{2}+4\widehat{\lambda}_{+}\widehat{\mu}}{\widehat{\lambda}(\widehat{\lambda}+2\widehat{\mu})}.
\end{displaymath}

To summarize, the asymptotic expansions of $p_{i,j}(k)$, in region I ($i\gg 1$, $j\gg 1$) are
\begin{equation}
\begin{array}{rl}
p_{i,j}(1)\sim& c\left\{\begin{array}{ll}
\left(\frac{\widehat{\lambda}}{2\widehat{\mu}}\right)^{i+j}\left(\frac{\widehat{\lambda}^{2}+4\widehat{\lambda}_{+}\widehat{\mu}}{\widehat{\lambda}^{2}+4(\widehat{\lambda}_{0}+\widehat{\lambda}_{+})\widehat{\mu}})\right)^{|i-j|},&i\neq j,\\
\left(\frac{\widehat{\lambda}}{2\widehat{\mu}}\right)^{i+j}\left(\frac{\widehat{\lambda}^{2}+4\widehat{\lambda}_{+}\widehat{\mu}}{\widehat{\lambda}(\widehat{\lambda}+2\widehat{\mu})}\right),&i=j.
\end{array}\right.\vspace{2mm}\\
p_{i,j}(0)=&\frac{\mu}{\lambda+2\alpha}p_{i,j}(1),
\end{array}\label{r1}
\end{equation}
where the multiplicative constant $c$ is the normalization constant.

We proceed with region II for $i=0$ or 1, and $j\gg 1$. Equations \eqref{e1}, \eqref{e2} for $j>i+1$, and \eqref{e3} are the only relevant. Using these equations we have to solve
\begin{eqnarray}
(\lambda+\alpha 1_{\{i>0\}}+\alpha)p_{i,j}(0)=\mu p_{i,j}(1),\label{ee1}\vspace{2mm}\\
(\widehat{\lambda}+2\widehat{\mu})p_{i,j}(1)=\widehat{\lambda}_{+}p_{i,j-1}(1)+[\widehat{\lambda}_{0}+\widehat{\lambda}_{+}]p_{i-1,j}(1)+\widehat{\mu}[p_{i+1,j}(1)+p_{i,j+1}(1)],\label{ee2}\vspace{2mm}\\
(\lambda(\lambda+\alpha)+\widehat{\mu})p_{0,j}(1)=\frac{\widehat{\mu}(\lambda+\alpha)}{\lambda+2\alpha}p_{1,j}(1)+\widehat{\mu}p_{0,j+1}(1)+\lambda_{+}(\lambda+\alpha)p_{0,j-1}(1).\label{ee3}
\end{eqnarray}
Clearly, the solution of \eqref{ee1}-\eqref{ee3} should agree with the expansion \eqref{r1} for $i\gg 1$. Thus, we should add the condition
\begin{displaymath}
\begin{array}{c}
p_{i,j}(1)\sim c\left(\frac{\widehat{\lambda}}{2\widehat{\mu}}\right)^{i+j}\left(\frac{\widehat{\lambda}^{2}+4\widehat{\lambda}_{+}\widehat{\mu}}{\widehat{\lambda}^{2}+4(\widehat{\lambda}_{0}+\widehat{\lambda}_{+})\widehat{\mu}}\right)^{j-i},\,i\gg 1.
\end{array}
\end{displaymath}
This condition implies that we seek for solutions of \eqref{ee1}-\eqref{ee3} of the form
\begin{equation}
\begin{array}{c}
p_{i,j}(1)=c\left(\frac{\widehat{\lambda}}{2\widehat{\mu}}\times \frac{\widehat{\lambda}^{2}+4(\widehat{\lambda}_{0}+\widehat{\lambda}_{+})\widehat{\mu}}{\widehat{\lambda}^{2}+4\widehat{\lambda}_{+}\widehat{\mu}}\right)^{j}F(i),
\end{array}\label{sug}
\end{equation}
where $F(i)\sim \left(\frac{\widehat{\lambda}(\widehat{\lambda}^{2}+4\widehat{\lambda}_{+}\widehat{\mu})}{2\widehat{\mu}(\widehat{\lambda}^{2}+4(\widehat{\lambda}_{0}+\widehat{\lambda}_{+})\widehat{\mu})}\right)^{i}$, $i\gg1$. Substitute \eqref{sug} in \eqref{ee2}, \eqref{ee3} we obtain the following set of equations that $F(i)$ should satisfy:
\begin{equation}
\begin{array}{rl}
(\lambda(\lambda+\alpha)+\widehat{\mu})F(0)=&\widehat{\mu}[\frac{\lambda+\alpha}{\lambda+2\alpha}F(1)+\frac{\widehat{\lambda}(\widehat{\lambda}^{2}+4\widehat{\lambda}_{+}\widehat{\mu})}{2\widehat{\mu}(\widehat{\lambda}^{2}+4(\widehat{\lambda}_{0}+\widehat{\lambda}_{+})\widehat{\mu})}F(0)]\\&+\lambda_{+}(\lambda+\alpha)\frac{2\widehat{\mu}(\widehat{\lambda}^{2}+4\widehat{\mu}(\widehat{\lambda}_{0}+\widehat{\lambda}_{+}))}{\widehat{\lambda}(\widehat{\lambda}^{2}+4\widehat{\mu}\widehat{\lambda}_{+})}F(0),\label{f0}
\end{array}
\end{equation}
\begin{equation}
\begin{array}{rl}
(\widehat{\lambda}+2\widehat{\mu})F(i)=&(\widehat{\lambda}_{0}+\widehat{\lambda}_{+})F(i-1)+\widehat{\lambda}_{+}\frac{2\widehat{\mu}(\widehat{\lambda}^{2}+4\widehat{\mu}(\widehat{\lambda}_{0}+\widehat{\lambda}_{+}))}{\widehat{\lambda}(\widehat{\lambda}^{2}+4\widehat{\mu}\widehat{\lambda}_{+})}F(i)
+\widehat{\mu}F(i+1)\\&+\frac{\widehat{\lambda}(\widehat{\lambda}^{2}+4\widehat{\mu}\widehat{\lambda}_{+})}{2(\widehat{\lambda}^{2}+4\widehat{\mu}(\widehat{\lambda}_{0}+\widehat{\lambda}_{+}))}F(i),\,i\geq 1.\label{f1}
\end{array}
\end{equation}
The form of \eqref{f1} implies that $F(i)=A_{+}x_{+}^{i}+A_{-}x_{-}^{i}$, where
\begin{displaymath}
\begin{array}{rl}
x_{+}=\frac{\widehat{\lambda}(\widehat{\lambda}^{2}+4\widehat{\mu}(\widehat{\lambda}_{0}+\widehat{\lambda}_{+}))}{2\widehat{\mu}(\widehat{\lambda}^{2}+4\widehat{\mu}\widehat{\lambda}_{+}},&
x_{-}=\frac{\widehat{\lambda}_{0}+\widehat{\lambda}_{+}}{\widehat{\mu}x_{+}},\\
A_{-}=&\frac{\lambda(\lambda+\alpha)+\widehat{\mu}(1-x_{-}(1+\frac{\lambda+\alpha}{\lambda+2\alpha}))+\frac{\lambda_{+}(\lambda+\alpha)}{x_{+}}}{\widehat{\mu}(\frac{\lambda+\alpha}{\lambda+2\alpha}x_{-}+x_{+}-1)-\lambda(\lambda+\alpha)-\frac{\lambda_{+}(\lambda+\alpha)}{x_{+}}}.
\end{array}
\end{displaymath}
Moreover, $A_{+}=1$ (due to the asymptotic form of $F(i)$ given below \eqref{sug}), while $A_{-}$ is derived by the boundary condition \eqref{f0}. The obtained expansion for region II coincides with the expression \eqref{r1} for $i\gg 1$. Thus, \eqref{sug} is uniformly valid for all $i<j$ with $j\gg 1$. Due to the symmetry of the model, the expansion for $i>j$, $j\gg 1$ is obtained by replacing in the derived expressions $(i,j)$ with $(j,i)$.
\subsection{Proof of Lemma \ref{asym}}\label{prasym}
The proof is similar to the one presented in Section \ref{prsym}, although we need an additional condition to ensure that the approximated joint orbit queue length distribution is well defined. We start with region I ($i\gg1,j\gg1$), where equations \eqref{e1}, \eqref{e2} become
\begin{eqnarray}
R_{m,n}(0)=\frac{\mu}{\lambda+\alpha_{1}+\alpha_{2}}R_{m,n}(1),\label{ase1}\\
(\lambda+\alpha_{1}+\alpha_{2})R_{m,n}(1)=[\lambda_{0}H(1-n)+\lambda_{2}]R_{m-1,n-1}(1)\nonumber\\+[\lambda_{0}H(n+1)+\lambda_{1}]R_{m-1,n+1}(1)
+\lambda R_{m,n}(0)+\alpha_{1}R_{m+1,n-1}(0)+\alpha_{2}R_{m+1,n+1}(0)],\label{ase2}
\end{eqnarray}
which give
\begin{eqnarray}
(\widehat{\lambda}+\widehat{\mu}_{1}+\widehat{\mu}_{2})R_{m,n}(1)=[\widehat{\lambda}_{0}H(1-n)+\widehat{\lambda}_{2}]R_{m-1,n-1}(1) \nonumber \\+[\widehat{\lambda}_{0}H(n+1)+\widehat{\lambda}_{2}]R_{m-1,n+1}(1)
+\widehat{\mu}_{1}R_{m+1,n-1}(1)+\widehat{\mu}_{2}R_{m,n}(1),\label{asse2}
\end{eqnarray}
where $\widehat{\lambda}=\lambda(\lambda+\alpha_{1}+\alpha_{2})$, $\widehat{\lambda}_{l}=\lambda_{l}(\lambda+\alpha_{1}+\alpha_{2})$, $l=0,1,2$, $\widehat{\mu}_{l}=\alpha_{l}\mu$, $l=1,2$. Contrary to the symmetric case we are now seeking for solutions of \eqref{asse2} of the form
\begin{equation}
\begin{array}{rl}
R_{m,n}(1)=\gamma^{m}\left\{\begin{array}{ll}
c_{+}\delta_{+}^{n},&n\geq 1,\\
c_{0},&n=0,\\
c_{-}\delta_{-}^{-n},&n\leq -1.
\end{array}\right.
\end{array}\label{solv}
\end{equation}
Substitution of \eqref{solv} to \eqref{asse2} yields five equations, corresponding to the cases $n=0$, $n=\pm 1$, $n\geq 2$, $n\leq -2$:
\begin{eqnarray}
(\widehat{\lambda}+\widehat{\mu}_{1}+\widehat{\mu}_{2})c_{0}=(\frac{\widehat{\lambda}_{0}+\widehat{\lambda}_{2}}{\gamma}+\widehat{\mu}_{1}\gamma)c_{-}\delta_{-}+(\frac{\widehat{\lambda}_{0}+\widehat{\lambda}_{1}}{\gamma}+\widehat{\mu}_{2}\gamma)c_{+}\delta_{+},\label{u1}\\
(\widehat{\lambda}+\widehat{\mu}_{1}+\widehat{\mu}_{2})c_{+}\delta_{+}=(\frac{\widehat{\lambda}_{0}+2\widehat{\lambda}_{2}}{2\gamma}+\widehat{\mu}_{1}\gamma)c_{0}+(\frac{\widehat{\lambda}_{0}+\widehat{\lambda}_{1}}{\gamma}+\widehat{\mu}_{2}\gamma)c_{+}\delta_{+}^{2},\label{u2}\\
(\widehat{\lambda}+\widehat{\mu}_{1}+\widehat{\mu}_{2})c_{-}\delta_{-}=(\frac{\widehat{\lambda}_{0}+2\widehat{\lambda}_{1}}{2\gamma}+\widehat{\mu}_{2}\gamma)c_{0}+(\frac{\widehat{\lambda}_{0}+\widehat{\lambda}_{2}}{\gamma}+\widehat{\mu}_{1}\gamma)c_{-}\delta_{-}^{2},\label{u3}\\
\widehat{\lambda}+\widehat{\mu}_{1}+\widehat{\mu}_{2}=(\frac{\widehat{\lambda}_{2}}{\gamma}+\widehat{\mu}_{1}\gamma)\frac{1}{\delta_{+}}+(\frac{\widehat{\lambda}_{0}+\widehat{\lambda}_{1}}{\gamma}+\widehat{\mu}_{2}\gamma)\delta_{+},\label{u4}\\
\widehat{\lambda}+\widehat{\mu}_{1}+\widehat{\mu}_{2}=(\frac{\widehat{\lambda}_{1}}{\gamma}+\widehat{\mu}_{2}\gamma)\frac{1}{\delta_{-}}+(\frac{\widehat{\lambda}_{0}+\widehat{\lambda}_{2}}{\gamma}+\widehat{\mu}_{1}\gamma)\delta_{-}.\label{u5}
\end{eqnarray}
The system \eqref{u1}-\eqref{u5} is a set of five equations for the five unknowns $\gamma$, $\delta_{+}$, $\delta_{-}$, $c_{+}/c_{0}$, $c_{-}/c_{0}$. Using \eqref{u4} in \eqref{u2} and \eqref{u5} in \eqref{u3} we easily realize that
\begin{equation}
\begin{array}{rl}
c_{-}=&\frac{\epsilon_{+}}{\epsilon_{-}}c_{+}.
\end{array}\label{cc}
\end{equation}
where
\begin{equation}
\begin{array}{lr}
\epsilon_{+}=(\frac{\widehat{\lambda}_{2}}{\gamma}+\widehat{\mu}_{1}\gamma)(\frac{\widehat{\lambda}_{0}+2\widehat{\lambda}_{1}}{2\gamma}+\widehat{\mu}_{2}\gamma),&\epsilon_{-}=(\frac{\widehat{\lambda}_{1}}{\gamma}+\widehat{\mu}_{2}\gamma)(\frac{\widehat{\lambda}_{0}+2\widehat{\lambda}_{2}}{2\gamma}+\widehat{\mu}_{1}\gamma).
\end{array}\label{epsi}
\end{equation}
Thus, a natural choice is $c_{-}:=\sqrt{\frac{\epsilon_{+}}{\epsilon_{-}}}$, and $c_{+}:=\sqrt{\frac{\epsilon_{-}}{\epsilon_{+}}}$. 
Note that in the symmetric case, $c_{+}=c_{-}$. Following a procedure similar to the one used in \cite{kness}, after heavy algebraic manipulations (similar to those given in the proof of Lemma \ref{kl2}) we realize that $\gamma:=\rho=\frac{\widehat{\lambda}}{\widehat{\mu}_{1}+\widehat{\mu}_{2}}$, which is crucial for the stability condition, thus, asking $\widehat{\lambda}<\widehat{\mu}_{1}+\widehat{\mu}_{2}$. 

For $\gamma=\rho$, \eqref{u4}, \eqref{u5} have each one from two roots, namely $\delta_{+}=1$ or $\delta_{+}=\frac{\widehat{\lambda}_{2}+\rho^{2}\widehat{\mu}_{1}}{\widehat{\lambda}_{0}+\widehat{\lambda}_{1}+\rho^{2}\widehat{\mu}_{2}}$, and $\delta_{-}=1$ or $\delta_{-}=\frac{\widehat{\lambda}_{1}+\rho^{2}\widehat{\mu}_{2}}{\widehat{\lambda}_{0}+\widehat{\lambda}_{2}+\rho^{2}\widehat{\mu}_{1}}$, respectively. Therefore, the feasible candidates are $\delta_{+}=\frac{\widehat{\lambda}_{2}+\rho^{2}\widehat{\mu}_{1}}{\widehat{\lambda}_{0}+\widehat{\lambda}_{1}+\rho^{2}\widehat{\mu}_{2}}$, $\delta_{-}=\frac{\widehat{\lambda}_{1}+\rho^{2}\widehat{\mu}_{2}}{\widehat{\lambda}_{0}+\widehat{\lambda}_{2}+\rho^{2}\widehat{\mu}_{1}}$. Note, that when $\widehat{\lambda}_{0}>|\rho^{2}(\widehat{\mu}_{2}-\widehat{\mu}_{1})+\widehat{\lambda}_{1}-\widehat{\lambda}_{2}|$, both $0<\delta_{+}<1$, and $0<\delta_{-}<1$. Note that this case along with the assumption that $\rho>max\{\rho_{1},\rho_{2}\}$, where $\rho_{j}=\widehat{\lambda}_{j}/\widehat{\mu}_{j}$, $j=1,2$, corresponds to the so called \textit{strongly pooled} case mentioned in \cite[Theorem 2]{foley} for the standard join the shortest queue model \textbf{without} retrials. Substituting $\gamma$, $\delta_{+}$, $\delta_{-}$ in \eqref{cc} and \eqref{u1} we obtain after some algebra
\begin{displaymath}
\begin{array}{rl}
c_{0}=&\frac{\widehat{\lambda}_{1}+\rho^{2}\widehat{\mu}_{2}}{\widehat{\lambda}(1+\rho)}\sqrt{\frac{\epsilon_{+}}{\epsilon_{-}}}+\frac{\widehat{\lambda}_{2}+\rho^{2}\widehat{\mu}_{1}}{\widehat{\lambda}(1+\rho)}\sqrt{\frac{\epsilon_{-}}{\epsilon_{+}}},\\
\epsilon_{-}=\frac{(\widehat{\lambda}_{1}+\widehat{\mu}_{2}\rho^{2})(\widehat{\lambda}_{0}+2\widehat{\lambda}_{2}+2\widehat{\mu}_{1}\rho^{2})}{2\rho^{2}},&\epsilon_{+}=\frac{(\widehat{\lambda}_{2}+\widehat{\mu}_{1}\rho^{2})(\widehat{\lambda}_{0}+2\widehat{\lambda}_{1}+2\widehat{\mu}_{2}\rho^{2})}{2\rho^{2}}.
\end{array}
\end{displaymath}
Thus for region I ($i\gg1,j\gg1$) our results are summarized in
\begin{equation}
\begin{array}{rl}
p_{i,j}(1)\sim &c\rho^{i+j}\left\{\begin{array}{ll}
\sqrt{\frac{\epsilon_{-}}{\epsilon_{+}}}\left(\frac{\widehat{\lambda}_{2}+\rho^{2}\widehat{\mu}_{1}}{\widehat{\lambda}_{0}+\widehat{\lambda}_{1}+\rho^{2}\widehat{\mu}_{2}}\right)^{j-i},&j>i,\vspace{2mm}\\
\frac{\widehat{\lambda}_{1}+\rho^{2}\widehat{\mu}_{2}}{\widehat{\lambda}(1+\rho)}\sqrt{\frac{\epsilon_{+}}{\epsilon_{-}}}+\frac{\widehat{\lambda}_{2}+\rho^{2}\widehat{\mu}_{1}}{\widehat{\lambda}(1+\rho)}\sqrt{\frac{\epsilon_{-}}{\epsilon_{+}}},&i=j,\vspace{2mm}\\
\sqrt{\frac{\epsilon_{-}}{\epsilon_{+}}}\left(\frac{\widehat{\lambda}_{1}+\rho^{2}\widehat{\mu}_{2}}{\widehat{\lambda}_{0}+\widehat{\lambda}_{1}+\rho^{2}\widehat{\mu}_{1}}\right)^{i-j},&j<i.
\end{array}\right.\\
p_{i,j}(0)=&\frac{\mu}{\lambda+\alpha_{1}+\alpha_{2}}p_{i,j}(1),
\end{array}
\end{equation}
where $c$ is a multiplicative constant.

Next, we consider region II for $i=0$ or $1$, and $j\gg 1$. We should solve
\begin{eqnarray}
(\lambda+\alpha_{1}1_{\{i>0\}}+\alpha_{2})p_{i,j}(0)=\mu p_{i,j}(1),\label{aee1}\vspace{2mm}\\
(\widehat{\lambda}+\widehat{\mu}_{1}+\widehat{\mu}_{2})p_{i,j}(1)=\widehat{\lambda}_{2}p_{i,j-1}(1)+[\widehat{\lambda}_{0}+\widehat{\lambda}_{1}]p_{i-1,j}(1)+\widehat{\mu}_{1}p_{i+1,j}(1)+\widehat{\mu}_{2}p_{i,j+1}(1),\label{aee2}\vspace{2mm}\\
(\lambda(\lambda+\alpha_{2})+\widehat{\mu}_{2})p_{0,j}(1)=\frac{\widehat{\mu}_{1}(\lambda+\alpha_{2})}{\lambda+\alpha_{1}+\alpha_{2}}p_{1,j}(1)+\widehat{\mu}_{2}p_{0,j+1}(1)+\lambda_{2}(\lambda+\alpha_{2})p_{0.j-1}(1),\label{aee3}
\end{eqnarray}
and 
\begin{displaymath}
\begin{array}{c}
p_{i,j}(1)\sim c\sqrt{\frac{\epsilon_{-}}{\epsilon_{+}}}\left(\rho\frac{\widehat{\lambda}_{2}+\rho^{2}\widehat{\mu}_{1}}{\widehat{\lambda}_{0}+\widehat{\lambda}_{1}+\rho^{2}\widehat{\mu}_{2}}\right)^{j}\left(\rho\frac{\widehat{\lambda}_{0}+\widehat{\lambda}_{1}+\rho^{2}\widehat{\mu}_{2}}{\widehat{\lambda}_{2}+\rho^{2}\widehat{\mu}_{1}}\right)^{i},\,i\gg1.
\end{array}
\end{displaymath}
Thus, we seek solutions of \eqref{aee2}-\eqref{aee3} of the form
\begin{displaymath}
\begin{array}{c}
p_{i,j}(1)\sim c\sqrt{\frac{\epsilon_{-}}{\epsilon_{+}}}\left(\rho\frac{\widehat{\lambda}_{2}+\rho^{2}\widehat{\mu}_{1}}{\widehat{\lambda}_{0}+\widehat{\lambda}_{1}+\rho^{2}\widehat{\mu}_{2}}\right)^{j}F(i),
\end{array}
\end{displaymath}
and obtain the following set of equations for $F(i)$:
\begin{equation}
\begin{array}{rl}
(\widehat{\lambda}+\widehat{\mu}_{1}+\widehat{\mu}_{2})F(i)=&(\widehat{\lambda}_{0}+\widehat{\lambda}_{1})F(i-1)+\left(\frac{\widehat{\lambda}_{2}(\widehat{\lambda}_{0}+\widehat{\lambda}_{1}+\rho^{2}\widehat{\mu}_{2})}{\rho(\widehat{\lambda}_{2}+\rho^{2}\widehat{\mu}_{1})}+\frac{\widehat{\mu}_{2}\rho(\widehat{\lambda}_{2}+\rho^{2}\widehat{\mu}_{1})}{\widehat{\lambda}_{0}+\widehat{\lambda}_{1}+\rho^{2}\widehat{\mu}_{2}}\right) F(i)\\&
+\widehat{\mu}_{1}F(i+1),
\end{array}\label{w1}
\end{equation}
\begin{equation}
\begin{array}{rl}
(\lambda(\lambda+\alpha_{2})+\widehat{\mu}_{2})F(0)=&\widehat{\mu}_{1}\frac{\lambda+\alpha_{2}}{\lambda+\alpha_{1}+\alpha_{2}}F(1)+\frac{\widehat{\mu}_{2}\rho(\widehat{\lambda}_{2}+\rho^{2}\widehat{\mu}_{1})}{\widehat{\lambda}_{0}+\widehat{\lambda}_{1}+\rho^{2}\widehat{\mu}_{2}}F(0)\\&+\lambda_{2}(\lambda+\alpha_{2})\frac{\widehat{\lambda}_{0}+\widehat{\lambda}_{1}+\rho^{2}\widehat{\mu}_{2}}{\rho(\widehat{\lambda}_{2}+\rho^{2}\widehat{\mu}_{1})}F(0),\label{w2}
\end{array}
\end{equation}
with $F(i)\sim \left(\frac{\rho(\widehat{\lambda}_{0}+\widehat{\lambda}_{1}+\rho^{2}\widehat{\mu}_{2})}{\widehat{\lambda}_{2}+\rho^{2}\widehat{\mu}_{1}}\right)^{i}$, $i\gg 1$. The solution to \eqref{w1}, \eqref{w2} is $F(i)=x_{+}^{i}+A_{-}x_{-}^{i}$, where
\begin{displaymath}
\begin{array}{rl}
x_{+}=\frac{\rho(\widehat{\lambda}_{0}+\widehat{\lambda}_{1}+\rho^{2}\widehat{\mu}_{2})}{\widehat{\lambda}_{2}+\rho^{2}\widehat{\mu}_{1}},&x_{-}=\frac{(\widehat{\lambda}_{0}+\widehat{\lambda}_{1})(\widehat{\lambda}_{2}+\rho^{2}\widehat{\mu}_{1})}{\widehat{\mu}_{1}\rho(\widehat{\lambda}_{0}+\widehat{\lambda}_{1}+\rho^{2}\widehat{\mu}_{2})},\vspace{2mm}\\
A_{-}=&\frac{x_{+}^{2}(\widehat{\mu}_{1}\frac{\lambda+\alpha_{2}}{\lambda+\alpha_{1}+\alpha_{2}}+\frac{\lambda_{2}(\lambda+\alpha_{2})}{\rho^{2}})-x_{+}(\lambda(\lambda+\alpha_{2})+\widehat{\mu}_{2})+\rho^{2}\widehat{\mu}_{2}}{x_{+}(\lambda(\lambda+\alpha_{2})+\widehat{\mu}_{2})-[(\lambda_{0}+\lambda_{1})(\lambda+\alpha_{2})+\rho^{2}\widehat{\mu}_{2}]-\frac{\lambda_{2}(\lambda+\alpha_{2})}{\rho^{2}}x_{+}^{2}},
\end{array}
\end{displaymath}
where $A_{-}$ is obtained using \eqref{w2}. Moreover, $x_{+}>1$ when $\widehat{\lambda}_{2}+\widehat{\mu}_{1}\rho^{2}<\rho(\widehat{\lambda}_{0}+\widehat{\lambda}_{1}+\widehat{\mu}_{2}\rho^{2})$ (note that this corresponds to the strongly balanced condition mentioned in Theorem \ref{decay}). Furthermore, $x_{-}/x_{+}<1$ when $(\widehat{\lambda}_{2}+\widehat{\mu}_{1}\rho^{2})\sqrt{\frac{\widehat{\lambda}_{0}+\widehat{\lambda}_{1}}{\widehat{\mu}_{1}}}<\rho(\widehat{\lambda}_{0}+\widehat{\lambda}_{1}+\widehat{\mu}_{2}\rho^{2})$. Thus, for $i<j$, $j\gg 1$,
\begin{equation}
\begin{array}{rl}
p_{i,j}(1)\sim&c\sqrt{\frac{\epsilon_{-}}{\epsilon_{+}}}\left(\frac{\rho(\widehat{\lambda}_{2}+\rho^{2}\widehat{\mu}_{1})}{\widehat{\lambda}_{0}+\widehat{\lambda}_{1}+\rho^{2}\widehat{\mu}_{2}}\right)^{i}\left(\frac{\rho(\widehat{\lambda}_{0}+\widehat{\lambda}_{1}+\rho^{2}\widehat{\mu}_{2})}{\widehat{\lambda}_{2}+\rho^{2}\widehat{\mu}_{1}}\right)^{j}[1+A_{-}\left(\frac{x_{-}}{x_{+}}\right)^{i}],\\
p_{i,j}(0)=&\frac{\mu}{\lambda+\alpha_{1}1_{\{i>0\}}+\alpha_{2}}p_{i,j}(1).
\end{array}\label{o1}
\end{equation}

An analogous analysis can be performed for region III ($i\gg1$, $j=0,1$), and results in
\begin{equation}
\begin{array}{rl}
p_{i,j}(1)\sim&c\sqrt{\frac{\epsilon_{+}}{\epsilon_{-}}}\left(\frac{\rho(\widehat{\lambda}_{1}+\rho^{2}\widehat{\mu}_{2})}{\widehat{\lambda}_{0}+\widehat{\lambda}_{2}+\rho^{2}\widehat{\mu}_{1}}\right)^{i}\left(\frac{\rho(\widehat{\lambda}_{0}+\widehat{\lambda}_{2}+\rho^{2}\widehat{\mu}_{1})}{\widehat{\lambda}_{1}+\rho^{2}\widehat{\mu}_{2}}\right)^{j}[1+B_{-}\left(\frac{y_{-}}{y_{+}}\right)^{j}],\,i>j,\,i\gg 1,\\
p_{i,j}(0)=&\frac{\mu}{\lambda+\alpha_{2}1_{\{j>0\}}+\alpha_{1}}p_{i,j}(1),
\end{array}\label{o2}
\end{equation}
where, 
\begin{displaymath}
\begin{array}{rl}
y_{-}=\frac{(\widehat{\lambda}_{0}+\widehat{\lambda}_{2})(\widehat{\lambda}_{1}+\rho^{2}\widehat{\mu}_{2})}{\widehat{\mu}_{2}\rho(\widehat{\lambda}_{0}+\widehat{\lambda}_{2}+\rho^{2}\widehat{\mu}_{1})},&y_{+}=\frac{\rho(\widehat{\lambda}_{0}+\widehat{\lambda}_{2}+\rho^{2}\widehat{\mu}_{1})}{\widehat{\lambda}_{1}+\rho^{2}\widehat{\mu}_{2}}\vspace{2mm}\\
B_{-}=&\frac{y_{+}^{2}(\widehat{\mu}_{2}\frac{\lambda+\alpha_{1}}{\lambda+\alpha_{1}+\alpha_{2}}+\frac{\lambda_{1}(\lambda+\alpha_{1})}{\rho^{2}})-y_{+}(\lambda(\lambda+\alpha_{1})+\widehat{\mu}_{1})+\rho^{2}\widehat{\mu}_{1}}{y_{+}(\lambda(\lambda+\alpha_{1})+\widehat{\mu}_{1})-[(\lambda_{0}+\lambda_{2})(\lambda+\alpha_{1})+\rho^{2}\widehat{\mu}_{1}]-\frac{\lambda_{1}(\lambda+\alpha_{1})}{\rho^{2}}y_{+}^{2}}.
\end{array}
\end{displaymath}
Moreover, $y_{+}>1$ when $\widehat{\lambda}_{1}+\widehat{\mu}_{2}\rho^{2}<\rho(\widehat{\lambda}_{0}+\widehat{\lambda}_{2}+\widehat{\mu}_{1}\rho^{2})$ (note that this corresponds to the strongly balanced condition mentioned in Theorem \ref{decay}), and $y_{-}/y_{+}<1$ when $(\widehat{\lambda}_{1}+\widehat{\mu}_{2}\rho^{2})\sqrt{\frac{\widehat{\lambda}_{0}+\widehat{\lambda}_{2}}{\widehat{\mu}_{2}}}<\rho(\widehat{\lambda}_{0}+\widehat{\lambda}_{2}+\widehat{\mu}_{1}\rho^{2})$.
\section{Numerical results}\label{num}
In the following, we numerically validate and compare the theoretical results obtained based on asymptotic analysis in subsection \ref{taildecay}, with those obtained by the heuristic approach in \ref{heur}. We will see that as $m\to\infty$, i.e., $\min\{N_{1,n},N_{2,n}\}\to\infty$ the expressions derived from  the tail asymptotic analysis coincide with the heuristic approximation expressions. Moreover, we notice that even when $m$ takes small values the difference of the derived expressions is negligible.  

Set $\lambda_{0}=0.15$, $\lambda_{1}=0.05$, $\lambda_{2}=0.01$, $\mu=0.44$. Then, for $\alpha_{1}=0.25$, $\alpha_{2}=0.1$ we investigate whether the results obtained by the asymptotic analysis presented in subsection \ref{decayp} agreed with those obtained in subsection \ref{heur} through our heuristic approach. In this direction, we focus on the absolute difference $|\mathbb{P}(N_{1}=i,N_{2}=j)-\mathbb{P}(min\{i,j\},j-i)|$, where $\mathbb{P}(N_{1}=i,N_{2}=j)=p_{i,j}(0)+p_{i,j}(1)$ obtained with the aid of Lemma \ref{asym} in subsection \ref{heur}, and $\mathbb{P}(min\{i,j\}=m,j-i=l)=\pi_{m,l}(0)+\pi_{m,l}(1)$ with the aid of Theorem \ref{decay} in subsection \ref{taildecay}. Note that under such a setting, $\rho=0.7636>max\{\rho_{1}=0.2545,\rho_{2}=0.1273\}$, and $\widehat{\lambda}_{0}-|\rho^{2}(\widehat{\mu}_{2}-\widehat{\mu}_{1})+\widehat{\lambda}_{1}-\widehat{\lambda}_{2}|=0.0679>0$, and $\rho(\gamma_{2}+\widehat{\lambda}_{0})>\gamma_{1}$, $\rho(\gamma_{1}+\widehat{\lambda}_{0})>\gamma_{2}$ (i.e., the orbit queues are strongly balanced); see Table \ref{table:notation}.
\begin{table}[!t]
	\small
	\caption{Validation of the asymptotic stationary approximations}
	\centering
	\begin{tabular}{ | c | l | c|l|}
		\hline
		\bfseries Indicated values $(i,j)$ & \bfseries Difference\\
		\hline
		$(10,100)$ &$4.0667*10^{-40}$\\
		\hline$(10,200)$ &$2.3404*10^{-81}$ \\
		\hline$(10,300)$ &$1.3469*10^{-122}$\\\hline $(10,400)$&$7.7516*10^{-164}$,\\
		\hline$(100,10)$& $1.2203*10^{-54}$\\
		\hline$(200,10)$& $4.555*10^{-112}$\\
		\hline$(300,10)$& $1.7003*10^{-169}$\\
		\hline$(400,10)$& $6.3466*10^{-227}$\\
		\hline
	\end{tabular} \label{table:notation}	
\end{table}
Our results show that the stationary approximations through the heuristic approach agreed with the results obtained by the asymptotic analysis.
\begin{figure}[H]
\centering
\includegraphics[scale=0.65]{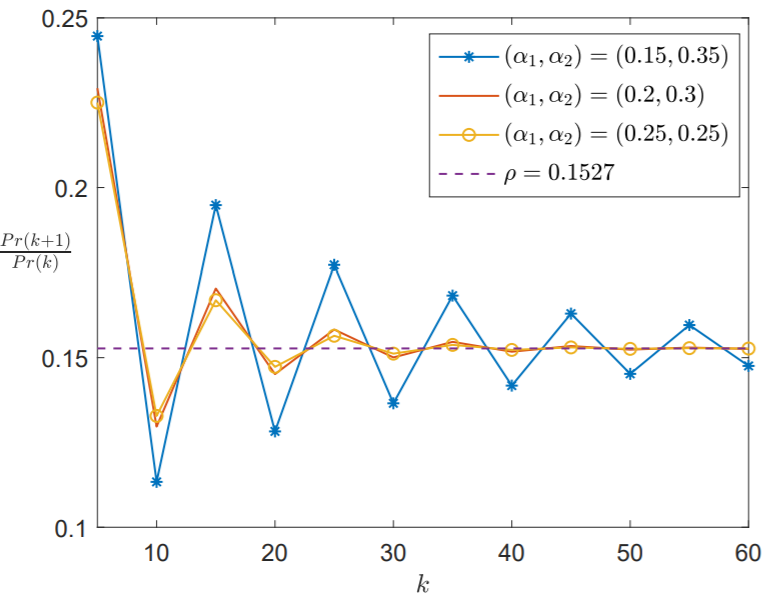}
\caption{The ratio $\frac{Pr(k+1)}{Pr(k)}$ for increasing values of $k=N_{1}+N_{2}$.}
\label{unibf}
\end{figure}
We now focus on the stationary approximations derived through the heuristic approach in subsection \ref{heur}. In Figure \ref{unibf} we observe the ratio $\frac{Pr(k+1)}{Pr(k)}$ for increasing values of the total number of jobs in orbits, i.e., $k=N_{1}+N_{2}$, where $Pr(k)=p_{i,j}(1)+p_{i,j}(0)$ with $k=i+j$, and $\lambda_{0}=0.04$, $\lambda_{1}=\lambda_{2}=0.01$, $\mu=0.44$. It is readily seen that $\lim_{k\to\infty}\frac{Pr(k+1)}{Pr(k)}=\rho=\frac{\widehat{\lambda}}{\widehat{\mu}_{1}+\widehat{\mu}_{2}}$. Moreover, we can observe that the more the difference $\alpha_{1}-\alpha_{2}$ get smaller, the faster this ratio tends to $\rho$. In other words, the asymmetry of the retrial rates affects the asymptotic behaviour.

Moreover, we can also observe that the presence of the dedicated traffic, also heavily affects the way the ratio $\frac{Pr(k+1)}{Pr(k)}$ tends to $\rho$. In the following table, we can observe this trend for $\lambda_{1}=0=\lambda_{2}$, $\lambda_{0}=0.06$, $\mu=0.44$, compared with the case where $\lambda_{1}=0.01=\lambda_{2}$, $\lambda_{0}=0.04$, $\mu=0.44$; see Table \ref{table2}
\begin{table}[!t]
	\small
	\caption{Effect of dedicated traffic}
	\centering
	\begin{tabular}{ |c| c | c|}
		\hline
		$\frac{Pr(k+1)}{Pr(k)}$ for $(\alpha_{1},\alpha_{2})=(0.15,0.35)$& $(\lambda_{0},\lambda_{1},\lambda_{2})=(0.06,0,0)$& $(\lambda_{0},\lambda_{1},\lambda_{2})=(0.04,0.01,0.01)$\\
		\hline
		$k=5$&$0.1526$&$0.2446$\\
		\hline$k=15$&$0.1527$ &$0.1948$ \\
		\hline$k=35$&$0.1527$ &$0.1683$\\\hline $k=55$&$0.1527$&$0.1596$\\
		\hline\hline
		$\frac{Pr(k+1)}{Pr(k)}$ for $(\alpha_{1},\alpha_{2})=(0.25,0.25)$& $(\lambda_{0},\lambda_{1},\lambda_{2})=(0.06,0,0)$& $(\lambda_{0},\lambda_{1},\lambda_{2})=(0.04,0.01,0.01)$\\
		\hline
		$k=5$&$0.1527$&$0.225$\\
		\hline$k=15$&$0.1527$ &$0.1669$ \\
		\hline$k=35$&$0.1527$ &$0.1538$\\\hline $k=55$&$0.1527$&$0.1527$\\
		\hline
	\end{tabular} \label{table2}	
\end{table}
We observe, that in case of no dedicated traffic the convergence of the ratio $\frac{Pr(k+1)}{Pr(k)}$ to $\rho$ is really faster compared with the case where we assumed dedicated traffic. Note that in these scenarios we have assumed that the total arrival rate is fixed and equal to $\lambda=0.06$, so that $\rho=0.1527$ is fixed in both cases.
 
Moreover, we note when $\rho$ increases, the ratio $\frac{Pr(k+1)}{Pr(k)}$ tends to $\rho$ very fast (from $k\geq 10$ the ratio is very close to $\rho$), as shown in Figure \ref{unibfc}, where $\lambda_{0}=0.19$, $\lambda_{1}=\lambda_{1}=0.01$ $\mu=0.44$ (in the case of \textit{no dedicated traffic}, we assumed $\lambda_{1}=\lambda_{2}=0$, $\lambda_{0}=0.21$).
\begin{figure}[H]
\centering
\includegraphics[scale=0.65]{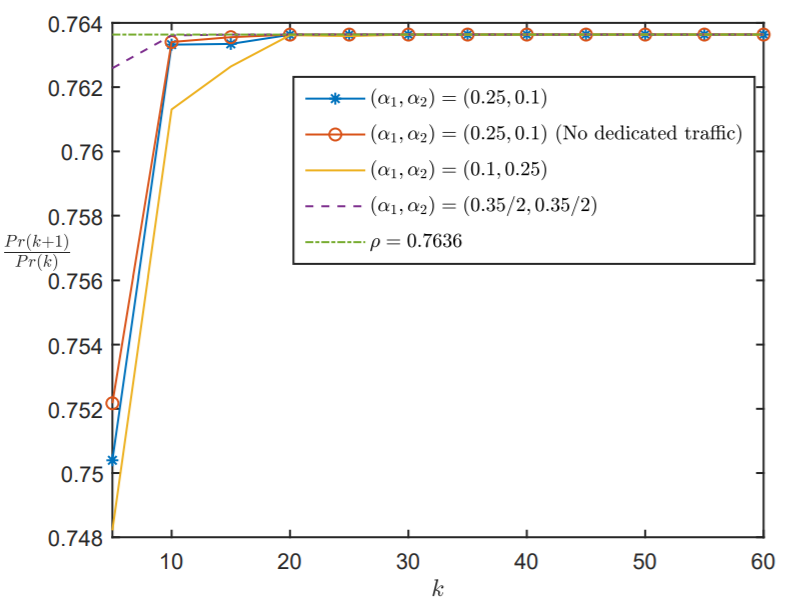}
\caption{The ratio $\frac{Pr(k+1)}{Pr(k)}$ for increasing values of $k=N_{1}+N_{2}$.}
\label{unibfc}
\end{figure}
\section{Conclusion and future work}
In this work, we introduced the generalized join the shortest queue system with retrials and two infinite capacity orbit queues. Three independent streams of jobs, namely a \textit{smart}, and two \textit{dedicated} streams, flow into a single server service
system, which can hold at most one job. Blocked smart jobs join the shortest orbit queue, while blocked jobs from the dedicated streams join their corresponding orbit queue. We remind that the system is described as a Markov modulated two-dimensional random walk, but this kind of modulation allowed for a completely tractable analysis.  

We establish the geometric tail asymptotics along the minimum direction for a fixed value of the difference among orbit queue lengths. Moreover, we apply a heuristic approach to obtain stationary approximations of the number of the joint orbit queue length distribution, which are accurate when one of the orbit queue lengths is large, and thus agreed with the asymptotic analysis. We have
shown that even though the exact solution to the problem may be quite complicated, the asymptotic expansions of $p_{i,j}(k)$ are relatively simple and clearly indicate the dependence of the stationary distribution on the parameters of the model.

We also cope with the stability condition, and based on the stability of the censored chain on the busy states along with extensive simulation experiments, we conjecture that the ergodicity conditions of the censored chain coincide with those of the original model. We postpone the formal proof of this conjecture in a future work. We also leave as a future work the complete investigation of the tail decay rate problem. More precisely, we proved the exact geometric decay under specific conditions given in Theorem \ref{decay} (which are similar to those given in \cite{foley}), but it is still an open problem the investigation of the decay rates when these conditions collapse. 

Finally, simulation experiments indicate that in heavy traffic (i.e., $\rho\to 1$) and when $\lambda_{0}>|\lambda_{1}-\lambda_{2}|$, $\alpha_{1}=\alpha_{2}$, our system exhibits a \textit{state-space collapse} behaviour; see Figure \ref{fgh} (right). It seems that when the first condition fails, our model does not retain this behaviour; see Figure \ref{fgh} (left), where $\lambda_{1}>\lambda_{0}>\lambda_{2}$, and $\lambda_{0}<|\lambda_{1}-\lambda_{2}|$ (i.e., the impact of dedicated traffic for orbit queue 1 is very crucial). Moreover, it also seems that the condition $\alpha_{1}=\alpha_{2}$ is not crucial (see Figure \ref{fghm} (left)).

This means that in heavy-traffic and under these conditions, our load balancing scheme collapses to a one-dimensional line where all the orbit queue lengths are equal. A similar result was proven for the GJSQ system \textbf{without} retrials in the seminal paper \cite{fosc}. State-space collapse occurs because the smart traffic flow dominates over the dedicated traffic flows and ``forces" the two orbit queues to be equal. Thus, it seams that in the heavy traffic regime the sum of the orbit queue lengths can be approximated by the \textit{reference system} described at the end of subsection \ref{taildecay}. The reference system behaves as if there is only a single orbit queue with all the ``servers" (i.e., the retrial servers) pooled together as an aggregated ``server" (in standard JSQ systems this is called \textit{complete resource pooling}). This result implies that under specific conditions, GJSOQ is asymptotically optimal, i.e., heavy-traffic delay
optimal, since the response time in the pooled single-orbit system is stochastically less than that of a typical load balancing
system, i.e. the reference system seams to serve as a lower bound (in the stochastic sense) on the
total orbit queue length of the original model. In a future study we plan to prove formally this justification.
\begin{figure}[H]
\centering
\includegraphics[scale=0.4]{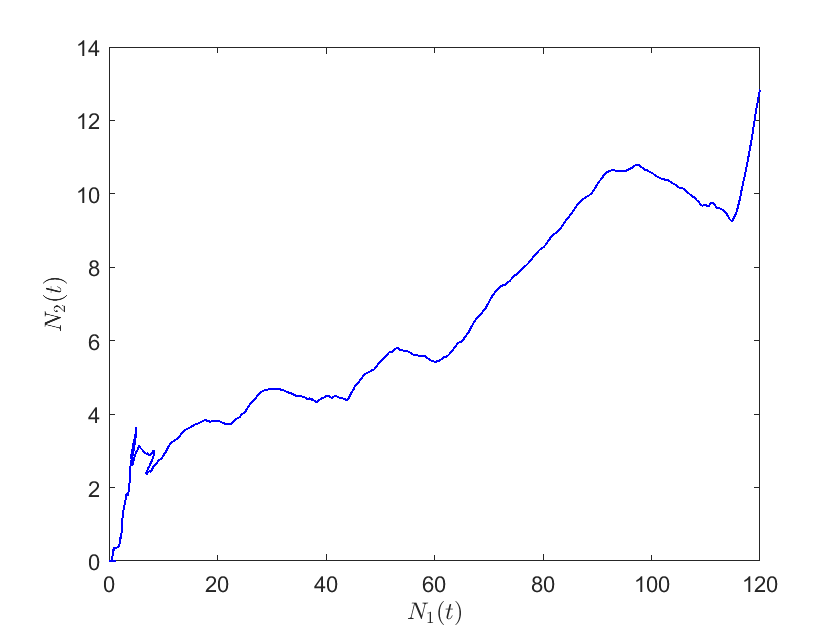}\includegraphics[scale=0.4]{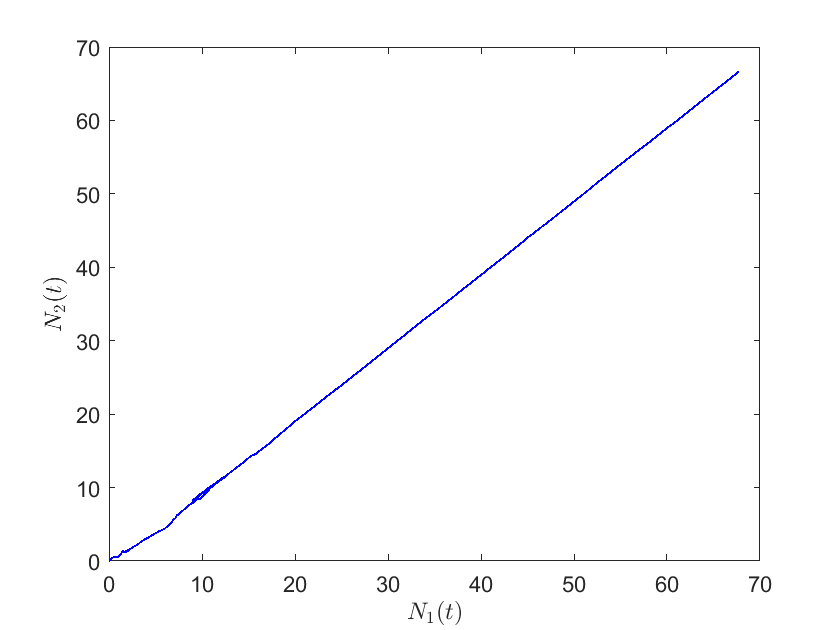}
\caption{Orbit dynamics for $\rho\to 1$ when $\lambda_{0}<|\lambda_{1}-\lambda_{2}|$ (left), and when $\lambda_{0}>|\lambda_{1}-\lambda_{2}|$ (right).}
\label{fgh}
\end{figure}
\section*{Acknowledgements}
I would like to thank the Editors and the reviewers for the insightful remarks, which helped to improve the original exposition.\vspace{0.5cm}\\
\textbf{Competing interests:} The author declare none.

\bibliographystyle{abbrv}
\bibliography{peis_CR}
\appendix
\section{Proof of Lemma \ref{stab}}\label{app1}
Denote by $(M_{i}^{(r_k)},M_{j}^{(r_k)})$ the mean jump vectors in angle $r_{k}$, $k=1,2,$ and by $(M_{i}^{(h)},M_{j}^{(h)})$, $(M_{i}^{(v)},M_{j}^{(v)})$ the mean jump vectors in rays $h$ and $v$, respectively. Then,
\begin{displaymath}
\begin{array}{rlcrl}
M_{i}^{(r_1)}=&\lambda_{1}-\frac{\widehat{\mu}_{1}}{\lambda+\alpha_{1}+\alpha_{2}},&&M_{j}^{(r_1)}=&\lambda_{0}+\lambda_{2}-\frac{\widehat{\mu}_{2}}{\lambda+\alpha_{1}+\alpha_{2}},\\
M_{i}^{(r_2)}=&\lambda_{0}+\lambda_{1}-\frac{\widehat{\mu}_{1}}{\lambda+\alpha_{1}+\alpha_{2}},&&M_{j}^{(r_2)}=&\lambda_{2}-\frac{\widehat{\mu}_{2}}{\lambda+\alpha_{1}+\alpha_{2}},\\
M_{i}^{(h)}=&\lambda_{1}-\frac{\widehat{\mu}_{1}}{\lambda+\alpha_{1}},&&M_{j}^{(h)}=&\lambda_{0}+\lambda_{2},\\
M_{i}^{(v)}=&\lambda_{0}+\lambda_{1},&&M_{j}^{(h)}=&\lambda_{2}-\frac{\widehat{\mu}_{2}}{\lambda+\alpha_{2}}.
\end{array}
\end{displaymath}
The mean jump vector from ray $d$ equals $\frac{1}{2}(M_{i}^{(r_1)}+M_{i}^{(r_2)},M_{j}^{(r_1)}+M_{j}^{(r_2)})$, while $M_{i}^{(r_1)}+M_{j}^{(r_1)}=M_{i}^{(r_2)}+M_{j}^{(r_2)}=\lambda-\frac{\widehat{\mu}_{1}+\widehat{\mu}_{2}}{\lambda+\alpha_{1}+\alpha_{2}}=\frac{\widehat{\lambda}-\widehat{\mu}_{1}-\widehat{\mu}_{2}}{\lambda+\alpha_{1}+\alpha_{2}}$. 

Let us first focus on the sufficient part. We use the well known Foster's criterion \cite[Theorem 2.3.3]{fay}, i.e., the Markov chain $\{X(n);n\geq 0\}$ is ergodic iff there exists a positive function $f(i,j)$ on $\mathbb{Z}_{+}^{2}$, $\epsilon>0$ and a finite set $A\in\mathbb{Z}_{+}^{2}$ such that
\begin{equation}
E(f(i+\theta_{i},j+\theta_{j}))-f(i,j)<-\epsilon,\,\,(x,y)\in\mathbb{Z}_{+}^{2}/A,\label{foster}
\end{equation}
where $(\theta_i, \theta_j)$ is a random vector distributed as a one-step jump of the chain $X(n)$ from the state $(i,j)$. 

We first assume that assertion 1 in Theorem \ref{stab} holds, i.e., $\widehat{\lambda}_{1}<\widehat{\mu}_{1}$, $\widehat{\lambda}_{2}<\widehat{\mu}_{2}$, $\widehat{\lambda}<\widehat{\mu}_{1}+\widehat{\mu}_{2}$. Consider the function $f(i,j)=\sqrt{i^2+j^2}$. Then, simple manipulation yields
\begin{displaymath}
E(f(i+\theta_{i},j+\theta_{j}))-f(i,j)=\frac{xE(\theta_{i})+yE(\theta_{j})}{f(i,j)}+o(1),\,\text{as }i^{2}+j^{2}\to\infty.
\end{displaymath} 
Then, in angle $r_{1}$ we have $E(\theta_{i})=M_{i}^{(r_{1})}=\frac{\widehat{\lambda}_{1}-\widehat{\mu}_{1}}{\lambda+\alpha_{1}+\alpha_{2}}<0$, and $E(\theta_{i})+E(\theta_{j})=M_{i}^{(r_{1})}+M_{j}^{(r_{1})}=\frac{\widehat{\lambda}-\widehat{\mu}_{1}-\widehat{\mu}_{2}}{\lambda+\alpha_{1}+\alpha_{2}}<0$. Moreover, in $r_{1}$
\begin{displaymath}
E(f(i+\theta_{i},j+\theta_{j}))-f(i,j)\leq \frac{j(E(\theta_{i})+E(\theta_{j}))}{j\sqrt{2}}+o(1)<-\epsilon_{1}
\end{displaymath} 
for $\epsilon_{1}>0$ as $i^{2}+j^{2}\to\infty$.

For $(i,j)$ in ray $h$, the condition $\widehat{\lambda}_{1}<\widehat{\mu}_{1}$, $i.e.,$ $M_{i}^{(r_{1})}<0$, implies $M_{i}^{(h)}<0$, since 
\begin{displaymath}
M_{i}^{(h)}=M_{i}^{(r_{1})}+\widehat{\mu}_{1}(\frac{1}{\lambda+\alpha_{1}+\alpha_{2}}-\frac{1}{\lambda+\alpha_{1}})=M_{i}^{(r_{1})}-\frac{\alpha_{2}\widehat{\mu}_{1}}{(\lambda+\alpha_{1})(\lambda+\alpha_{1}+\alpha_{2})}<0.
\end{displaymath}
Therefore,
\begin{displaymath}
E(f(i+\theta_{i},j+\theta_{j}))-f(i,j)\leq \frac{i M_{i}^{(h)}}{i}+o(1)<-\epsilon_{2}
\end{displaymath} 
for $\epsilon_{2}>0$ as $i^{2}\to\infty$. The case for angle $r_{2}$ and ray $v$ are symmetric to those for $r_{1}$ and $h$, respectively, thus, \eqref{foster} is verified similarly. The case for the ray $d$, i.e., $i=j$ is treated similarly. Indeed, having in mind that $(E(\theta_{i}),E(\theta_{j}))=\frac{1}{2}(M_{i}^{(r_1)}+M_{i}^{(r_2)},M_{j}^{(r_1)}+M_{j}^{(r_2)})$, \eqref{foster} reads
\begin{displaymath}
\begin{array}{rl}
E(f(i+\theta_{i},j+\theta_{j}))-f(i,j)=&\frac{i2E(\theta_{i})}{i\sqrt{2}}+o(1)=\frac{M_{i}^{(r_1)}+M_{j}^{(r_1)}}{\sqrt{2}}+o(1)<-\epsilon_{3},\,\text{as }i\to\infty.
\end{array}
\end{displaymath}

Assume now that assertion 2 in Theorem \ref{stab} holds: $\widehat{\lambda}_{1}\geq \widehat{\mu}_{1}$, $\widehat{\lambda}_{2}< \widehat{\mu}_{2}$, and $M_{i}^{(r_1)}M_{j}^{(h)}-M_{i}^{(h)}M_{j}^{(r_1)}<0\Leftrightarrow f_{1}<0$. In such a case the following hold:
\begin{displaymath}
\begin{array}{rl}
M_{i}^{(r_2)}>&M_{i}^{(r_1)}\geq 0,\\
M_{j}^{(r_2)}<&M_{j}^{(r_1)}<0,\\
M_{i}^{(r_1)}M_{j}^{(h)}<&M_{i}^{(h)}M_{j}^{(r_1)}\Rightarrow M_{i}^{(h)}<\frac{M_{i}^{(r_1)}}{M_{j}^{(r_1)}}M_{j}^{(h)}<0\,\,\,\,(\text{since }M_{j}^{(r_1)}<0),\\
M_{j}^{(v)}=&M_{j}^{(r_2)}-\frac{\alpha_{1}\widehat{\mu}_{2}}{(\lambda+\alpha_{1})(\lambda+\alpha_{1}+\alpha_{2})}<0.
\end{array}
\end{displaymath}
Note also that $M_{i}^{(r_1)}M_{j}^{(h)}-M_{i}^{(h)}M_{j}^{(r_1)}<0$ is equivalent to $\lambda_{1}<\mu_{1}^{*}-\frac{\lambda_{0}+\lambda_{2}}{\tilde{\mu}_{2}}(\mu_{1}^{*}-\tilde{\mu}_{1})$, where $\tilde{\mu}_{i}=\frac{\widehat{\mu}_{i}}{\lambda+\alpha_{1}+\alpha_{2}}$, $i=1,2$, and $\mu_{1}^{*}=\frac{\widehat{\mu}_{1}}{\lambda+\alpha_{1}}$. Then,
\begin{displaymath}
\begin{array}{rl}
M_{i}^{(r_1)}+M_{j}^{(r_2)}=\lambda-\tilde{\mu}_{1}-\tilde{\mu}_{2}<\lambda_{1}-\mu_{1}^{*}+\frac{\lambda_{0}+\lambda_{2}}{\tilde{\mu}_{2}}(\mu_{1}^{*}-\tilde{\mu}_{1})<0.
\end{array}
\end{displaymath}
Indeed, 
\begin{displaymath}
\begin{array}{rl}
\lambda-\tilde{\mu}_{1}-\tilde{\mu}_{2}<\lambda_{1}-\mu_{1}^{*}+\frac{\lambda_{0}+\lambda_{2}}{\tilde{\mu}_{2}}(\mu_{1}^{*}-\tilde{\mu}_{1})\Leftrightarrow&(\frac{\lambda_{0}+\lambda_{2}}{\tilde{\mu}_{2}}-1)(\tilde{\mu}_{2}+\tilde{\mu}_{1}-\mu_{1}^{*})<0.
\end{array}
\end{displaymath}
Since $M_{j}^{(r_1)}=\lambda_{0}+\lambda_{2}-\tilde{\mu}_{2}<0$, it suffices to show that $\tilde{\mu}_{2}+\tilde{\mu}_{1}-\mu_{1}^{*}>0$. Indeed, simple calculations yields $\tilde{\mu}_{2}+\tilde{\mu}_{1}-\mu_{1}^{*}=\frac{\lambda\tilde{\mu}_{2}}{\lambda+\alpha_{1}}>0$. Thus, $M_{i}^{(r_1)}+M_{j}^{(r_1)}=M_{i}^{(r_2)}+M_{j}^{(r_2)}<0.$

We construct the function $f(i,j) =\sqrt{ki^{2}+lj^{2}+wij}$ with $k,l>0$, $kl>w^{2}/4$, satisfying \eqref{foster}. We first choose $k,w>0$ such that 
\begin{displaymath}
\begin{array}{rl}
2kM_{i}^{(r_1)} + lM_{j}^{(r_1)}<&0,\\
2kM_{i}^{(h)} + lM_{j}^{(h)}<&0,
\end{array}
\end{displaymath}
or equivalently
\begin{displaymath}
\frac{M_{j}^{(r_1)}}{M_{i}^{(r_1)}}<-\frac{2k}{l}<\frac{M_{j}^{(h)}}{M_{i}^{(h)}},
\end{displaymath}
which is possible under the assumption $M_{i}^{(r_1)}M_{j}^{(h)}-M_{i}^{(h)}M_{j}^{(r_1)}<0$. Next, take $l>w^{2}/(4l)$
sufficiently large. Then due to the inequalities we derived above, we ensure that
\begin{displaymath}
\begin{array}{rl}
2kM_{i}^{(r_1)}+w(M_{i}^{(r_1)}+M_{j}^{(r_1)})+2lM_{j}^{(r_1)}<&0,\\
wM_{i}^{(r_2)}+2kM_{j}^{(r_2)}<&0,\\
2kM_{i}^{(r_2)}+w(M_{i}^{(r_2)}+M_{j}^{(r_2)})+2lM_{j}^{(r_2)}<&0,\\
wM_{i}^{(v)}+2kM_{j}^{(v)}<&0.
\end{array}
\end{displaymath} 
Thus, the function $f(i,j)$ satisfies \eqref{foster}, and using \cite[Lemma 2.3.3.3]{fay}, as $i^{2}+j^{2}\to\infty$,
\begin{displaymath}
E(f(i+\theta_{i},j+\theta_{j}))-f(i,j)=\frac{i(2kE(\theta_{i})+wE(\theta_{j}))+j(wE(\theta_{i})+2lE(\theta_{j}))}{2f(i,j)}+o(1).
\end{displaymath} 
For $i>j>0$, $(E(\theta_{i}),E(\theta_{j}))=(M_{i}^{(r_1)},M_{j}^{(r_1)})$ and 
\begin{displaymath}
E(f(i+\theta_{i},j+\theta_{j}))-f(i,j)\leq \frac{j(2kM_{i}^{(r_1)}+w(M_{i}^{(r_1)}+M_{j}^{(r_1)})+2lM_{j}^{(r_1)})}{j\sqrt{k+l+w}}+o(1)<-\epsilon_{1},
\end{displaymath} 
for some $\epsilon_{1}>0$, as $i^{2}+j^{2}\to\infty$. Similar argumentation proves the validity of assertion 2 of Theorem \ref{stab} for the rest of the cases. The proof of assertion 3 is symmetric to the proof of assertion 2 and further details are omitted. 

We now turn our attention in the necessary part, and show that the chain $\{\tilde{X}(n);n\geq 0\}$ is non-ergodic if none of the assertions 1, 2, 3 holds. To cope with this task, we apply \cite[Theorem 2.2.6]{fay}, which states that for Markov chain $L$ to be non-ergodic, it is sufficient that there exist a function $f (i,j)$ on $\mathbb{Z}^{2}_{+}$ and a constant $C > 0$ such that
\begin{eqnarray}
E(f(i+\theta_{i},j+\theta_{j}))-f(i,j)\geq 0,\label{none}
\end{eqnarray}
for all $(i,j)\in \{(i,j)\mathbb{Z}^{2}_{+}:f(i,j)>C\}$, where the sets $\{(i,j)\in \mathbb{Z}^{2}_{+}:f(i,j)>C\}$ and $\{(i,j)\mathbb{Z}^{2}_{+}:f(i,j)<C\}$ are not empty.

Assume first that assertion 1 does not hold, i.e., $M_{i}^{(r_1)}+M_{j}^{(r_1)}=M_{i}^{(r_2)}+M_{j}^{(r_2)}\geq 0$, and set $f(i,j)=i+j$. If $i,j>0$,
\begin{displaymath}
\begin{array}{c}
E(f(i+\theta_{i},j+\theta_{j}))-f(i,j)=M_{i}^{(r_1)}+M_{j}^{(r_1)}=M_{i}^{(r_2)}+M_{j}^{(r_2)}\geq 0.
\end{array}
\end{displaymath} 
If $i>0,j=0$,
\begin{displaymath}
\begin{array}{c}
E(f(i+\theta_{i},j+\theta_{j}))-f(i,j)=M_{i}^{(h)}+M_{j}^{(h)}=M_{i}^{(r_1)}+M_{j}^{(r_1)}+\frac{\lambda\tilde{\mu}_{2}}{\lambda+\alpha_{1}}>0.
\end{array}
\end{displaymath} 
If $i=0,j>0$,
\begin{displaymath}
\begin{array}{c}
E(f(i+\theta_{i},j+\theta_{j}))-f(i,j)=M_{i}^{(v)}+M_{j}^{(v)}=M_{i}^{(r_1)}+M_{j}^{(r_1)}+\frac{\lambda\tilde{\mu}_{1}}{\lambda+\alpha_{2}}>0.
\end{array}
\end{displaymath} 
Thus, $f(i,j)$ satisfies \eqref{none} and the chain is non-ergodic when assertion 1 does not hold.

Assume now that 
\begin{eqnarray}
M_{i}^{(r_1)}\geq 0,\\
f_{1}\geq 0\Leftrightarrow M_{i}^{(r_1)}M_{j}^{(h)}-M_{i}^{(h)}M_{j}^{(r_1)}\geq 0.\label{vcz}
\end{eqnarray}
We further focus only to the case where $M_{i}^{(r_1)}+M_{j}^{(r_1)}<0$ (since the case $M_{i}^{(r_1)}+M_{j}^{(r_1)}\geq 0$ has been
already considered). Moreover, $M_{j}^{(r_1)}<0$ and assume $M_{i}^{(r_1)}>0$ (we omit the case $M_{i}^{(r_1)}=0$, since it implies that $M_{i}^{(h)}=-\frac{\alpha_{1}\tilde{\mu}_{2}}{\lambda+\alpha_{1}}<0$. With that in mind and by taking into account \eqref{vcz}, we have that $M_{j}^{(r_1)}\geq0$, so that $M_{i}^{(r_1)}+M_{j}^{(r_1)}\geq 0$ and the chain is non-ergodic).

The assumptions $M_{i}^{(r_1)}>0$, $M_{j}^{(r_2)}<0$, implies
\begin{eqnarray}
M_{i}^{(r_2)}>M_{i}^{(r_1)}>0,\\
M_{j}^{(r_2)}<M_{j}^{(r_1)}<0.
\end{eqnarray}
Let $f(i,j)=-M_{j}^{(r_1)}i+M_{i}^{(r_1)}j$. Then, for $i>j>0$,
\begin{eqnarray}
E(f(i+\theta_{i},j+\theta_{j}))-f(i,j)=f(M_{i}^{(r_1)},M_{j}^{(r_1)})=0.
\end{eqnarray}
If $i>0,j=0$,
\begin{eqnarray}
E(f(i+\theta_{i},j+\theta_{j}))-f(i,j)=f(M_{i}^{(h)},M_{j}^{(h)})\geq 0,
\end{eqnarray}
due to \eqref{vcz}. If $j>i>0$, we know that since $M_{i}^{(r_1)}+M_{j}^{(r_2)}<0$, and $M_{i}^{(r_2)}>M_{i}^{(r_1)}$ we have
\begin{displaymath}
\begin{array}{rl}
E(f(i+\theta_{i},j+\theta_{j}))-f(i,j)=&f(M_{i}^{(r_2)},M_{j}^{(r_2)})=-M_{i}^{(r_2)}M_{j}^{(r_1)}+M_{i}^{(r_1)}M_{j}^{(r_2)}\\=&-M_{i}^{(r_2)}M_{j}^{(r_1)}+M_{i}^{(r_1)}(M_{j}^{(r_2)}+M_{i}^{(r_2)}-M_{i}^{(r_2)})\\
=&-M_{i}^{(r_2)}M_{j}^{(r_1)}+M_{i}^{(r_1)}(M_{j}^{(r_1)}+M_{i}^{(r_1)}-M_{i}^{(r_2)})\\=&-(M_{j}^{(r_1)}+M_{i}^{(r_1)})(M_{i}^{(r_2)}-M_{i}^{(r_1)})>0.
\end{array}
\end{displaymath}
Finally, for $i=0,j>0$,
\begin{displaymath}
\begin{array}{l}
E(f(i+\theta_{i},j+\theta_{j}))-f(i,j)=f(M_{i}^{(v)},M_{j}^{(v)})=-M_{i}^{(v)}M_{j}^{(r_1)}+M_{i}^{(r_1)}M_{j}^{(v)}\\=-M_{j}^{(r_1)}(M_{i}^{(v)}-M_{i}^{(r_{2})})+M_{i}^{(r_1)}(M_{j}^{(v)}-M_{j}^{(r_2)})+f(M_{i}^{(r_2)},M_{j}^{(r_2)})\\
=f(M_{i}^{(r_2)},M_{j}^{(r_2)})-\tilde{\mu}_{1}(M_{j}^{(r_1)}+M_{i}^{(r_1)}\frac{\alpha_{2}}{\lambda+\alpha_{2}})>f(M_{i}^{(r_2)},M_{j}^{(r_2)})-\tilde{\mu}_{1}(M_{j}^{(r_1)}+M_{i}^{(r_1)})>0.
\end{array}
\end{displaymath}
Therefore, in case assertion 2 does not hold, the chain is non-ergodic. The investigation of the case $M_{j}^{(r_2)}\geq 0$, $f_{2}\geq 0\Leftrightarrow M_{j}^{(r_2)}M_{i}^{(v)}-M_{i}^{(v)}M_{j}^{(r_2)}\geq 0$, i.e., the violation of assertion 3, is symmetric to the previous one, and further details are omitted.
\section{On the decay rate of the \textit{reference system}}\label{app0}
We consider the \textit{reference system} described at the end of subsection \ref{taildecay} and show that the tail decay rate of the stationary orbit queue length distribution equals $\rho$. Let $N(t)$ be the number of orbiting jobs at time $t$ in this \textit{reference system}. Then, $L(t)=\{(N(t),C(t));t\geq 0\}$ is the continuous Markov chain with state space $\mathbb{N}_{0}\times\{0,1\}$. The process $\{L(t);t\geq 0\}$ is a homogeneous QBD process with transition rate 
\begin{displaymath}
\mathbb{Q}=\begin{pmatrix}
\Lambda_{0}^{(0)}&\Lambda_{1}&&&&\\
\Lambda_{-1}^{(0)}&\Lambda_{0}^{(1)}&\Lambda_{1}&&&\\
&\Lambda_{-1}&\Lambda_{0}&\Lambda_{1}&&\\
&&\Lambda_{-1}&\Lambda_{0}&\Lambda_{1}&\\
&&&\ddots&\ddots&\ddots\\
\end{pmatrix},
\end{displaymath}
where
\begin{displaymath}
\begin{array}{c}
\Lambda_{0}^{(0)}=\begin{pmatrix}
-\lambda&\lambda\\
\mu&-(\lambda+\mu)
\end{pmatrix},\,\Lambda_{1}=\begin{pmatrix}
0&0\\
0&\lambda
\end{pmatrix},\,\Lambda_{-1}^{(0)}=\begin{pmatrix}
0&\alpha_{1}\\
0&0
\end{pmatrix},\,\Lambda_{-1}=\begin{pmatrix}
0&\alpha_{1}+\alpha_{2}\\
0&0
\end{pmatrix}\\
\Lambda_{0}^{(1)}=\begin{pmatrix}
-(\lambda+\alpha_{1})&\lambda\\
\mu&-(\lambda+\mu)
\end{pmatrix},\,\Lambda_{0}=\begin{pmatrix}
-(\lambda+\alpha_{1}+\alpha_{2})&\lambda\\
\mu&-(\lambda+\mu)
\end{pmatrix}.
\end{array}
\end{displaymath}
Under usual assumptions, the transition rate matrix $\mathbb{Q}$ has a single communicating class. Let $\Lambda=\Lambda_{-1}+\Lambda_{0}+\Lambda_{1}:=\begin{pmatrix}
-(\lambda+\alpha_{1}+\alpha_{2})&\lambda+\alpha_{1}+\alpha_{2}\\
\mu&-\mu
\end{pmatrix}$, and $\mathbf{u}:=(u_{0},u_{1})$ its stationary vector. Then, $\mathbf{u}\Lambda=\mathbf{0}$, $\mathbf{u}\mathbf{1}=1$, yields $u_{0}=\mu$, $u_{1}=\lambda+\alpha_{1}+\alpha_{2}$ (remind that we have assumed $\lambda+\alpha_{1}+\alpha_{2}+\mu=1$). Then, following \cite{neuts}, $\{L(t);t\geq 0\}$ is stable if and only if $\mathbf{u}\Lambda_{1}\mathbf{1}<\mathbf{u}\Lambda_{-1}\mathbf{1}$. Straightforward calculations indicate that the stability condition for the reference system is $\widehat{\lambda}<\widehat{\mu}_{1}+\widehat{\mu}_{2}\Rightarrow\rho<1$.

Note that $\Lambda_{1}=\mathbf{a}.\mathbf{z}$, where $\mathbf{a}=(0,\lambda)^{T}$, $\mathbf{z}=(0,1)$. Thus, by applying the matrix geometric method \cite{neuts} we are able to derive the equilibrium distribution of $\{L(t);t\geq 0\}$, for which the rate matrix is obtained explicitly. Moreover, the Perron-Frobenius eigenvalue of the rate matrix, i.e., the tail decay rate $\zeta$, is given as the unique root in $(0, 1)$ of the determinant equation
 \begin{displaymath}
 \begin{array}{c}
 det(\Lambda_{1}+\Lambda_{0}\zeta+\Lambda_{-1}\zeta^{2})=0.
 \end{array}
\end{displaymath}  
In our case $det(\Lambda_{1}+\Lambda_{0}\zeta+\Lambda_{-1}\zeta^{2})=0\Rightarrow \zeta(1-\zeta)(\zeta(\widehat{\mu}_{1}+\widehat{\mu}_{2})-\widehat{\lambda})=0$. Therefore, $\zeta=\rho$.
\section{Proof of Lemma \ref{kl1}}\label{app2}
Condition \eqref{s1} reads as follows
\begin{equation}
\begin{array}{rl}
x_{l}(z)=&x_{l-1}(z)(\frac{\widehat{\mu}_{1}}{\lambda+\alpha_{1}+\alpha_{2}}+(\lambda_{0}+\lambda_{2})z)+x_{l}(z)(\alpha_{1}+\alpha_{2}+\frac{\lambda\mu}{\lambda+\alpha_{1}+\alpha_{2}})\\&+x_{l+1}(z)(\frac{\widehat{\mu}_{2}}{\lambda+\alpha_{1}+\alpha_{2}}z^{-1}+\lambda_{1}),\,l\leq -2,\label{d1}
\end{array}
\end{equation}
\begin{equation}
\begin{array}{rl}
x_{-1}(z)=&x_{-2}(z)(\frac{\widehat{\mu}_{1}}{\lambda+\alpha_{1}+\alpha_{2}}+(\lambda_{0}+\lambda_{2})z)+x_{-1}(z)(\alpha_{1}+\alpha_{2}+\frac{\lambda\mu}{\lambda+\alpha_{1}+\alpha_{2}})\\&+x_{0}(z)(\frac{\widehat{\mu}_{2}}{\lambda+\alpha_{1}+\alpha_{2}}z^{-1}+\lambda_{1}+\frac{\lambda_{0}}{2}),\label{d2}
\end{array}
\end{equation}
\begin{equation}
\begin{array}{rl}
x_{0}=&x_{-1}(z)(\frac{\widehat{\mu}_{1}}{\lambda+\alpha_{1}+\alpha_{2}}+(\lambda_{0}+\lambda_{2})z)+x_{0}(\alpha_{1}+\alpha_{2}+\frac{\lambda\mu}{\lambda+\alpha_{1}+\alpha_{2}})\\&+x_{1}(z)(\frac{\widehat{\mu}_{2}}{\lambda+\alpha_{1}+\alpha_{2}}+(\lambda_{0}+\lambda_{1})z),\label{d3}
\end{array}
\end{equation}
\begin{equation}
\begin{array}{rl}
x_{1}(z)=&x_{0}(z)(\frac{\widehat{\mu}_{1}}{\lambda+\alpha_{1}+\alpha_{2}}z^{-1}+\frac{\lambda_{0}}{2}+\lambda_{2})+x_{1}(z)(\alpha_{1}+\alpha_{2}+\frac{\lambda\mu}{\lambda+\alpha_{1}+\alpha_{2}})\\&+x_{2}(z)(\frac{\widehat{\mu}_{2}}{\lambda+\alpha_{1}+\alpha_{2}}+(\lambda_{1}+\lambda_{0})z),\label{d22}
\end{array}
\end{equation}
\begin{equation}
\begin{array}{rl}
x_{l}(z)=&x_{l-1}(z)(\frac{\widehat{\mu}_{1}}{\lambda+\alpha_{1}+\alpha_{2}}z^{-1}+\lambda_{2})+x_{l}(z)(\alpha_{1}+\alpha_{2}+\frac{\lambda\mu}{\lambda+\alpha_{1}+\alpha_{2}})\\&+x_{l+1}(z)(\frac{\widehat{\mu}_{2}}{\lambda+\alpha_{1}+\alpha_{2}}+(\lambda_{0}+\lambda_{1})z),\,l\geq 2,\label{d4}
\end{array}
\end{equation}
and
\begin{equation}
\begin{array}{rl}
y_{l}(z)=&y_{l-1}(z)(\frac{\widehat{\mu}_{2}}{\lambda+\alpha_{1}+\alpha_{2}}z^{-1}+\lambda_{1})+y_{l}(z)(\alpha_{1}+\alpha_{2}+\frac{\lambda\mu}{\lambda+\alpha_{1}+\alpha_{2}})\\&+y_{l+1}(z)(\frac{\widehat{\mu}_{1}}{\lambda+\alpha_{1}+\alpha_{2}}z^{-1}+(\lambda_{0}+\lambda_{2})z),\,l\leq -1,\label{d5}
\end{array}
\end{equation}
\begin{equation}
\begin{array}{rl}
y_{0}(z)=&y_{-1}(z)(\frac{\widehat{\mu}_{2}}{\lambda+\alpha_{1}+\alpha_{2}}z^{-1}+\frac{\lambda_{0}}{2}+\lambda_{1})+y_{0}(\alpha_{1}+\alpha_{2}+\frac{\lambda\mu}{\lambda+\alpha_{1}+\alpha_{2}})\\&+y_{1}(z)(\frac{\widehat{\mu}_{1}}{\lambda+\alpha_{1}+\alpha_{2}}z^{-1}+\frac{\lambda_{0}}{2}+\lambda_{2}),\label{d33}
\end{array}
\end{equation}
\begin{equation}
\begin{array}{rl}
y_{l}(z)=&y_{l-1}(z)(\frac{\widehat{\mu}_{2}}{\lambda+\alpha_{1}+\alpha_{2}}+(\lambda_{0}+\lambda_{1})z)+y_{l}(z)(\alpha_{1}+\alpha_{2}+\frac{\lambda\mu}{\lambda+\alpha_{1}+\alpha_{2}})\\&+y_{l+1}(z)(\frac{\widehat{\mu}_{1}}{\lambda+\alpha_{1}+\alpha_{2}}z^{-1}+\lambda_{2}),\,l\geq 1.\label{d6}
\end{array}
\end{equation}

Then, for $\eta_{min}<\eta_{max}$, $\theta_{min}<\theta_{max}$, a general solution for $\mathbf{x}$ is given by,
\begin{eqnarray}
x_{l}(z)=\left\{\begin{array}{ll}
k_{1}\eta_{max}^{-l}+k_{2}\eta_{min}^{-l},&l\leq-2,\\
k_{3}\theta_{min}^{l}+k_{4}\theta_{max}^{l},&l\geq 2,
\end{array}\right.
\end{eqnarray}
where $k_{1},k_{2},k_{3},k_{4}$, are constants to be determined by \eqref{d2}-\eqref{d22}. Similarly,
\begin{eqnarray}
y_{l}(z)=\left\{\begin{array}{ll}
s_{1}\eta_{min}^{l}+s_{2}\eta_{max}^{l},&l\leq-1,\\
s_{3}\theta_{max}^{-l}+s_{4}\theta_{min}^{-l},&l\geq 1,
\end{array}\right.\label{yl}
\end{eqnarray}
where $s_{1},s_{2},s_{3},s_{4}$ are constants, and $\eta_{min}$, $\eta_{max}$, $\theta_{min}$, $\theta_{max}$ are obtained in the way of solving the difference equations \eqref{d1}-\eqref{d6}, which result in \eqref{pol1}, \eqref{pol2}. 

Following \cite{limiya}, the condition \eqref{s2} requires \eqref{ss1}, $\mathbf{x}(z)$ to be constructed using $\eta_{min}$, $\theta_{min}$, and $\mathbf{y}(z)$ to be constructed using $\eta_{max}^{-1}$, $\theta_{max}^{-1}$, so that $k_{1}=k_{4}=s_{1}=s_{4}=0$. Next we focus on \eqref{d5}, \eqref{d6} and substitute \eqref{yl} to realize that $s_{2}=s_{3}=y_{0}$, so that 
\begin{displaymath}
y_{-1}(z)=y_{0}\eta_{max}^{-1},\,y_{1}(z)=y_{0}\theta_{max}^{-1}.
\end{displaymath}
Then by substituting in \eqref{d33}, the resulting equation is $f(z)=1$. Using a similar argumentation and \eqref{d1}, \eqref{d2} yield $x_{-1}(z)=c_{2}\eta_{min}$, with 
\begin{displaymath}
c_{2}=\frac{\widehat{\mu}_{2}z^{-1}+\widehat{\lambda}_{1}+\frac{\widehat{\lambda}_{0}}{2}}{\widehat{\mu}_{2}z^{-1}+\widehat{\lambda}_{1}}x_{0}.
\end{displaymath}
Similarly, by using \eqref{d22}, \eqref{d4}, $x_{1}(z)=c_{3}\theta_{min}$, where now,
\begin{displaymath}
c_{3}=\frac{\widehat{\mu}_{1}z^{-1}+\widehat{\lambda}_{2}+\frac{\widehat{\lambda}_{0}}{2}}{\widehat{\mu}_{1}z^{-1}+\widehat{\lambda}_{2}}x_{0}.
\end{displaymath}
Then, substituting back in \eqref{d3} yields again $f(z)=1$, and the conditions are necessary and sufficient. Equation \eqref{inva} is obtained straightforwardly from the above results. 
\section{Proof of Lemma \ref{kl2}}\label{appro}
We show that $z=\rho^{-2}$ is the unique solution of $f(z)=1$, $z>1$. For convenience let
\begin{displaymath}
\begin{array}{lr}
\tilde{\mu}_{i}=\frac{\widehat{\mu}_{i}}{\widehat{\lambda}+\widehat{\mu}_{1}+\widehat{\mu}_{2}},\,i=1,2,&\tilde{\lambda}_{i}=\frac{\widehat{\lambda}_{i}}{\widehat{\lambda}+\widehat{\mu}_{1}+\widehat{\mu}_{2}},\,i=0,1,2,\\
\xi_{1}=2(\tilde{\mu}_{1}z^{-1}+\tilde{\lambda}_{2})+\tilde{\lambda}_{0},&\xi_{2}=2(\tilde{\mu}_{2}z^{-1}+\tilde{\lambda}_{1})+\tilde{\lambda}_{0}.
\end{array}
\end{displaymath}
Under such a setting, the following identities are derived straightforwardly: 
\begin{displaymath}
\begin{array}{lr}
\tilde{\mu}_{1}z^{-1}+\tilde{\lambda}_{2}=\frac{1}{2}(\xi_{1}-\tilde{\lambda}_{0}),\,&\tilde{\mu}_{2}z^{-1}+\tilde{\lambda}_{1}=\frac{1}{2}(\xi_{2}-\tilde{\lambda}_{0}),\\
\tilde{\mu}_{1}+(\tilde{\lambda}_{2}+\tilde{\lambda}_{0})z=\frac{1}{2}(\xi_{1}+\tilde{\lambda}_{0})z,\,&\tilde{\mu}_{2}+(\tilde{\lambda}_{1}+\tilde{\lambda}_{0})z=\frac{1}{2}(\xi_{2}+\tilde{\lambda}_{0})z.
\end{array}
\end{displaymath}
Then, $f(z)=1$ is rewritten as
\begin{displaymath}
\begin{array}{rl}
\tilde{\lambda}_{0}(\xi_{1}+\xi_{2}-2\tilde{\lambda}_{0})=&\xi_{2}(\xi_{1}-\tilde{\lambda}_{0})\sqrt{1-(\xi_{1}+\tilde{\lambda}_{0})(\xi_{2}-\tilde{\lambda}_{0})z}+\xi_{1}(\xi_{2}-\tilde{\lambda}_{0})\sqrt{1-(\xi_{2}+\tilde{\lambda}_{0})(\xi_{1}-\tilde{\lambda}_{0})z}.
\end{array}
\end{displaymath}
Taking squares at both sides and rearrange terms yields
\begin{equation}
\begin{array}{l}
\tilde{\lambda}_{0}^{2}(\xi_{1}+\xi_{2}-2\tilde{\lambda}_{0})^{2}-\xi_{1}^{2}(\xi_{2}-\tilde{\lambda}_{0})^{2}-\xi_{2}^{2}(\xi_{1}-\tilde{\lambda}_{0})^{2}\\=-\xi_{1}^{2}(\xi_{2}-\tilde{\lambda}_{0})^{2}(\xi_{1}-\tilde{\lambda}_{0})(\xi_{2}+\tilde{\lambda}_{0})z-\xi_{2}^{2}(\xi_{1}-\tilde{\lambda}_{0})^{2}(\xi_{2}-\tilde{\lambda}_{0})(\xi_{1}+\tilde{\lambda}_{0})z\\
+2\xi_{1}\xi_{2}(\xi_{1}-\tilde{\lambda}_{0})(\xi_{2}-\tilde{\lambda}_{0})\sqrt{1-(\xi_{1}-\tilde{\lambda}_{0})(\xi_{2}+\tilde{\lambda}_{0})z}\sqrt{1-(\xi_{2}-\tilde{\lambda}_{0})(\xi_{1}+\tilde{\lambda}_{0})z}.
\end{array}\label{bq}
\end{equation}
Having in mind that $\xi_{1}-\tilde{\lambda}_{0}>0$, $\xi_{2}-\tilde{\lambda}_{0}>0$, \eqref{bq} is rewritten after manipulations as
\begin{displaymath}
\begin{array}{l}
2(2\tilde{\lambda}_{0}^{2}-\xi_{1}\xi_{2})+(\xi_{1}^{2}(\xi_{2}^{2}-\tilde{\lambda}_{0}^{2})+\xi_{2}^{2}(\xi_{1}^{2}-\tilde{\lambda}_{0}^{2}))z=2\xi_{1}\xi_{2}\sqrt{1+2(\tilde{\lambda}_{0}^{2}-\xi_{1}\xi_{2})z+(\tilde{\lambda}_{0}^{2}-\xi_{1}^{2})(\tilde{\lambda}_{0}^{2}-\xi_{2}^{2})z^{2}}.
\end{array}
\end{displaymath}
Taking again squares at both sides, we obtain after lengthy but straightforward manipulations the following equation
\begin{displaymath}
\begin{array}{c}
((\xi_{1}+\xi_{2})^{2}z-4)((\xi_{1}-\xi_{2})^{2}\tilde{\lambda}_{0}^{2}z+4(\xi_{1}\xi_{2}-\tilde{\lambda}_{2}))=0.
\end{array}
\end{displaymath}
Note that $\xi_{1}\xi_{2}-\tilde{\lambda}_{2}=2[2(\tilde{\mu}_{1}z^{-1}+\tilde{\lambda}_{2})(\tilde{\mu}_{2}z^{-1}+\tilde{\lambda}_{1})+\tilde{\lambda}_{0}((\tilde{\mu}_{1}+\tilde{\mu}_{1})z^{-1}+\tilde{\lambda}_{2}+\tilde{\lambda}_{1})]>0$, and thus, $f(z)=1$ if and only if
\begin{displaymath}
(\xi_{1}+\xi_{2})^{2}z=4.
\end{displaymath}
It is easily seen that this quadratic equation with respect to $z$ has two solutions, i.e., $z=1$ and $z=\rho^{-2}$. Thus, the only possible solution of $f(z)=1$ such that $z>1$ is $\rho^{-2}$. For $z=\rho^{-2}$ we can easily derive the zeros of \eqref{pol1}, \eqref{pol2}, as given in \eqref{rot1}, \eqref{rot2}, respectively.\vspace{1cm}\\


\end{document}